\documentclass[twosided, 11pt]{article}
\usepackage{geometry}
\usepackage{import}
\usepackage{glossaries}
\usepackage{float}
\usepackage[title,titletoc]{appendix}
\usepackage{titlesec}
\usepackage{amsmath, amsfonts, bbm, mathrsfs, graphicx, amssymb, amsthm, hyperref, centernot, enumerate, bbm, xcolor, lmodern, mathdots, wasysym, mathtools,tikz}
\usepackage{hyperref}
\usepackage{tcolorbox}
\usepackage{cleveref}
\usepackage{import}
\usepackage[retainorgcmds]{IEEEtrantools}
\usepackage{verbatim}
\usepackage[inline]{enumitem}
\usepackage{parskip}
\usepackage{xcolor}
\usetikzlibrary{automata, positioning}
\usetikzlibrary{calc}
\usepackage [english]{babel}
\usepackage [autostyle, english = american]{csquotes}
\usepackage[T1]{fontenc}  
\usepackage{enumerate}
\setitemize[1]{label=\textbullet}
\setenumerate[1]{label=(\roman*), itemsep=0pt, topsep=-10pt, leftmargin=0.65cm}
\usepackage[nottoc]{tocbibind}
\usepackage{bm}
\usepackage{systeme}

\usepackage{makeidx}

\makeindex

\usepackage[intoc]{nomencl}

\makeatletter
\AtBeginEnvironment{thenomenclature}{%
  \patchcmd{\@makeschapterhead}{\vspace*{50\p@}}{}{}{}%
}
\makeatother
\makenomenclature

\usepackage{amssymb,latexsym,graphicx}
\usepackage{amsmath,amscd,amsthm}
\usepackage{mathrsfs}
\usepackage{graphicx}
\usepackage{esint}
\usepackage{pifont}
\usepackage{subcaption}
\usepackage{breqn}
\usepackage{tikz}
\usepackage{tikz-cd} 
\usetikzlibrary{arrows,decorations.markings, shapes.geometric}
\usepackage{environ}
\usepackage{pgfplots}
\usepackage{math tools}
\numberwithin{equation}{section}
\theoremstyle{plain}
\usepackage{relsize}
\usepackage{hyperref}
\hypersetup{
    colorlinks=true,
    linkcolor=black,
    filecolor=black,      
    urlcolor=black,
    citecolor=black,
}
\pgfplotsset{compat=1.17}
\usepackage{pgfplots}
\usepgfplotslibrary{fillbetween}
\usetikzlibrary{patterns}

\newtheorem{thm}{Theorem}[section]
\newtheorem*{thm*}{Theorem}

\newtheorem{prop}[thm]{Proposition}

\newtheorem{cor}[thm]{Corollary}
\newtheorem{remark}[thm]{Remark}
\newtheorem{defn}[thm]{Definition}
\newtheorem{lemma}[thm]{Lemma}
\newtheorem{question}[thm]{Question}

\newcommand{\diff}{\mathrm{d}}

\newcommand{\eps}{\epsilon}

\newcommand{\Rthree}{\mathbb{R}^{3}}

\DeclareMathOperator{\sgn}{sgn}

\newcommand{\klo}{k_{\tiny{\text{lo}}}}
\newcommand{\khi}{k_{\tiny{\text{hi}}}}

\definecolor{indigo(web)}{rgb}{0.29, 0.0, 0.51}
\definecolor{islamicgreen}{rgb}{0.0, 0.56, 0.0}
\definecolor{ballblue}{rgb}{0.13, 0.67, 0.8}
\definecolor{alizarin}{rgb}{0.82, 0.1, 0.26}
\definecolor{goldenrod}{rgb}{0.85, 0.65, 0.13}
\definecolor{forest}{rgb}{0.13, 0.55, 0.13}
\definecolor{brandeisblue}{rgb}{0.0, 0.44, 1.0}	
\definecolor{cadmiumorange}{rgb}{0.93, 0.53, 0.18}
\definecolor{darkelectricblue}{rgb}{0.33, 0.41, 0.47}
\definecolor{brightmaroon}{rgb}{0.76, 0.13, 0.28}
\definecolor{darkcerulean}{rgb}{0.03, 0.27, 0.49}
\definecolor{darkviolet}{rgb}{0.58, 0.0, 0.83}
\definecolor{mauvelous}{rgb}{0.94, 0.6, 0.67}
\definecolor{hotpink}{rgb}{1.0, 0.41, 0.71}
\definecolor{darkspringgreen}{rgb}{0.09, 0.45, 0.27}
\definecolor{ceruleanblue}{rgb}{0.16, 0.32, 0.75}
\definecolor{mint}{rgb}{0.24, 0.71, 0.54}
\definecolor{moonstoneblue}{rgb}{0.45, 0.66, 0.76}
\definecolor{pistachio}{rgb}{0.58, 0.77, 0.45}
\definecolor{pinkpearl}{rgb}{0.91, 0.67, 0.81}
\definecolor{tiffanyblue}{rgb}{0.04, 0.73, 0.71}
\definecolor{lightseagreen}{rgb}{0.13, 0.7, 0.67}
\definecolor{darkgray}{rgb}{0.66, 0.66, 0.66}
\definecolor{denim}{rgb}{0.08, 0.38, 0.74}
\definecolor{blue(pigment)}{rgb}{0.2, 0.2, 0.6}
\definecolor{cobalt}{rgb}{0.0, 0.28, 0.67}
\definecolor{mikadoyellow}{rgb}{1.0, 0.77, 0.05}
\definecolor{caribbeangreen}{rgb}{0.0, 0.8, 0.6}
\definecolor{blush}{rgb}{0.87, 0.36, 0.51}
\definecolor{brinkpink}{rgb}{0.98, 0.38, 0.5}

\usepackage[intoc]{nomencl}

\makeatletter
\AtBeginEnvironment{thenomenclature}{%
  \patchcmd{\@makeschapterhead}{\vspace*{50\p@}}{}{}{}%
}
\makeatother
\makenomenclature

\title{Scattering for the Quartic Generalized Benjamin-Bona-Mahony Equation}
\author{A. George Morgan\footnote{Department of Mathematics,
University of Toronto,
40 St. George St., Room 6290,
Toronto, Ontario, CA,
M5S 2E4. Institutional email: 
adam.morgan@mail.utoronto.ca}
}

\date{}

\begin{document}

\maketitle

 \begin{abstract}
 \noindent 
The generalized Benjamin-Bona-Mahony equation (gBBM) is a model for nonlinear dispersive waves which, in the long-wave limit, is approximately equivalent to the generalized Korteweg-de Vries equation (gKdV). While the long-time behaviour of small solutions to gKdV is well-understood, the corresponding theory for gBBM has progressed little since the 1990s. Using a space-time resonance approach, I establish linear dispersive decay and scattering for small solutions to the quartic-nonlinear gBBM. To my knowledge, this result provides the first global-in-time pointwise estimates on small solutions to gBBM with a nonlinear power less than or equal to five. Owing to nonzero inflection points in the linearized gBBM dispersion relation, there exist isolated space-time resonances without null structure, but in the course of the proof I show these resonances do not obstruct scattering. 
  \end{abstract}








    \tableofcontents

\section{Introduction}
\noindent Consider the Cauchy problem for the \textbf{generalized Benjamin-Bona-Mahony equation (gBBM)}: for a fixed $p\in \mathbb{N}$ and $u_{0}(x)$ belonging to a class of suitably nice functions, we want to find $$u(t,x)\colon \left[0,\infty\right)\times \mathbb{R}\rightarrow\mathbb{R}$$ solving 
 \begin{equation}
 \label{eqn:gbbm_INTRO}
\left\{\begin{aligned}
u_{t} - u_{xxt} + \partial_{x}\left(u + u^{p+1}\right) &= 0 \phantom{u_{0}(x)}\quad \forall \ (t,x)\in \left(0,\infty\right)\times\mathbb{R}
\\
u\left(t=0,x\right) &= u_{0}(x)\phantom{0} \quad \forall \ x\in \mathbb{R}. 
\end{aligned}\right.
\end{equation}
The dispersive term $-u_{xxt}+u_{x}$ causes the solution to spread out over time and separate into its constituent Fourier modes, while the nonlinear term $\partial_{x}\left(u^{p+1}\right)$ causes the solution to steepen over time. Roughly speaking, the balance between these two competing effects ensures that smooth initial data gives rise to a smooth global-in-time solution. In particular, for rapidly decaying smooth solutions we have the energy conservation law
\begin{equation}
    \label{eqn:energy_conservation}
    \left\|u(t,x)\right\|_{H^1_{x}} = \left\|u_0(x)\right\|_{H^1_{x}} \quad \forall \ t\geq0. 
\end{equation}
\noindent For certain values of $p$, gBBM has been used as a model of gravity waves in shallow water \cite{BBM1972, Peregrin1966} and blood pulses in elastic arteries \cite{BonaCascaval2008, MDLP2019}. Essentially, any physical process that can be modelled by the much more well-studied \textbf{generalized Korteweg-de Vries equation} (gKdV)
 \begin{equation}
     \label{eqn:gkdv_INTRO}
     u_{t} + u_{xxx} +\partial_{x}\left(u+u^{p+1}\right) =0
 \end{equation}
 can also be modelled by gBBM, at least in the long-wave limit. Interestingly, the linearized gBBM (here called  \textbf{linBBM}) features finite speed of propagation while the linearized gKdV does not, implying that gBBM is in a sense more physical than gKdV. Here, we understand ``finite speed of propagation'' to mean ``bounded group velocity''; see figure \ref{fig:bbmCG}. Additionally, it's easy to see that gBBM lacks any scaling symmetry, which is quite rare among the PDEs of mathematical physics (the Klein-Gordon equation is another notable exception). In a nutshell, then, gBBM is worth investigating rigorously for its physical applications, its similarity to gKdV, and its unusual properties (linear finite speed of propagation and lack of scale-invariance). 
 \begin{figure}
\centering
\begin{tikzpicture}[scale=1]
\begin{axis}[xmin=-6, xmax=6, ymin=-0.4, ymax=1.05, samples=300, no markers,
      axis x line = middle,
      axis y line = middle,
    ylabel=\Large{$\omega'(\xi)$},
    xlabel=\Large{$\xi$},
    every axis x label/.style={
    at={(ticklabel* cs:1.1)},
},
every axis y label/.style={
    at={(ticklabel* cs:1.05)},
    anchor=south,
},
xticklabel style={yshift=0.5ex, anchor=south},
yticklabel style={xshift=-0.5ex, anchor=east},
yticklabels={,,,1},
xticklabels={,,,,,2,4,6},
]
  \addplot[indigo(web), ultra thick, domain=-7:7]  {(1-x*x)/((1+x*x)*(1+x*x))};
  \addlegendentry{BBM}
   \addplot[cadmiumorange, dotted, ultra thick, domain=-4:4]  {1-3*x*x};
     \addlegendentry{KdV}
\end{axis}
\end{tikzpicture}
\caption{Plot of the group velocity $\omega'(\xi)$ for the linearized BBM and the linearized KdV.} 
\label{fig:bbmCG}
\end{figure}
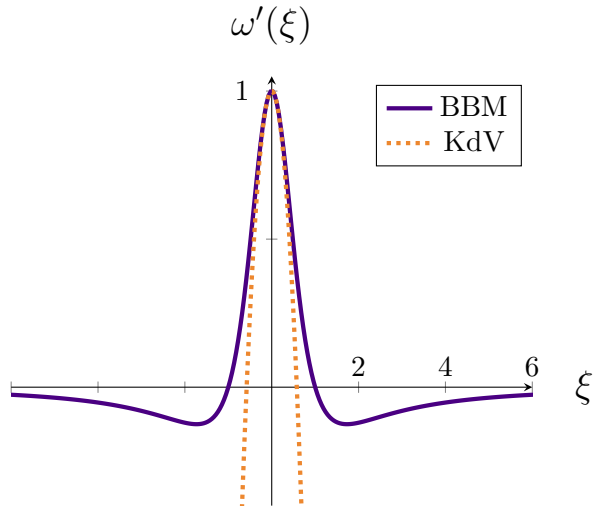
 \par Now, the most basic question one can ask about an initial-value problem is whether or not it is even well-posed; that is, whether or not a unique solution exists and depends continuously on the initial data. For \eqref{eqn:gbbm_INTRO}, the answer to this question is very well-understood: 
\begin{thm} 
If $u_{0}\in H^{1}_{x}\left(\mathbb{R}\right)\cap C^{2}_{x}\left(\mathbb{R}\right),$ then there exists a unique $u(t,x)$ satisfying \eqref{eqn:gbbm_INTRO}. Further, the initial-data-to-solution map
$$ u_{0} \in H^{1}_{x}\left(\mathbb{R}\right)\cap C^{2}_{x}\left(\mathbb{R}\right) \mapsto u(t,x) \in C\left(\left[0,\infty\right); H^{1}_{x}\left(\mathbb{R}\right)\cap C^{2}_{x}\left(\mathbb{R}\right)\right)$$
is continuous. 
\end{thm}
\noindent See \cite{BBM1972} for a  proof of this theorem. Also, there are analogues of the above result valid for rougher initial data \cite{BT2009}. 
\subsection{The Scattering Problem. Existing Results.}
\noindent With a global solution at our disposal, we may ask more complex questions about the dynamics of gBBM, including
\begin{question} [Long-Time Asymptotics of Small Solutions to gBBM] \label{question:main_INTRO}
Suppose the initial state $u_{0}(x)$ is small with respect to a suitable norm. Does the amplitude of $u(t,x)$ have the same dispersive decay as a solution of linBBM? More strongly, does $u(t,x)$ \textbf{scatter} (become indistinguishable from a solution of linBBM as $t\rightarrow \infty$)? 
\end{question}
\noindent Question \ref{question:main_INTRO} was answered for $p>4$ in the 1980s and 1990s by Albert \cite{Albert1986, Albert1989} and Dziuba\'{n}ski \& Karch \cite{DK1996} (see also Souganidis and Strauss \cite{SS1990}): for such $p$, solutions to gBBM indeed scatter. This makes intuitive sense, as when $p$ is large we expect the nonlinear term $\partial_{x}\left(u^{p+1}\right)$ to be extremely small and therefore its contribution to the dynamics ought to be insignificant. 
\par Now, to my knowledge there has been little progress since the mid-1990s in obtaining rigorous asymptotics at or below the ``quintic'' case $p=4$. Several authors have obtained partial results valid for a dissipative variant of gBBM,
\begin{equation}
\label{eqn:disp_gbbm}
u_t-u_{xxt} +u_{x} - \nu u_{xx} +\partial_{x}\left(u^{p+1}\right) =0, \quad \nu>0 \ \text{fixed}. 
\end{equation}
For example, Bona and Luo \cite{BonaLuo1993, BonaLuo1995} were able to prove pointwise decay of small solutions to \eqref{eqn:disp_gbbm} given any $p\geq 2$, but the time decay obtained was that associated to solutions of the heat equation and not that of solutions to linBBM. Other works on dissipative perturbations of gBBM (see for example Amick et al. \cite{ABS1989}, Karch \cite{Karch1997, Karch1999a, Karch2000}, Mei \cite{Mei1998}, and Zhang \cite{Zhang1994}) present similar conclusions, so we can confidently say that dissipation outcompetes dispersion over long time scales. Though interesting, this discovery says little about how to resolve question \ref{question:main_INTRO}. In 2019, Kwak and Mu\~{n}oz \cite{KM2019} established a nontraditional notion of scattering for $p\geq 1$: they showed that, in the long-time limit, initially small solutions do not leak energy outside of the region allowed by linBBM's finite speed of propagation property. The primary downside to the Kwak \& Mu\~{n}oz approach is that it does not provide pointwise bounds on $u(t,x)$, which we need in order to compare $u(t,x)$ to solutions of linBBM and therefore resolve question \ref{question:main_INTRO}. Finally, in 2021 Banquet \& Villamizar-Roa \cite{BVR2021} established some novel time decay estimates for the linBBM flow in modulation spaces, although the only application to the (nonlinear) BBM they presented concerns existence of solutions with rough Cauchy data and there is no discussion of long-time asymptotics. 
\par Despite the lack of a fully-developed scattering theory for gBBM and the similarities between gBBM and gKdV, the long-time asymptotics for gKdV are quite well-understood. Scattering for gKdV with $p\geq 4$ has been known since Strauss' 1974 paper on the subject \cite{Strauss1974}. By 1998, Hayashi \& Naumkin  \cite{HN1998_gkdv} were able to prove scattering all the way down to $p>2$; important intermediary contributions on shrinking the scattering threshold were made by Ponce \& Vega \cite{PV1990}, Christ \& Weinstein \cite{CW1991}, and others. For gKdV with $p=2$, one can establish a mildly modified notion of scattering for small initial data. From the late 1990s to the present, there has been plenty of work in this regard, see for instance Hayashi \& Naumkin \cite{HN1999, HN2001}, Harrop-Griffiths \cite{H-G2016}, and Germain et al. \cite{GPR2016}. 
\par There are three main reasons why progress on scattering for gBBM with $p\leq 4$ has nearly stagnated for almost three decades. 
\begin{enumerate}
\item For $p>4$, establishing scattering is easy using a basic combination of Duhamel's formula, bootstrapped linear dispersive decay, and rudimentary Sobolev space theory: this is the approach taken in Albert \cite{Albert1986, Albert1989} and Dziuba\'{n}ski \& Karch \cite{DK1996}. For $p\leq 4$ this approach is insufficient. 
    \item Second, the lack of any scaling symmetry implies that the method of vector fields applied by Hayashi and Naumkin to establish scattering for gKdV with $p>2$ \cite{HN1998_gkdv} cannot be used. 
    \item Finally, the gBBM dispersion relation 
    $$
    \omega(\xi) = \frac{\xi}{\langle\xi\rangle^2}
    $$
    has nonzero inflection points at $\xi=\pm \sqrt{3}$. This causes a slowdown in $u(t,x)$'s dispersive decay that cannot be counteracted by the $\partial_{x}$ in the nonlinear term $\partial_{x}\left(u^{p+1}\right)$, so propagating smallness of $u$ may be very difficult. Indeed, as far as I am aware the only existing result on detailed long-time asymptotics for dispersive waves with such inflection points (and a small nonlinear power) is Deng et al.'s work on the $3$D gravity-capillary water waves system \cite{DIPP2017}. 
\end{enumerate}
\begin{remark}
     Global theory for the $2$D gravity-capillary water waves problem (which, like the $3$D problem, features a dispersion relation with nonzero inflection points) is an open question. This provides some additional motivation for studying gBBM; that is, we can think of understanding gBBM as the gateway to understanding more complicated problems featuring dispersion relations with nonzero inflection points. 
\end{remark}
\subsection{Main Result}
\noindent The purpose of this article is to establish scattering of small solutions to gBBM in the ``quartic'' case $p=3$: 
\begin{thm}[Main Theorem]
\label{thm:main_thm}
Pick $s\geq 100$. Consider the Cauchy problem
\begin{equation}\label{eqn:3bbm_cauchy_problem_asympt}
\left\{
 \begin{aligned}
     u_t - u_{xxt} +\partial_{x}\left(u+u^{4}\right) &= 0
     \\
     u(t=1,x)&= u_0(x).
 \end{aligned}
\right. 
\end{equation}
There exists $\eps_0\in\left(0,1\right)$ such that
$$
\left\|\widehat{u_0}\right\|_{L^{\infty}_{\xi}} + \left\|xu_0(x)\right\|_{L^2_x} + \left\|u_0(x)\right\|_{H^{s}_x} < \eps_0
$$
implies the following: 
\begin{itemize}
    \item there exists a unique global-in-time solution $u(t,x)$ to \eqref{eqn:3bbm_cauchy_problem_asympt};
    \item if we denote $\omega(\xi)\doteq \xi\langle\xi\rangle^{-2}$ and $\widehat{f}\doteq e^{i\omega(\xi)t}\widehat{u}$, then there exist $p_0, p_1\in \left(0,\frac16\right)$ such that $f(t,x)$ satisfies the mild growth estimate
    $$
\sup_{t\in [1,\infty)}\left[\left\|\widehat{f}\right\|_{L^{\infty}_{\xi}}+ t^{-p_0}\left\|xf\right\|_{L^{2}_{x}} + t^{-p_1}\left\|f\right\|_{H^{s}_{x}}\right]\lesssim \eps_0;
    $$
    \item $u(t,x)$ obeys the dispersive decay bound 
$$
\left\|u(t,x)\right\|_{L^{\infty}_{x}} \lesssim \eps_0 t^{-\frac13} \quad \forall \ t\geq 1;
$$
\item $u(t,x)$ scatters in $L^2_{x}$: there exists $F(x)\in L^2_{x}$ such that
$$
\lim_{t\rightarrow\infty}\left\|u(t,x)- e^{-i\omega\left(\frac{1}{i}\partial_{x} \right) t}F(x)\right\|_{L^2_x} =0.
$$
\end{itemize}

\end{thm}
\noindent Before outlining the proof of this result, I'd like to make some comments. First, I do not claim that $100$ is the \emph{smallest} $s$ for which the theorem holds. That being said, the particular method of proof I employ makes dealing with high-frequency components of $u(t,x)$ much easier if $s$ is ``very large'', and tracking the minimum admissible $s$-value through each relevant case would be unwieldy. Identifying a sharp regularity threshold for scattering would be an interesting problem, but I do not think the tools I use here would be helpful for this purpose. Additionally, one can control $\left\|\widehat{u_0}\right\|_{L^{\infty}_{\xi}}$ with a weighted $L^2_{x}$-norm, so it appears that the smallness constraint on $u_0$ is slightly more complicated than it needs to be. The reason I keep the $\left\|\widehat{u_0}\right\|_{L^{\infty}_{\xi}}$ piece separate from the weighted-$L^2_{x}$ piece is pedagogical: by keeping these two norms distinct in the smallness condition, I am mentally preparing the reader to appreciate their very different behaviour in the claimed mild growth estimate in the second item of the conclusion. Finally, note that the main theorem only covers $p=3$, which is strictly below the open quintic case $p=4$. While it appears more natural to first handle the quintic case, choosing $p=3$ makes the computations in appendix \ref{section:res_comp} a bit simpler. However, I am confident that my $p=3$ arguments will work for $p=4$ as well.

\subsection{Strategy of the Proof: the Method of Space-Time Resonances}
\noindent  \noindent The core idea underlying the proof of theorem \ref{thm:main_thm}
 is the \textbf{method of space-time resonances} (STR). In applying STR, one identifies what pairs (or triples, \textit{et cetera} depending on the nonlinearity) of wavepacket-like solutions interact ``resonantly''. These resonant interactions \emph{may} cause the nonlinear flow to substantially deviate from the linear flow. To establish scattering, then, one must identify all resonant interactions and prove that they have little effect on the dynamics over long time scales. Resonant interactions may be enumerated using a computation relying on simple stationary phase ideas, outlined in the next subsection. 
 \par STR is a relatively young technique, introduced by Germain, Masmoudi, and Shatah in 2009 \cite{GMS2009} to understand asymptotics for a quadratic-nonlinear Schr\"{o}dinger equation in three space dimensions (see also Gustafson, Nakanishi, and Tsai \cite{GNT2007, GNT2009}). Since its introduction, the method has also been applied to many difficult problems including long-time regularity for the $3$D gravity-capillary water waves system \cite{DIPP2017}. Good introductory sources on STR include Germain \cite{Germain2011} and Kato and Pusateri \cite{KP2011}. For more history on STR and its connections to other approaches, see again \cite{GMS2009} and the references therein. 
 \par To understand why an STR approach is a promising way to deal with the three issues outlined earlier, it's helpful to recall some results from existing literature. First, Germain et al. \cite{GPR2016}  were able to establish modified scattering for the $p=2$ gKdV with the help of STR techniques. That is, an STR approach can handle situations where $p$ is relatively small. Additionally, in Kato and Pusateri's work on modified scattering for nonlinear Schr\"{o}dinger equations \cite{KP2011}, an STR framework was applied to estimate weighted norms without using the method of vector fields as in Hayashi and Naumkin's approach to the problem \cite{HN1998_nls}. Finally, the analysis of the two-horizontal-dimensional gravity-capillary water waves system by Deng et al. \cite{DIPP2017} demonstrates that STR can, at least in two space dimensions, handle nonzero inflection points in the governing dispersion relation. Therefore, STR has, in various different circumstances, been able to individually handle all three of the issues we discussed earlier. Accordingly, it's not unreasonable to apply the method to a problem where \emph{all three issues appear simultaneously}. 
 \par I emphasize that STR provides a framework, rather than an algorithm, for understanding the asymptotics of dispersive waves. Each PDE has its own particular resonances that need to be handled on a case-by-case basis. In particular, for the quartic gBBM one can't simply cut-and-paste from existing STR-based papers to produce a proof of theorem \ref{thm:main_thm}: below, we'll find that the quartic gBBM has ``anomalous'' non-null isolated resonances that, as far as I know, haven't been dealt with in the literature before. Also, the particularities of an STR approach are the main reason why my proof of theorem \ref{thm:main_thm} does not immediately carry over to the quintic case (in brief, changing the nonlinear power $p$ gives rise to different resonances).  
 \subsection{Space-Time Resonance Framework for gBBM}
\noindent
Now, we discuss the STR approach in more precise terms. For convenience, we define
\begin{equation}
    \label{eqn:eta4_defn}
    \eta_4 = \xi-\sum_{j=1}^{3}\eta_j.
\end{equation} 
We also define the \textbf{profile} of $u(t,x)$ (the solution to $p=3$ gBBM) by
$$
f(t,x) = e^{i\omega\left(\frac{1}{i}\partial_{x} \right) t}u(t,x)
$$
where $\omega(\xi)=\xi\langle\xi\rangle^{-2}$. Note that, if $u(t,x)$ was a solution to linBBM instead of gBBM, then $f = u(t=0,x)$, hence in particular $\partial_{t}f=0$. In other words, time-variation of our profile is the signature of nonlinearity. Next, we transform gBBM into an ODE for $\widehat{f}$. Using the PDE in \eqref{eqn:3bbm_cauchy_problem_asympt}, we know that 
\begin{equation*}
\partial_{t}\hat{f} = -i\omega(\xi)e^{it\omega(\xi)} \left(u^4\right)^{\wedge}.
\end{equation*}
Through repeated application of the convolution theorem for Fourier transforms, we find
\begin{align}\label{eqn:ODE_for_fhat_framework}
\partial_{t}\hat{f} = -i \left(\frac{1}{2\pi}\right)^{3/2}\omega(\xi) \int  \diff \eta_1 \diff \eta_2 \diff\eta_3 \ e^{-it\varphi} \hat{f}(\eta_1)  \hat{f}(\eta_2)   \hat{f}(\eta_3)   \hat{f}\left(\eta_4\right)
\end{align}
where the \textbf{nonlinear phase function} is 
\begin{equation}\label{eqn:nonlinear_phase_fnc}
\varphi(\eta_1,\eta_2,\eta_3;\xi) = -\omega(\xi)+ \sum_{j=1}^{4}\omega(\eta_j). 
\end{equation}
Note that we view $\xi$ as a parameter in this phase function. Also, notice that that this ODE for $\widehat{f}$ is really just the Fourier-space version of Duhamel's formula, if we treat the nonlinear piece of our PDE as a forcing term. By integrating \eqref{eqn:ODE_for_fhat_framework} over $[1,t]$, we obtain 
\begin{equation}\label{eqn:fhat_integrated_framework}
\hspace{-0.5cm}
    \widehat{f}(t,\xi) = \widehat{f}(1,\xi) -i  \frac{  \omega(\xi)}{\sqrt{(2\pi)^3}} \int_{1}^{t} \diff \tau  \int_{\Rthree}  \diff \eta_{1} \ \diff \eta_{2}\  \diff \eta_{3} \ e^{-i\tau\varphi} \  \widehat{f}(\eta_1) \ \widehat{f}(\eta_2)\ \widehat{f}(\eta_3)\  \widehat{f}\left(\eta_4\right).
\end{equation}
The above expression forms the bedrock for the analysis in section \ref{s:the_technical_proof}. 
\par  In physical terms, we can think of the integral in \eqref{eqn:fhat_integrated_framework} as representing the interaction between normal modes. As mentioned earlier, we want to study the long-time behaviour of $\hat{f}$ by determining whether or not any of these normal mode interactions are resonant. Mathematically, this is accomplished by viewing the right-hand side of \eqref{eqn:fhat_integrated_framework} as an \textbf{oscillatory integral}. Thinking in terms of the classical stationary phase estimate, we make the following definition: 
\begin{defn}
We call a quadruple $(\eta_{1}^{*},\eta_{2}^{*},\eta_{3}^{*};\xi^{*})$ a \textbf{space resonance} if 
$$\nabla_{\eta_1, \eta_2, \eta_3}\varphi\left(\eta_{1}^{*},\eta_{2}^{*},\eta_{3}^{*};\xi^{*}\right) =0 $$
or a \textbf{time resonance} if 
$$
\varphi\left(\eta_{1}^{*},\eta_{2}^{*},\eta_{3}^{*};\xi^{*}\right)  =0 . 
$$
If $(\eta_{1}^{*},\eta_{2}^{*},\eta_{3}^{*};\xi^{*})$ is both a space resonance and a time resonance, we call it a \textbf{space-time resonance}. 
\end{defn}
\noindent Appealing to stationary phase ideas, the dominant contributions to $\widehat{f}$ come from those regions of $\eta_1\eta_2\eta_3$-space that touch one or more space-time resonances: these are the regions where the oscillatory term $e^{-i\tau\varphi}$ is approximately constant, so no oscillatory cancellations in the integrand can occur. Therefore, once we control the near-resonant part of the integral in \eqref{eqn:fhat_integrated_framework}, the hardest part of our job is finished! This is the main punchline of the STR framework.
\subsection{Steps of the Proof}
\label{ss:steps_of_proof}
\noindent Now that we have developed a heuristic understanding of the STR approach to gBBM, I can meaningfully outline the major steps in my proof of theorem \ref{thm:main_thm}. 
\begin{enumerate}
    \item Obtain a dispersive estimate controlling the amplitude loss of linear waves over time. This estimate, stated precisely in corollary \ref{cor:ultimate_disp_est}, tells us that $\left\|u\right\|_{L^{\infty}_{x}}$ (where $u$ is our \emph{nonlinear} wave) can be controlled by various norms of the Fourier-transformed profile $\widehat{f}$.
    \item Based on this new linear dispersive estimate, define a Banach space $X_{T}$ of functions depending on $t\in \left[1,T\right]$ and $x\in \mathbb{R}$ in such a way that 
    \begin{equation}\label{eqn:smallness_in_XT_implies}
    \left\|u\right\|_{X_{T}}\lesssim \eps_0 \Rightarrow \left\|u\right\|_{L^{\infty}_{x}} \lesssim \eps_0 t^{-\frac13} \quad \forall \ t\in \left[1,T\right]. 
    \end{equation}
    Thus, to establish linear dispersive decay of $u(t,x)$ it suffices to show that there exists a global-in-time solution $u\in X_{\infty}$.
    \item Show that there is some $T\in \left(1,2\right)$ such that \eqref{eqn:fhat_integrated_framework} has a unique solution living in $X_{T}$. Further, show that the map $\tau\mapsto \left\|u\right\|_{X_{\tau}}$ is continuous on $\left[1,T\right]$. This is the content of proposition \ref{prop:local_theory}.
    \item Make the \textit{a priori} assumption that there is some $\eps_1\in \left(\eps_0, 1\right)$ such that $\left\|u\right\|_{X_{T}}\leq \eps_1$. Then, use stationary phase-type arguments and harmonic analysis tools (especially \textbf{multilinear estimates}) to control the Duhamel term on the right-hand side of \eqref{eqn:fhat_integrated_framework} in terms of $\eps_1$. Along the way, one needs to decompose $\eta_1\eta_2\eta_3\xi$-space into different pieces to properly separate non-resonant regions from resonant ones. This is accomplished via \textbf{Littlewood-Paley (LP) decomposition}. Handling all the cases resulting from LP decomposition is the longest and most technical part of the argument. 
    \item After controlling the oscillatory integral, one is left with an estimate of the form 
    $$
     \left\|u\right\|_{X_{T}} \leq \eps_0 + C\eps_1^4
    $$
    for some absolute constant $C>0$. This is established carefully in proposition \ref{prop:bootstrap_close}. By shrinking $\eps_0$ (based on $C$) and choosing $\eps_1=2\eps_0$, this implies the stronger bound $\left\|u\right\|_{X_{T}} \leq \frac32 \eps_0$. 
    \item Iterate steps (iii)-(v) to show that $u(t,x)$ can be extended out to exist on $t\in\left[1,\infty\right)$ in such a way that $\left\|u\right\|_{X_{\infty}} \lesssim \eps_0$, which by design ensures us that $u$ has the dispersive decay of a solution to linBBM. The iteration is guaranteed to converge by the \textbf{bootstrap principle}.  
    \item Once the global solution $u\in X_{\infty}$ is obtained, a straightforward modification of the arguments in item (iv) show that $u$ scatters in $L^2_{x}$ (see proposition \ref{prop:scattering_in_L^2} and section \ref{s:scattering_proof}). 
\end{enumerate} \ \par \noindent 
 
\subsection{Notation}
 \label{s:notation}
\noindent Below, $d$ always denotes a positive integer greater than or equal to $1$. 
\begin{itemize}
\item If $a,b\in \mathbb{R}$, we say that $a\lesssim b$ if there exists an absolute constant $c>0$ such that $a\leq bc$.  We say that $a\simeq b$ or $a\sim b$ if $a\lesssim b$ and $b\lesssim a$. We say $a=\mathcal{O}\left(b\right)$ if $|a|\lesssim b$. 
\item If $a,b\in \mathbb{R}$ and $\lambda$ is a parameter or a function, we say that $a\lesssim_{\lambda} b$ if there exists  $c=c\left(\lambda\right)>0$ such that $a\leq bc$. Similarly, if we include multiple subscripts on the ``$\lesssim$'', then the bounding constant depends on all the subscripts. 
\item Given a real number $a$, we define its \textbf{Japanese bracket} by $\langle a\rangle =\left(1+a^2\right)^{\frac12}$. 
\item Given a real number $a$, $\left(a\right)^{-}$ denotes a number that is close to, but strictly less than, $a$ (for concreteness, we can say, $a-\left(a\right)^{-}<\frac{1}{2}$). 
\item If $(a,b)\subseteq \mathbb{R}$, then $(a,b)^{\mathsf{c}}$ denotes the complement of this interval. 
    \item For $q\in \left[1,\infty\right]$, $L^{q}_{x}\left(\mathbb{R}^{d}\right)$ denotes the usual Lebesgue space.
   \item For $s\geq 0$, $H^{s}_{x}\left(\mathbb{R}^d\right)$ denotes the usual Sobolev space.
   \item $\mathcal{S}_{x}\left(\mathbb{R}^d\right)$ denotes the Schwartz space. 
   \item Given a Banach space $X$, we let $C\left([0,\infty); X\right)$ denote the space of continuous curves in $X$.
   \item If $f\in L^1_{x}\left(\mathbb{R}^d\right)$, we denote its Fourier transform by $\widehat{f}$. 
    \item If $f\in L^2_{\xi}\left(\mathbb{R}^d\right)$, we denote its inverse Fourier transform by $\overset{\vee}{f}$.
\end{itemize}
We also need to set up some conventions for Littlewood-Paley (LP) decomposition. First, we define our bump functions. 
\begin{defn}\label{defn:bumps}
\noindent 
\begin{enumerate}
    \item $\phi(\xi)$ denotes a fixed smooth bump function supported in $[-2,2]$ with $\phi\equiv 1$ on $\left[-1, 1\right]$. 
\item $\psi(\xi) \doteq \phi(\xi)-\phi\left(2\xi\right)$. 
\item Given $k\in \mathbb{Z}$, $\phi_{\leq k}(\xi) \doteq \phi\left(\frac{\xi}{2^{k}}\right)$ and $\psi_{k}(\xi) \doteq \psi\left(\frac{\xi}{2^{k}}\right)$. 
\end{enumerate} \ \par\noindent 
\end{defn}

\noindent Note that each $\psi_k$ is supported inside a one-dimensional annulus. For this reason, we say each $\psi_k$ is an \textbf{annular bump function}. We sometimes call the $\phi_{\leq k}$ \textbf{vanilla bump functions} when we want to contrast them with the annularly supported $\psi_k$.
\begin{defn}[LP Projections]
For any $f\in L^2_{x}\left(\mathbb{R}\right)$ and any $k\in \mathbb{Z}$, define
    \begin{align*}
        f_{k} \doteq \left(\psi_{k}\widehat{f}\right)^{\vee} \quad \text{and} \quad f_{\leq k} \doteq \left(\phi_{\leq k}\widehat{f}\right)^{\vee}.
    \end{align*} 
    \noindent $f_{k}$ is called the $k^{\text{th}}$ \textbf{LP piece of} $f(x)$. The multiplier operators $f\mapsto f_{k}$, $f\mapsto f_{\leq k}$ are called \textbf{LP projections}. 
\end{defn}
\noindent 
In practice, if $u(t,x)$ is a wavepacket solving gBBM, we think of $u_{k}(t,x)$ as a localized version of a normal mode with frequency $\xi\sim 2^{k}$. The entire wavepacket can be reconstructed as a sum of LP pieces: 
\begin{lemma}[LP Decomposition]
For any $f\in L^2_{x}$ and any fixed $k_{\text{lo}}\in \mathbb{Z}$, the following identities hold:
\begin{align*}
    f(x) &= \sum_{k\in \mathbb{Z}} f_{k}(x) \phantom{f_{\leq 0}(x) +} \quad \left(\text{\textbf{homogeneous LP decomposition}}\right)
    \\
    f(x) &= f_{\leq k_{\text{lo}}}(x) + \sum_{k> k_{\text{lo}}} f_{k}(x)  \quad \left(\text{\textbf{inhomogeneous LP decomposition}}\right). 
\end{align*}
\end{lemma}
\qed 

\section{Linear Estimates}
\noindent In this section, I establish a dispersive decay estimate for the linBBM flow. Variants of this dispersive estimate have been obtained much earlier by Albert \cite{Albert1986, Albert1989}, Souganidis \& Strauss \cite{SS1990}, and Dziuba\'{n}ski \& Karch \cite{DK1996}. The new estimate here involves a weighted norm, thereby providing greater flexibility with regards to possible directions for integration by parts in the Duhamel term of \eqref{eqn:fhat_integrated_framework}. Crucially, don't need $u(t,x)$ to solve linBBM for these estimates to apply.
\begin{thm}[Dispersive Estimates on LP Pieces] \label{thm:LP_disp_est}
 Let $C_{\mathrm{hi}} \geq 2^{4}, C_{\mathrm{lo}}\leq 2^{-2}$ and pick $s\geq \frac{11}{2}$. For any sufficiently nice function $u(t,x)$ with profile $f(t,x)$, we have the bounds 
\begin{equation}
    \label{disp_est_on_LP_pieces}
    \left\| u_{k}(t,x)\right\|_{L^{\infty}_{x}} \lesssim \begin{cases}
    2^{-k(s-1)}\left\|f\right\|_{H^{s}_{x}} \quad & \text{if} \quad 2^{k} \geq C_{\mathrm{hi}}t^{\frac19}
    \\ 
    t^{-\frac12}2^{\frac{3k}{2}}\left\|\widehat{f}\right\|_{L^{\infty}_{\xi}} + t^{-\frac34}2^{\frac{9k}{4}} \left\|\partial_{\xi}\widehat{f}_{k}\right\|_{L^2_{\xi}} \quad & \text{if} \quad 2^{3} \leq 2^{k} < C_{\mathrm{hi}}t^{\frac19}
    \\ 
    t^{-\frac13}\left\|\widehat{f}\right\|_{L^{\infty}_{\xi}} + t^{-\frac12} \left\|\partial_{\xi}\widehat{f}_{k}\right\|_{L^2_{\xi}}  \quad & \text{if} \quad 2^{-1} \leq 2^{k} < 2^{3}
    \\ 
      t^{-\frac12}2^{-\frac{k}{2}}\left\|\widehat{f}\right\|_{L^{\infty}_{\xi}} + t^{-\frac34}2^{-\frac{3k}{4}}\left\|\partial_{\xi}\widehat{f}_{k}\right\|_{L^2_{\xi}}  \quad & \text{if} \quad C_{\mathrm{lo}}t^{-\frac13} \leq 2^{k} < 2^{-1}
      \\ 
     2^{k} \left\|\widehat{f}\right\|_{L^{\infty}_{\xi}} \quad & \text{if} \quad 2^{k}< C_{\mathrm{lo}}t^{-\frac13}. 
    \end{cases}
\end{equation}
\end{thm}

\begin{proof}
\
\par \noindent 
\textbf{Case 1:  $2^{k} \geq C_{\mathrm{hi}}t^{\frac19}$}
\par \noindent Applying Sobolev embedding and then Bernstein's inequalities gives
\begin{align*}
    \left\|u_{k}\right\|_{L^{\infty}_{x}} &\lesssim  2^{-k(s-1)}\left[ 2^{k(s-1)}\left\|f_{k}\right\|_{L^{2}_{x}} +   2^{k(s-1)}\left\|f_{k}\right\|_{\dot{H}^{1}_{x}}\right] \lesssim 2^{-k(s-1)} \left\|f_{k}\right\|_{H^{s}_{x}}.
    \end{align*}
\par \noindent 
\textbf{Case 2:  $2^{3}\leq 2^{k} < C_{\mathrm{hi}}t^{\frac19}$}
\par \noindent
We start from the oscillatory integral representation of $u_{k}$: 
$$
u_{k}(t,x) =\frac{1}{\sqrt{2\pi}}  \int_{-\infty}^{\infty}  e^{it\Phi(\xi)} \ \widehat{f}_{k}(t,\xi) \ \diff \xi,
$$
where the phase $\Phi(\xi)$ is defined by $\Phi(\xi) = \xi x - \omega(\xi) t$. The support of the integrand may include the points $\pm\xi_0$ satisfying
\begin{equation}
    \label{eqn:xi0_defn}
      \xi_0>0 \quad \text{and} \quad \partial_{\xi}\Phi = 0 \iff x= \omega'(\xi_0)t, 
\end{equation}
so obtaining time decay via a \emph{global} integration by parts is not allowed. We instead split into two subcases based on when integration by parts is permitted. To cleanly identify appropriate subcases based on $x/t$, first observe that for $\xi\in[2^3,\infty)$ we have
\begin{equation}
    \label{eqn:far_field_cg_asymptotics}
    \frac19 |\xi|^{-2} \leq |\omega'(\xi)| \leq 2|\xi|^{-2},
\end{equation}
so $|\omega'(\xi)|\sim |\xi|^{-2}$ on the region of interest. Accordingly, the stationary frequencies $\pm\xi_0$ satisfy 
\begin{equation}
    \label{eqn:case2_subdefns}
   \frac13 \Big|\frac{t}{x}\Big|^{\frac12} \leq |\xi_0| \leq 2^{\frac12} \Big|\frac{t}{x}\Big|^{\frac12}. 
\end{equation}
Consequently, integration by parts is perfectly fine for $2^{k} \in \left[2^3, C_{\text{hi}}t^{\frac19}\right) \cap \left[2^{-4}\Big|\frac{t}{x}\Big|^{\frac12}, 2^2\Big|\frac{t}{x}\Big|^{\frac12}\right]^{\mathsf{c}}$ but a different strategy must be developed if we drop the complement symbol $\mathsf{c}$. 
\ 
\par \noindent 
\textbf{Subcase 2.1:  $2^{k} \in \left[2^3, C_{\mathrm{hi}}t^{\frac19}\right) \cap \left[2^{-4}\Big|\frac{t}{x}\Big|^{\frac12}, 2^2\Big|\frac{t}{x}\Big|^{\frac12}\right]^{\mathsf{c}}$}
\par \noindent
First, we must bound $\left|\partial_{\xi}\Phi\right|$ from below. To begin, consider $2^{k}<2^{-4}\left|\frac{t}{x}\right|^{\frac12}$. Then, we have $\left|\frac{x}{t}\right| < 2^{-8}2^{-2k}$, so using \eqref{eqn:far_field_cg_asymptotics} we have $\left|\partial_{\xi}\Phi\right| \gtrsim t2^{-2k}$. Alternatively, if $2^{k}>2^{2}\left|\frac{t}{x}\right|^{\frac12}$ then $\left|\frac{x}{t}\right| > 2^{4}2^{-2k}$, hence \eqref{eqn:far_field_cg_asymptotics} gives $\left|\partial_{\xi}\Phi\right| \gtrsim t2^{-2k}$. So, in this subcase we always have 
\begin{align}
    \left|\partial_{\xi}\Phi\right| &\gtrsim t2^{-2k}   \label{eqn:2.1_der_bound}.
\end{align}
Additionally, notice that
\begin{equation} \label{eqn:2.1_acc_bound}
    |\partial^2_{\xi}\Phi| \lesssim t|\xi|^{-3} \lesssim t2^{-3k}. 
\end{equation}
Integrating by parts in $u_{k}$ then gives 
\begin{align*}
    u_{k} = -\frac{i}{\sqrt{2\pi}} \int_{-\infty}^{\infty} e^{it\Phi(\xi)} \left(-\frac{\partial^2_{\xi}\Phi}{\left(\partial_{\xi}\Phi\right)^2} \widehat{f}_{k} + \frac{1}{\partial_{\xi}\Phi} \partial_{\xi}\widehat{f}_{k}\right)\ \diff \xi. 
\end{align*}
Applying \eqref{eqn:2.1_der_bound}, \eqref{eqn:2.1_acc_bound}, and Cauchy-Schwarz yields
\begin{align*}
    |u_k| \lesssim \int_{-\infty}^{\infty} \left\|\widehat{f}_{k}\right\|_{L^{\infty}_{\xi}} \frac{t2^{-3k}}{\left(t2^{-2k}\right)^2} \mathbf{1}_{|\xi|\sim 2^{k}} +\left|\partial_{\xi}\widehat{f}_{k}\right| \frac{1}{t2^{-2k}} \mathbf{1}_{|\xi|\sim 2^{k}} \ \diff \xi \lesssim t^{-1} 2^{2k} \left\|\widehat{f}_{k}\right\|_{L^{\infty}_{\xi}} + t^{-1} 2^{\frac{5k}{2}} \left\|\partial_{\xi}\widehat{f}_{k}\right\|_{L^2_{\xi}}.
\end{align*}
To complete this subcase, observe that $2^{k} \lesssim t^{\frac19}$ implies $t^{-1}2^{2k} \leq t^{-\frac12}2^{\frac{3k}{2}}$ and $t^{-1}2^{\frac{5k}{2}} \leq t^{-\frac34}2^{\frac{9k}{4}}$.
\
\par \noindent 
\textbf{Subcase 2.2:  $2^{k} \in \left[2^3, C_{\mathrm{hi}}t^{\frac19}\right) \cap \left[2^{-4}\Big|\frac{t}{x}\Big|^{\frac12}, 2^2\Big|\frac{t}{x}\Big|^{\frac12}\right]$}
\par \noindent
We perform a second inhomogeneous LP localization centred around $\xi_0$ (the nonnegative frequency of stationary phase) at a dyadic scale $2^{\ell}$. Let's take $\ell$ ranging among the integers $[\ell_0, k+32]$, where $\ell_0$ is the \emph{smallest} integer satisfying $2^{\ell_0} \geq t^{-\frac12}2^{\frac{3k}{2}}$. Note that, since $\ell_0$ is the smallest integer with the above property, we have
\begin{equation}\label{eqn:thomas}
2^{\ell_0} \lesssim t^{-\frac12}2^{\frac{3k}{2}}.
\end{equation}
Additionally, the ``32'' appearing in $\ell\in [\ell_0, k+32]$ is just to make sure the $2^{\ell}$ patches cover the support of our integrand; any other large number would also work. Also, without loss of generality, assume $\mathrm{supp}\ \widehat{u}\subseteq [0,\infty)$. This ensures we only have one stationary point $\xi_0>0$ to worry about. 
We decompose $u_{k}$ according to $u_{k} = \sum_{\ell=\ell_0}^{k+32} u_{k,\ell}$ with 
\begin{equation}
    u_{k,\ell} = \begin{cases} \frac{1}{\sqrt{2\pi}}  \int_{-\infty}^{\infty}  e^{it\Phi(\xi)} \ \widehat{f}_{k}(t,\xi) \ \phi_{\ell_0}(\xi-\xi_0) \ \diff \xi \quad &\text{if} \quad \ell=\ell_0,
    \\ 
    \frac{1}{\sqrt{2\pi}}  \int_{-\infty}^{\infty}  e^{it\Phi(\xi)} \ \widehat{f}_{k}(t,\xi) \ \psi_{\ell}(\xi-\xi_0) \ \diff \xi  \quad &\text{if} \quad \ell>\ell_0,
    \end{cases} 
\end{equation}
then estimate each $|u_{k,\ell}|$ and check that $\sum_{\ell}|u_{k,\ell}|$ is always finite. Controlling $|u_{k,\ell_0}|$ is trivial: using \eqref{eqn:thomas}, 
\begin{align}
    |u_{k,\ell_0}| &\lesssim t^{-\frac12}2^{\frac{3k}{2}} \left\|\widehat{f}_{k}\right\|_{L^{\infty}_{\xi}}. \label{eqn:case2.2_ell0_bound}
\end{align}
To estimate the other $u_{k,\ell}$'s, we need to bound $|\partial_{\xi}\Phi|$ and $|\partial^2_{\xi}\Phi|$ below when $|\xi-\xi_0|\sim 2^{\ell}$. Using the easily verified bound
\begin{equation}
    \label{eqn:case2.2_omega''}
    2^{-6}|\xi|^{-3} \leq |\omega''(\xi)| \leq 2^{3}|\xi|^{-3}
\end{equation}
we immediately obtain 
\begin{equation}
    \label{eqn:2.2_acc_bound}
    |\partial^2_{\xi}\Phi| \lesssim t2^{-3k}. 
\end{equation}
Then, we use Taylor expansion about $\xi_0$ to discover
\begin{align*}
    |\partial_{\xi}\Phi|
    &\geq t \ |\xi-\xi_0|\ \inf_{|\xi-\xi_0|\sim 2^{\ell}} |\omega''(\xi)|.
\end{align*}
Combining the above with  \eqref{eqn:case2.2_omega''} gives
\begin{align}
\label{eqn:2.2_der_bound}
    |\partial_{\xi}\Phi| &\gtrsim t 2^{-3k} 2^{\ell}
\end{align}
Now we can integrate by parts and use \eqref{eqn:2.2_acc_bound}, \eqref{eqn:2.2_der_bound}, and Cauchy-Schwarz to find
\begin{align}
    |u_{k,\ell}| &\lesssim  t^{-1}2^{3k-\ell} \left\|\widehat{f}_{k}\right\|_{L^{\infty}_{\xi}}+ t^{-1}2^{3k-\frac{\ell}{2}}\left\|\partial_{\xi}\widehat{f}_{k}\right\|_{L^{2}_{\xi}}.  \label{eqn:case2.2_higher_ells}
\end{align}
Putting \eqref{eqn:case2.2_ell0_bound}, \eqref{eqn:case2.2_higher_ells}, and \eqref{eqn:thomas} together gives what we want:
\begin{align*}
    |u_{k}| &\lesssim t^{-\frac12}2^{\frac{3k}{2}} \left\|\widehat{f}_{k}\right\|_{L^{\infty}_{\xi}} +  t^{-\frac34}2^{\frac{9k}{4}}\left\|\partial_{\xi}\widehat{f}_{k}\right\|_{L^{2}_{\xi}}.
\end{align*}
\par \noindent 
\textbf{Case 3:  $2^{-1} \leq 2^{k} < 2^{3}$} 
\par \noindent
This case is noteworthy because it allows for $2^{k}\approx \sqrt{3}$, so the support of our integrand may include a \emph{degenerate} stationary point. We start by splitting the oscillatory integral $u_{k}$ into three pieces: 
\begin{enumerate}
    \item $|\xi| \in [2, 2^{4})$: since $2^{k}$ is absolutely bounded this is easier than case 2, so we skip over the details. 
        \item $|\xi| \in \left[\frac32, 2\right)$: here, we are near $|\xi|=\sqrt{3}$, which is a degenerate point of stationary phase when $\frac{x}{t}=-\frac18$. Our approach is to first localize about $\xi=\sqrt{3}$ at a scale $2^{m}$. From there, we can bound the non-stationary contributions easily. For the stationary contributions, however, we must localize again about all stationary points at a new scale $2^{\ell}$. The interplay between the different scales $2^{m}$ and $2^{\ell}$ becomes important when bounding $|\partial_{\xi}\Phi|$ from below. 
    \item $|\xi| \in \left[2^{-2}, \frac32\right)$: here, we have no degenerate stationary points, so at first glance it seems like we can just copy the arguments of case 2. However, since the linBBM group velocity vanishes at $|\xi|=1$, this scenario is actually more like a hybrid between case 2 and $|\xi|\in \left[\frac23, 2\right)$. 
\end{enumerate} \  \par \noindent
So, we are splitting the oscillatory integral with indicator functions: 
\begin{align*}
u_{k}(t,x) &=\frac{1}{\sqrt{2\pi}} \int_{-\infty}^{\infty}  e^{it\Phi(\xi)} \ \widehat{f}_{k}(t,\xi) \ \left(\mathbf{1}_{2\leq|\xi|\leq 2^{4}} +\mathbf{1}_{2^{-2}\leq|\xi|\leq \frac32}+ \mathbf{1}_{\frac32\leq|\xi|\leq 2}\right) \ \diff \xi
\\
&\doteq u^{I}_k + u^{II}_k + u^{III}_k. 
\end{align*}
The use of such indicator functions gives rise to boundary terms when we integrate by parts. Thus $|\xi| \in [2, 2^{4})$ is technically a bit different from case 2. However, as we'll see below in the most difficult situation $|\xi| \in \left[\frac32, 2\right)$, the boundary terms cause no trouble. I believe it is possible to obtain the correct result using carefully defined ``soft'' cutoff functions instead (and this approach would avoid picking up boundary terms), but for myself using indicator functions just seems simpler. Additionally, I remark that we have already explained why controlling $u^{I}_{k}$ is trivial. 
\par \noindent 
\textbf{Bounds on $u^{III}_{k}$: Setup}
\par \noindent 
We localize about $|\xi|=\sqrt{3}$ at the scale $2^{m}$, even when $\frac{x}{t}$ is such that there are no stationary points in the support of the integrand. Let $m_0$ be the smallest integer satisfying $2^{m_0} \geq t^{-\frac13}$. Without loss of generality suppose again that $\mathrm{supp}\left\{\widehat{u}\right\}\subseteq [0,\infty)$ so we don't have to carry around absolute values. We break $u_{k}^{III}$ into pieces $u_{k,m}$ via $u^{III}_{k} = \sum_{m\geq m_0}^{k+32} u_{k,m}$ where we have defined 
\begin{equation}
    u_{k,m} = \begin{cases} \frac{1}{\sqrt{2\pi}}  \int_{-\infty}^{\infty}  e^{it\Phi(\xi)} \ \widehat{f}_{k}(t,\xi) \ \phi_{m_0}\left(\xi-\sqrt{3}\right) \ \mathbf{1}_{\left[\frac32,2\right)}  \ \diff \xi \quad &\text{if} \quad m=m_0,
    \\ 
    \frac{1}{\sqrt{2\pi}}  \int_{-\infty}^{\infty}  e^{it\Phi(\xi)} \ \widehat{f}_{k}(t,\xi) \ \psi_{m}\left(\xi-\sqrt{3}\right) \ \mathbf{1}_{\left[\frac32,2\right)}  \ \diff \xi  \quad &\text{if} \quad m>m_0.
    \end{cases} 
\end{equation}
Notice that for $m=m_0$ we trivially have the bound
\begin{equation}
    \label{eqn:bnd_near_sqrt{3}}
    \left\|u_{k,m_0}\right\|_{L^{\infty}_{x}} \lesssim t^{-\frac13} \left\|\widehat{f}\right\|_{L^{\infty}_{\xi}},
\end{equation}
so it remains to control $|u_{k,m}|$ for $m>m_0$. For certain $m>m_0$, we encounter no stationary points, and for all other $m>m_0$ we have to deal with \emph{two} stationary points: $\xi_0 \in \left[\frac32, \sqrt{3}\right]$ and $r(\xi_0) \in \left[\sqrt{3},2\right]$. Subcase 3.1 tackles the non-stationary case, and Subcase 3.2 tackles the stationary case (which involves another localization at a scale $2^{\ell}$). Once both subcases are done, we'll sum over $m$ and thereby control $\left|u^{III}_{k}\right|$. 
\ 
\par \noindent 
\textbf{Subcase 3.1:  $|\xi|  \in \left[\frac32, 2\right)$ and $2^{m}\in \left[2^{-1}\left|\frac{x}{t}+\frac18\right|^{\frac12}, 2^{5}\left|\frac{x}{t}+\frac18\right|^{\frac12}\right]^{\mathsf{c}}$} 
\par \noindent
Taylor expansion shows 
\begin{equation}
\label{eqn:case3.3_omega'_expansion}
2^{-5} \left||\xi|-\sqrt{3}\right|^2 \leq \left|\omega'(\xi)+\frac18\right| \leq 2^{-2} \left||\xi|-\sqrt{3}\right|^2, 
\end{equation}
so $\left|\omega'(\xi)+\frac18\right|\sim 2^{2m}$. Using \eqref{eqn:case3.3_omega'_expansion}, we see that there are no stationary points in this subcase. Therefore, it seems like we can na\"{i}vely copy Subcase 2.1. However, unlike Subcase 2.1, here our best lower bound is
\begin{equation}
    \label{eqn:case3.3_phi'_bnd_ell}
    \left|\partial_{\xi}\Phi\right| \gtrsim t 2^{2m}. 
\end{equation}
Additionally, Taylor expansion gives the upper bound.
\begin{equation}
\label{eqn:case3.3_phi''_bnd_ell}
\left|\partial_{\xi}^2\Phi\right| \lesssim t2^{m}. 
\end{equation}
Integration by parts in $u_{k,m}$, \eqref{eqn:case3.3_phi'_bnd_ell}, \eqref{eqn:case3.3_phi''_bnd_ell}, and earlier arguments all together give 
\begin{align}
\left|u_{k,m}\right|&\lesssim t^{-1}2^{-2m} \left\|\widehat{f}\right\|_{L^{\infty}_{\xi}} + t^{-1}2^{-\frac{3m}{2}} \left\|\partial_{\xi}\widehat{f}_{k}\right\|_{L^2_{\xi}}. \label{eqn:subcase3.3_result}
 \end{align}
 We discuss summation over $m$ a bit later. 
\ 
\par \noindent 
\textbf{Subcase 3.2:  $|\xi| \in \left[\frac32, 2\right)$ and $2^{-1}\left|\frac{x}{t}+\frac18\right|^{\frac12}\leq 2^{m}\leq 2^{5}\left|\frac{x}{t}+\frac18\right|^{\frac12} $}
\par \noindent
This case includes two positive stationary points, namely $\xi_0 \in \left[\frac32, \sqrt{3}\right]$ and $r(\xi_0) \in \left[\sqrt{3}, 2\right)$, which coincide when $\xi_0=\sqrt{3}$. Accordingly, we have to choose a new localization strategy that takes into account both $\xi_0$ and $r(\xi_0)$. First, we have to localize each $u_{k,m}$ again at a scale $2^{\ell}$. We choose our scale so that $\ell\geq m_0$ where $m_0$ is the same $m_0$ defined in earlier. We then split $u_{k,m}$ again according to $u_{k,m} = \sum_{\ell=m_0}^{k+32} u_{k,\ell,m}$ with 
\begin{equation}
\hspace{-1cm}
    u_{k,\ell,m} \simeq \begin{cases}  \int_{-\infty}^{\infty} e^{it\Phi}  \widehat{f}_{k}   \psi_{m}(\xi-\sqrt{3}) \left[ \phi_{m_0}(\xi-\xi_0) + \phi_{m_0}(\xi-r(\xi_0)) \right]  \diff \xi \quad &\text{if} \quad \ell=m_0,
    \\ 
     \int_{-\infty}^{\infty}  e^{it\Phi} \widehat{f}_{k}  \psi_{m}(\xi-\sqrt{3})  \left[\psi_{\ell}(\xi-\xi_0)\mathbf{1}_{\left[\frac32,\sqrt{3}\right]}  +  \psi_{\ell}(\xi-r(\xi_0))\mathbf{1}_{\left[\sqrt{3},2\right]} \right]  \diff \xi  \quad &\text{if} \quad \ell>m_0.
    \end{cases} 
\end{equation}
Notice that a point which is in the support of the integrand of $u_{k,\ell,m}$ with $\ell>m_0$ \emph{does not necessarily satisfy} $|\xi-\xi_0|\sim 2^{\ell}$. However, if $|\xi-\xi_0|\sim 2^{\ell}$ doesn't hold for a particular $\xi$, then $|\xi-r(\xi_0)|\sim 2^{\ell}$ definitely does. For $\ell=m_0$, we immediately have
\begin{equation}
\label{eqn:case3.4_ell0}
   \sum_{m=m_0}^{k+32} \left| u_{k,m_0,m} \right| \lesssim t^{-\frac13} \left\|\widehat{f}\right\|_{L^{\infty}_{\xi}}
\end{equation}
since the $\psi_{m}$ are nonnegative and form a partition of unity. For $\ell>m_0$ we start by using careful Taylor expansions to establish the bound
\begin{equation}\label{eqn:3.4_parabola_stronger}
    |\partial_{\xi}\Phi(\xi)|\gtrsim t2^{\ell}2^{m}
\end{equation}
for all $\xi$ in the support of the integrand of $u_{k,\ell,m}$. Taylor expansion about $\xi_0$ or $r(\xi_0)$ also shows that
\begin{equation}
    \label{eqn:3.4:Phi''_bnd}
    \left|\partial_{\xi}^2\Phi\right| \lesssim t2^{\ell}
\end{equation}
in the support of the integrand. We are now ready to integrate by parts in $u_{k,\ell,m}$. Doing so and using \eqref{eqn:3.4_parabola_stronger} and \eqref{eqn:3.4:Phi''_bnd} in the vein of subcase 3.1, we find
\begin{align}
    |u_{k,\ell,m}| 
    &\lesssim  t^{-1} 2^{-\ell}2^{-m} \left\|\widehat{f}\right\|_{L^{\infty}_{\xi}} + t^{-1}2^{-\frac{3\ell}{4}}2^{-\frac{3m}{4}} \left\|\partial_{\xi}\widehat{f}_{k}\right\|_{L^{2}_{\xi}}.  \label{eqn:case3.4_ell>ell0}
\end{align}
Taking the sum over $\ell> m_0$ using  \eqref{eqn:case3.4_ell>ell0} gives
\begin{equation}\label{eqn:subcase3.4_result}
 \left| u_{k,m} \right| \lesssim  t^{-\frac23} 2^{-m} \left\|\widehat{f}\right\|_{L^{\infty}_{\xi}} + t^{-\frac34} 2^{-\frac{3m}{4}} \left\|\partial_{\xi}\widehat{f}_{k}\right\|_{L^{2}_{\xi}}. 
\end{equation}
\ 
\par \noindent 
\textbf{Bounds on $u^{III}_{k}$: Endgame} 
\par \noindent
We now estimate $u^{III}_{k}$ by summing up the $u_{k,m}$. Putting  \eqref{eqn:bnd_near_sqrt{3}}, \eqref{eqn:subcase3.3_result}, \eqref{eqn:case3.4_ell0},  and \eqref{eqn:subcase3.4_result} together, we obtain the estimate we want: 
\begin{align*}
    \left|u_{k}^{III}\right|
    &\lesssim t^{-\frac13}\left\|\widehat{f}\right\|_{L^{\infty}_{\xi}} + t^{-\frac12}\left\|\partial_{\xi}\widehat{f}_{k}\right\|_{L^2_{\xi}}.
\end{align*}
\par \noindent 
\textbf{Bounds on $u^{II}_{k}$}
\par \noindent 
Recall from subcase 2.1 that the philosophy for bounding nonstationary contributions is to properly balance terms so that we can prove a bound like 
    $$
    |\partial_{\xi}\Phi| \gtrsim t \  \min_{\substack{|\xi| \in \left[2^{-2}, \frac32\right) \\ \xi \ \text{not stationary}}}|\omega'(\xi)|
    $$
    However, $\omega'(\pm 1) =0$, so even when there is no stationary point in the interval of interest, the best lower bound we can get is $|\partial_{\xi}\Phi| \gtrsim t \ \left|1-|\xi|\right|$. This means we must perform another LP localization around $|\xi|=1$. Since we have already worked through the details of such a splitting in a more difficult situation, I omit the details. 

\par \noindent 
\textbf{Case 4:  $C_{\mathrm{lo}}t^{-\frac13} \leq 2^{k} < 2^{-1}$} 
\par \noindent
While the temporal frequency $\omega$ satisfies $|\omega'| \sim |\xi|^{-2}$ in case 2, here we find that $|1-\omega'(\xi)| \sim |\xi|^2$. However, the actual mechanics of the proof are identical in both cases.
\
\par \noindent
\textbf{Case 5: $2^{k}\leq C_{\mathrm{lo}}t^{-\frac13}$}
\par \noindent
 Using the size of the Fourier region we integrate over, we immediately have $\left| u_{k}(t,x)\right| \lesssim  2^{k} \left\|\widehat{f}(t,\xi)\right\|_{L^{\infty}_{\xi}}$. 
\end{proof}
\begin{cor}[Dispersive Estimate for linBBM Flow]
\label{cor:ultimate_disp_est}
For $s\geq \frac{11}{2}$, we have 
\begin{equation}
   \left\|u\right\|_{L^{\infty}_{x}} \leq \sum_{k\in \mathbb{Z}} \left\|u_{k}\right\|_{L^{\infty}_{x}} \lesssim t^{-\frac13}  \left\|\widehat{f}\right\|_{L^{\infty}_{\xi}}  + t^{-\frac12}\left(\left\|xf(t,x)\right\|_{L^{2}_{x}}+ \left\|f(t,x)\right\|_{H^{s}_{x}}\right). 
\end{equation}
\end{cor}
\qed
\section{Proof of Theorem \ref{thm:main_thm}}
\noindent With our linear dispersive estimate in hand, we are ready to prove theorem \ref{thm:main_thm}. First, recall that our initial state $u_0(x)$ satisfies, for $s\geq 100$ and some positive $\eps_0\ll 1$,
\begin{equation}
\label{eqn:smallness_condition_on_ICs}
\left\|u_0\right\|_{X_1} \doteq \left\|\widehat{u_0}(\xi)\right\|_{L^{\infty}_{\xi}} + \left\|xu_0(x)\right\|_{L^2_{x}} + \left\|u_0(x)\right\|_{H^{s}_{x}} < \eps_0. 
\end{equation}
The Banach space $X_{T}$ is defined by 
    $$
    X_{T} \doteq \left\{u \ \text{such that} \ f=e^{it\omega\left(\frac{1}{i}\partial_{x}\right)}u\in C\left(\left[1,T\right]; X_{1}\right) \ \text{with} \ \left\|u\right\|_{X_{T}} < \infty  \right\},
    $$
 where our \textbf{bootstrap norm} is 
 \begin{equation}\label{eqn:bootstrap_norm_defn}
        \left\|u\right\|_{X_{T}} \doteq \sup_{t\in [1,T]} \left\{\left\|\widehat{f}\right\|_{L^{\infty}_{\xi}} + t^{-p_0}\left\|xf\right\|_{L^2_{x}} +  t^{-p_1}\left\|f\right\|_{H^{s}_{x}}\right\}
    \end{equation}
    and $p_0= \left(\frac{1}{6}\right)^{-}$, $p_1 = 10^{-3}$. Observe that corollary \ref{cor:ultimate_disp_est} implies any function in $X_{T}$ undergoes linear dispersive decay at leading order.  
\par Before proving any bounds on a solution $u(t,x)$ we should first show that, when our initial conditions obeys \eqref{eqn:smallness_condition_on_ICs}, there exists a unique solution to \eqref{eqn:3bbm_cauchy_problem_asympt} that also remains of size $\mathcal{O}\left(\eps_0\right)$ for $T-1$ small enough. 
\begin{prop}\label{prop:local_theory}
Suppose $u_0(x)$ satisfies \eqref{eqn:smallness_condition_on_ICs}. There exists a strictly positive $T>1$ such that the Cauchy problem \eqref{eqn:3bbm_cauchy_problem_asympt} admits a unique solution $u(t,x)\in X_{T}$. Further, the map $\tau\in \left[1,T\right]\mapsto \left\|u\right\|_{X_\tau} \in \left[0,\infty\right)$ is continuous. 
\end{prop}
\begin{proof}
    The proof is a routine fixed-point argument. 
\end{proof}
\par Now, we go from local to global. All of our global theory relies on the following bootstrapped estimate, the proof of which makes up section \ref{s:the_technical_proof}: 
\begin{prop}
    \label{prop:bootstrap_close}
    Suppose $\eps_1\in \left(\eps_0, 1\right)$. Additionally, suppose that $\left\|u\right\|_{X_T} \leq \eps_1$. Then, there is a constant $C>0$ such that 
    \begin{equation}\label{eqn:bootstrap_close}
      \left\|u\right\|_{X_{T}} \leq \eps_0 + C\eps_1^4. 
 \end{equation}
\end{prop}
\noindent Using the above, global well-posedness and decay follow easily: 
\begin{thm}\label{thm:big_result}
    If $\eps_0\in (0,1)$ is sufficiently small and $\left\|u_0\right\|_{X_1}<\eps_0$, there exists a unique global-in-time solution $u(t,x)$ to \eqref{eqn:3bbm_cauchy_problem_asympt} satisfying $\left\|u\right\|_{X_{\infty}} \lesssim \eps_0$. In particular, $u(t,x)$ obeys
    \begin{equation}\label{eqn:pointwise_decay_NL}
    \left\|u(t,x)\right\|_{L^{\infty}_{x}}\lesssim \eps_0t^{-\frac13} \quad \forall \ t>1. 
    \end{equation}
\end{thm}
\begin{proof}
Let $C$ be the constant from proposition \ref{prop:bootstrap_close}. Pick $\eps_0$ such that 
\begin{equation}
    \label{eqn:smallness_of_eps0}
    \eps_0 < \left(\frac{3}{32C}\right)^{\frac13}. 
\end{equation}
Now, using proposition \ref{prop:local_theory}, we can find $T_1\in(1,2)$ and a local-in-time solution $u(t,x)$ valid on the time interval $[1,T_1]$. By shrinking $T_1$ if necessary, we can be sure that $\left\|u\right\|_{X_{T_1}} \leq 2\eps_0$ using the continuity of $\tau\mapsto \left\|u\right\|_{X_\tau}$ guaranteed by proposition \ref{prop:local_theory}. Using proposition \ref{prop:bootstrap_close} and \eqref{eqn:smallness_of_eps0}, this means that we actually have the stronger bound $\left\|u\right\|_{X_{T_1}} \leq \frac32\eps_0$. By applying proposition \ref{prop:local_theory} again with initial state $u\left(T_1\right)$, we can find $T_2>T_1$ such that $u(t,x)$ is still a solution on $[1,T_2]$. Using continuity as above, we have $\left\|u\right\|_{X_{T_{2}}}\leq 2\eps_0$. Iterating this argument and using the bootstrap principle gives the claim. To show that the global solution satisfies \eqref{eqn:pointwise_decay_NL}, we use corollary \ref{cor:ultimate_disp_est} and $\left\|u\right\|_{X_{\infty}}\lesssim \eps_0$ to obtain 
\begin{align*}
    \left\|u\right\|_{L^{\infty}_{x}} \lesssim \eps_0 t^{-\frac13} + \eps_0 t^{-\frac12+p_0} + \eps_0 t^{-\frac12 +p_1}. 
\end{align*}
Since $p_0, p_1<\frac16$, this gives \eqref{eqn:pointwise_decay_NL}. 
\end{proof}
\noindent Scattering is established using the same tools applied to get proposition \ref{prop:bootstrap_close}; a complete proof is postponed to section \ref{s:scattering_proof}. 
\begin{prop}\label{prop:scattering_in_L^2}
Consider the same notational setup as theorem \ref{thm:big_result}. There exists a unique $F(x)\in L^2_{x}$ with $\widehat{F}\in L^2_{\xi} \cap L^{\infty}_{\xi}$ such that the profile $f(t,x)$ of $u(t,x)$ satisfies
$$
\lim_{t\rightarrow\infty}\left\|f(t,x)-F(x)\right\|_{L^2_x} =\lim_{t\rightarrow\infty}\left\|\widehat{f}(t,\xi)-\widehat{F}(\xi)\right\|_{L^\infty_\xi} =0 .
$$
In particular, $u(t,x)$ scatters in $L^2_{x}$: 
$$
\lim_{t\rightarrow\infty}\left\|u(t,x)- e^{-it\omega\left(\frac{1}{i}\partial_{x}\right)}F(x)\right\|_{L^2_x} =0.
$$
\end{prop}
\noindent To prove theorem \ref{thm:main_thm}, simply combine theorem \ref{thm:big_result} with proposition \ref{prop:scattering_in_L^2}. 
\section{Proof of Proposition \ref{prop:bootstrap_close}}
\label{s:the_technical_proof}
\subsection{Part 0: Setup}
\noindent Suppose we already have chosen some smallness threshold $\eps_0$. Pick $\eps_1\in \left(\eps_0, 1\right)$. For the remainder of this chapter, we assume \textit{a priori} that the weak bound 
\begin{equation}\label{eqn:bootstrap_open}
      \left\|u\right\|_{X_{T}} \leq \eps_1
 \end{equation}
 holds. Theorem \ref{thm:LP_disp_est} and \eqref{eqn:bootstrap_open} together imply 
\begin{equation}
    \label{eqn:bootstrap_on_LP_pieces}
    \left\| u_{k}(t,x)\right\|_{L^{\infty}_{x}} \lesssim \eps_1
    \begin{cases}
    t^{p_1}2^{-k(s-1)} \quad & \text{if} \quad 2^{k} \geq C_{\mathrm{hi}}t^{\frac19}
    \\ 
    t^{-\frac12}2^{\frac{3k}{2}} \quad & \text{if} \quad 2^{3} \leq 2^{k} < C_{\mathrm{hi}}t^{\frac19}
    \\ 
    t^{-\frac13} \quad & \text{if} \quad 2^{-1} \leq 2^{k} < 2^{3}
    \\ 
      t^{-\frac12}2^{-\frac{k}{2}}  \quad & \text{if} \quad C_{\mathrm{lo}}t^{-\frac13} \leq 2^{k} < 2^{-1}
      \\ 
     2^{k} \quad & \text{if} \quad 2^{k}< C_{\mathrm{lo}}t^{-\frac13}. 
    \end{cases}
\end{equation}
Additionally, we know that the worst-case dispersive decay bound holds:
\begin{equation}
\label{eqn:bootstrap_amplitude}
\left\|u\right\|_{L^{\infty}_{x}} \lesssim \eps_1 t^{-\frac13}. 
\end{equation}

\subsection{Part I: $H^s_{x}$ Bounds}
\noindent We need to show that \eqref{eqn:bootstrap_open} implies
\begin{equation}
\label{eqn:Sobolev_bootstrap_close}
\left\|f\right\|_{H^{s}_{x}} \lesssim \eps_0+\eps_1^4 t^{p_1} \quad \forall \ t\in [1,T]
\end{equation}
where $p_1=10^{-3}$ (in the course of our argument, we'll actually find that any $p_1>0$ works). To begin, we use \eqref{eqn:fhat_integrated_framework} and smallness of the Cauchy data to write
\begin{align}
\left\|f\right\|_{H^s_{x}}
&\lesssim \eps_0 + \int_{1}^{t} \ \diff \tau \left\|  \langle\xi\rangle^{s} \int \diff \eta_1 \diff \eta_2 \diff \eta_3 \ e^{-i\tau\varphi} \ \widehat{f}\left(\eta_1\right) \ \widehat{f}\left(\eta_2\right) \ \widehat{f}\left(\eta_3\right)\ \widehat{f}\left(\eta_4\right)\right\|_{L^2_{\xi}} \label{eqn:sobolev_bnd_step1}
\end{align}
(remember that $\eta_4$ is defined by \eqref{eqn:eta4_defn}). We now ``untangle'' the convolutions over $\eta_1, \eta_2, \eta_3$ and then apply Plancherel's theorem to bring us back over $x$-space: 
\begin{align*}
\left\|f\right\|_{H^s_{x}} \lesssim \eps_0 + \int_{1}^{t} \diff \tau \left\|u^4\right\|_{H^s_{x}} \lesssim \eps_0 + \int_{1}^{t} \diff \tau \left\|u\right\|_{L^\infty_{x}}^{3} \ \left\|u\right\|_{H^s_{x}}.
\end{align*}
Using \eqref{eqn:bootstrap_open} and \eqref{eqn:bootstrap_amplitude}, we secure plenty of time decay inside the $\tau$-integral and finish up the proof:
$$
\left\|f\right\|_{H^s_{x}} \lesssim \eps_0 + \eps_1^4\int_{1}^{t} \diff \tau \ \tau^{-1+p_{1}} \lesssim \eps_0+ \eps_1^4 t^{p_{1}}.
$$

\subsection{Part II: Determination of Space-time Resonances}
A detailed computation of all space-time resonances in presented in appendix \ref{section:res_comp}. Here, I report the final results of these calculations. If we define $r(\xi)$ by 
\begin{equation}\label{eqn:refl_denf}
r(\xi) = \sgn\left(\xi\right)\sqrt{\frac{\xi^2+3}{\xi^2-1}}
\end{equation}
then up to permuting and negating variables, the only space-time resonances for $p=3$ gBBM are
\begin{itemize}
    \item on the line $L$ defined by
    \begin{equation}
        L = \left\{\left(\eta_1,\eta_2,\eta_3;\xi\right) = \left(-\eta,\eta,\eta; 0\right) \ | \ \eta\in \mathbb{R} \right\}\subseteq \mathbb{R}^4,
        \label{eqn:resonant_line_discussion}
    \end{equation}
    including the points $(0,0,0;0)$ and $\left(-\sqrt{3},\sqrt{3},\sqrt{3};0\right)$; 
    \item on the curve $\Gamma$ defined by 
      \begin{equation}
    \Gamma =  \left\{\left(\eta_1,\eta_2,\eta_3;\xi\right) = \left(\eta,r(\eta),-r(\eta); 0\right) \ | \ \left|\eta\right|>1 \right\}\subseteq \mathbb{R}^4;
      \label{eqn:resonant_curve_discussion}
    \end{equation}
    \item at the point $\left(\eta_0, \eta_0, \eta_0;\xi_0\right)$ where $\eta_0\approx 5.1$ and $\xi_0\approx 14.2$ are the unique positive solutions to the system 
\begin{subequations}
\label{eqn:fam_1_badpoints}
\begin{align}
3\eta_{1} - r\left(\eta_{1}\right) &=\xi,
\\
3\omega(\eta_{1}) - \omega(r\left(\eta_{1}\right))-\omega(3\eta_{1}-r\left(\eta_{1}\right))&=0.
 \end{align}
 \end{subequations}
\end{itemize} 
Since the isolated resonance $\left(\eta_0, \eta_0, \eta_0;\xi_0\right)$ involves numbers that do not appear in the linear theory, I refer to it as the \textbf{anomalous resonance} in the sequel. 

 \subsection{Part III: Weighted $L^2_x$ and $L^{\infty}_{\xi}$ Bounds Near Resonances}
 \noindent To estimate weighted norms of our solution, we start by differentiating both sides of \eqref{eqn:ODE_for_fhat_framework} with respect to $\xi$:
\begin{align*}
\partial_{\xi} \widehat{f}(t,\xi) &= \partial_{\xi} \widehat{f}(1,\xi) 
\\
&\phantom{=} -i\left(2\pi\right)^{-\frac32}\omega(\xi)
    \int_{1}^{t} \diff s \int\diff\eta_{1}\ \diff\eta_{2}\ \diff\eta_{3}\ e^{-i\tau\varphi} \left\{(-i\tau) \ \partial_{\xi}\varphi \ \hat{f}(\eta_{1})\ \hat{f}(\eta_{2})\ \hat{f}(\eta_{3})  \ \hat{f}\left(\eta_{4} \right)\right.
    \\
    &\phantom{=} + \left.
    \hat{f}(\eta_{1})\ \hat{f}(\eta_{2})\ \hat{f}(\eta_{3})  \ \partial_{\xi}\hat{f}\left(\eta_{4}\right)
    \right\}
    \\
    &\phantom{=} -i\left(2\pi\right)^{-\frac32}\omega'(\xi)
    \int_{1}^{t} \diff s \int\diff\eta_{1}\ \diff\eta_{2}\ \diff\eta_{3}\ e^{-i\tau\varphi} \hat{f}(\eta_{1})\ \hat{f}(\eta_{2})\ \hat{f}(\eta_{3})  \ \hat{f}\left(\eta_{4} \right). 
\end{align*}
Define 
\begin{subequations}
    \begin{align}
    I_{\text{easy}} &\doteq  \omega(\xi)
    \int_{1}^{t} \diff \tau \int\diff\eta_{1}\ \diff\eta_{2}\ \diff\eta_{3}\ e^{-i\tau\varphi}  \hat{f}(\eta_{1})\ \hat{f}(\eta_{2})\ \hat{f}(\eta_{3})  \ \partial_{\xi}\hat{f}\left(\eta_4\right)
    \\
     I_{\text{easier}} &\doteq  \omega'(\xi)
    \int_{1}^{t} \diff \tau \int\diff\eta_{1}\ \diff\eta_{2}\ \diff\eta_{3}\ e^{-i\tau\varphi}  \hat{f}(\eta_{1})\ \hat{f}(\eta_{2})\ \hat{f}(\eta_{3})  \ \hat{f}\left(\eta_{4}\right)
    \\
    I &\doteq \omega(\xi)
    \int_{1}^{t}  \diff \tau \ \tau \int\diff\eta_{1}\ \diff\eta_{2}\ \diff\eta_{3}\ e^{-i\tau\varphi} \partial_{\xi}\varphi \ \hat{f}(\eta_{1})\ \hat{f}(\eta_{2})\ \hat{f}(\eta_{3})  \ \hat{f}\left(\eta_4 \right)
    \end{align}
\end{subequations}
so that 
\begin{equation}
\label{eqn:basic_expression_for_d_xi_fhat}
\partial_{\xi} \widehat{f}(t,\xi) = \partial_{\xi} \widehat{f}(1,\xi) - \left(2\pi\right)^{-\frac32} \left[iI_{\text{easy}} +i I_{\text{easier}}+ I\right].
\end{equation}
In particular, we have 
\begin{equation}\label{eqn:basic_weighted_bound} 
\left\|xf(t,x)\right\|_{L^2_x} \lesssim \eps_0 + \left\|I_{\text{easy}}(t,\xi)\right\|_{L^2_{\xi}} + \left\|I_{\text{easier}}(t,\xi)\right\|_{L^2_{\xi}} + \left\|I(t,\xi)\right\|_{L^2_{\xi}}
\end{equation}
As the notation suggests, bounding $I_{\text{easy}}$ is very simple. We follow the convolution-untangling procedure that was used to estimate $\left\|f\right\|_{H^s_{x}}$, giving
\begin{align*}
\left\|I_{\text{easy}}(t,\xi)\right\|_{L^2_{\xi}} &\lesssim \int_{1}^{t} \diff \tau \left\|u^3 \ \exp\left(-it\omega\left(\frac{1}{i}\partial_{x}\right)\right)\left(xf\right)\left(\tau, x\right)\right\|_{L^2_x}.
\end{align*}
Combining the above with \eqref{eqn:bootstrap_amplitude} and \eqref{eqn:bootstrap_open} implies 
\begin{align*}
\left\|I_{\text{easy}}(t,\xi)\right\|_{L^2_{\xi}} &\lesssim \eps_1^4\int_{1}^{t} \diff \tau \ \tau^{-1+p_0} \lesssim \eps_1^4 t^{p_0}. 
\end{align*}
The same ideas give $\left\|I_{\text{easier}}(t,\xi)\right\|_{L^2_{\xi}} \lesssim \eps_1^4 t^{p_0}$. Together with \eqref{eqn:basic_weighted_bound}, these estimates imply
\begin{equation}\label{eqn:pupal_weighted_bound} 
\left\|xf(t,x)\right\|_{L^2_x} \lesssim \eps_0 +\eps_1^4 t^{p_0} + \left\|I(t,\xi)\right\|_{L^2_{\xi}}. 
\end{equation}
It then remains to prove 
\begin{equation}
    \label{eqn:bootstrap_weighted_hard_claim}
    \left\|I(t,\xi)\right\|_{L^2_{\xi}} \lesssim \eps_1^4 t^{p_0} . 
\end{equation}
Owing to the extra $\tau$ appearing in the integrand of $I$, the arguments are much more difficult than those used to control $I_{\text{easy}}$ and $I_{\text{easier}}$. This is where LP techniques come to the rescue, and where additional constraints on the growth rate $p_0$ appear. Additionally, to finish obtaining the desired control on $\left\|u\right\|_{X_T}$,  we need to bound $\left\|\widehat{f}\right\|_{L^{\infty}_{\xi}}$. Using \eqref{eqn:ODE_for_fhat_framework}, this reduces to showing that the integral \begin{equation}
    \label{eqn:J_defn}
    J(\xi) = \omega(\xi)\int_{1}^{t} \diff \tau \int \diff\eta_1\diff\eta_2\diff\eta_3 \ e^{-i\tau\varphi} \  \widehat{f}\left(\eta_1\right) \ \widehat{f}\left(\eta_2\right) \ \widehat{f}\left(\eta_3\right) \ \widehat{f}\left(\eta_{4}\right). 
\end{equation}
obeys 
\begin{equation}
    \label{eqn:Linfty_bnd_wts}   \left\|J\right\|_{L^{\infty}_{\xi}}\lesssim \eps_1^4
\end{equation}
under the assumption \eqref{eqn:bootstrap_open}.
\par We begin by performing a dyadic decomposition of $[1,t]$: if $m_{*}$ is the smallest integer satisfying $2^{m_{*}}\geq t$, then we write
$$
[1,t] = \bigcup_{m=0}^{m_{*}} \left\{\tau \sim 2^{m}\right\}\cap [1,t]. 
$$
Note that each dyadic piece has length $\sim 2^{m}$, and there are a total of $\mathcal{O}\left(\log t\right)$ dyadic pieces. Then, we fix some $m=0,...,m_{*}$ and define $t_1, t_2$ by $[t_1, t_2] = \left\{\tau \sim 2^{m}\right\}\cap [1,t]$ and
\begin{subequations}
\begin{align}
\label{eqn:I_defn}
I_{m} &\doteq\omega(\xi)
    \int_{t_1}^{t_2}  \diff \tau \ \tau \int\diff\eta_{1}\ \diff\eta_{2}\ \diff\eta_{3}\ e^{-i\tau\varphi}\  \partial_{\xi}\varphi \ \hat{f}(\eta_{1})\ \hat{f}(\eta_{2})\ \hat{f}(\eta_{3})  \ \hat{f}\left(\eta_{4} \right)
    \\ \label{eqn:Jdefn}
    J_{m}(t) &\doteq \omega(\xi)\int_{t_1}^{t_2} \diff \tau \int \diff\eta_1\diff\eta_2\diff\eta_3 \ e^{-i\tau\varphi} \  \widehat{f}\left(\eta_1\right) \ \widehat{f}\left(\eta_2\right) \ \widehat{f}\left(\eta_3\right) \ \widehat{f}\left(\eta_{4}\right).
\end{align}
\end{subequations}
We want to prove that, under the bootstrap hypothesis \eqref{eqn:bootstrap_open}, we have
\begin{subequations}
    \begin{align}
        \label{eqn:claimed_bound}
\sum_{m=0}^{m_{*}}\left\|I_{m} \right\|_{L^2_{\xi}} &\lesssim \eps_1^4 t^{p_0}
\\
  \label{eqn:claimed_bound_Linfty}
\sum_{m=0}^{m_{*}}\left\|J_{m} \right\|_{L^2_{\xi}} &\lesssim \eps_1^4 . 
    \end{align}
\end{subequations}
This is accomplished by localizing the integrands of $I_{m}$ and $J_m$ in $\eta_1\eta_2\eta_3\xi$-space.   
\subsubsection{Anomalous Resonance}
\label{ss:anomalous}
\noindent We begin by working with the anomalous space-time resonance $(\eta_{1}, \eta_{2}, \eta_{3}; \xi)\ = \left(\eta_{0}, \eta_{0}, \eta_{0}; \xi_{0}\right)$ defined in \eqref{eqn:fam_1_badpoints}. Throughout this proof, $\alpha>0$ denotes a small number that will be constrained over the course of our discussion.  For each fixed $m$, let $k_{*}$ be the largest integer such that $2^{k_{*}} \leq 2^{-20}2^{m\left(\alpha-\frac12\right)}$. We perform an LP decomposition of the integrands of $I_{m}, J_m$ with respect to $\xi-\xi_0$ and each $\eta_j-\eta_0$.  
\par \noindent 
\textbf{Case 1}
\par \noindent
First, we use vanilla bumps in $\eta_1, \eta_2, \eta_3, \xi$ to localize: 
\begin{align*}
\left|\eta_{j}-\eta_0\right| \leq 2^{k_{*}} \quad \forall \ j=1,2,3, \quad \left|\xi-\xi_0\right| \leq 2^{3}2^{2k_{*}}. 
\end{align*}
That is, we are bounding 
\begin{equation}
    \hspace{-1cm} I_{mk_{*}} \doteq \omega(\xi)\int_{t_1}^{t_2}\diff\tau \ \tau \phi_{2k_{*}+3}\left(\xi-\xi_0\right) \int\diff\eta_1  \diff \eta_{2}  \diff \eta_{3}  \ e^{-i\varphi \tau} \partial_{\xi}\varphi \widehat{f}_{\lesssim 2^{k_{*}}} \left(\eta_{1}\right) \widehat{f}_{\lesssim 2^{k_{*}}} \left(\eta_{2}\right) \widehat{f}_{\lesssim 2^{k_{*}}} \left(\eta_{3}\right) 
    \widehat{f}\left(\eta_{4}\right). 
\end{equation}
Our only option is to take advantage of the size of integration domain in Fourier space: I call this the ``volume trick''. Since $\omega$ and $\partial_{\xi}\varphi$ are bounded, such an approach gives 
\begin{align*}
\left\|I_{k*} \right\|_{L^2_{\xi}} 
&\lesssim 
\int_{t_1}^{t_2}\diff\tau \ \tau \left\| \phi_{2k_{*}+3}\left(\xi-\xi_0\right)\int\diff\eta_1  \diff \eta_{2}  \diff \eta_{3} \ e^{-i\varphi \tau} \partial_{\xi}\varphi \ \widehat{f}_{\lesssim 2^{k_{*}}} \left(\eta_{1}\right) \widehat{f}_{\lesssim 2^{k_{*}}} \left(\eta_{2}\right) \widehat{f}_{\lesssim 2^{k_{*}}} \left(\eta_{3}\right) 
    \widehat{f}\left(\eta_{4}\right) \right\|_{L^2_{\xi}}
\\
&\lesssim \eps_1^4 \int_{t_1}^{t_2} \diff\tau \ \tau \ 2^{3k_{*}}\  \sqrt{\text{vol}\left(|\xi-\xi_0|\lesssim 2^{2mk_{*}}\right)} \lesssim \eps_1^4 2^{4\alpha m}. 
\end{align*}
Summing over $m$ yields
\begin{align*}
    \sum_{m=0}^{m_{*}}\left\|I_{k*} \right\|_{L^2_{\xi}} \lesssim \eps_1^4 t^{4\alpha}.
\end{align*}
Then, we must have $4\alpha \leq p_0$. The same ideas applied to the $L^{\infty}_{\xi}$ piece give
$$
\sum_{m}\left\|J_{mk_{*}}\right\|_{L^{\infty}_{\xi}} \lesssim \eps_1^4 t^{-\frac12 +3\alpha}. 
$$
Since $4\alpha \leq p_0<\frac16$, the above is $\lesssim \eps_1^4$ as desired.  
\par \noindent 
\textbf{Case 2}
\par \noindent
Now, localize using annular bumps about $\xi_0$ and vanilla bumps for each $\eta_j-\eta_0$: 
\begin{align*}
\left|\eta_{j}-\eta_0\right| \leq 2^{k_{*}} \quad \forall \ j=1,2,3 , \quad \left|\xi-\xi_0\right| \sim 2^{k}, \quad \text{where} \quad 
2^{3}2^{2k_{*}} \leq 2^{k} \leq 2^{-10}.
\end{align*}
For such $\eta_j$ and $\xi$, we use Taylor expansion to discover
\begin{align}
\varphi &= \left[\omega'(r(\eta_0))-\omega'(\xi_0)\right]\eta_4-\frac12\omega''(\xi_{*})\left(\xi-\xi_0\right)^2 \nonumber
\\
&\phantom{=} + 
       \frac12 \omega''(\eta_{**})\left(\eta_4+r(\eta_0)\right)^2 + \frac12\sum_{j=1}^{3}\omega''(\eta_{*,j})\left(\eta_j-\eta_0\right)^2,
       \label{eqn:case2:taylor_expansion_phi}
\end{align}
where $\xi_*\approx \xi_0$, $\eta_{*,j}$ near $\eta_0$ for $j=1,2,3$, and $\eta_{**}\approx -r(\eta_0)$. This implies (perhaps up to massaging the constants used to define this case)
\begin{equation}\label{eqn:case2_phase_bnd_below}
|\varphi| \gtrsim 2^{k},
\end{equation}
meaning we can integrate by parts in $\tau$. Denote the part of $I_{m}$ we're bounding by 
\begin{equation}
\hspace{-0.5cm}
    I_{mk} \doteq \omega(\xi) \int_{t_1}^{t_2}\diff\tau \ \tau \psi_{k}\left(\xi-\xi_0\right) \int\diff\eta_1  \diff \eta_{2}  \diff \eta_{3} \ e^{-i\varphi \tau} \partial_{\xi}\varphi \widehat{f}_{\lesssim 2^{k_{*}}} \left(\eta_{1}\right) \widehat{f}_{\lesssim 2^{k_{*}}} \left(\eta_{2}\right) \widehat{f}_{\lesssim 2^{k_{*}}} \left(\eta_{3}\right) 
    \widehat{f}\left(\eta_{4}\right). 
\end{equation}
Using \eqref{eqn:case2_phase_bnd_below}, we are clear to integrate by parts in $\tau$ within $I_{mk}$, yielding 
\begin{align*}
    I_{mk} &= \omega(\xi) \psi_{k}(\xi-\xi_0) \int_{t_1}^{t_2} \diff \tau \int\diff \eta_{1} \diff \eta_{2} \diff \eta_{3}  \ e^{-i\tau\varphi} \frac{\partial_{\xi}\varphi}{\varphi}  \tau  \partial_{\tau}\left(  \widehat{f}_{\lesssim 2^{k_{*}}} \left(\eta_{1}\right) \widehat{f}_{\lesssim 2^{k_{*}}} \left(\eta_{2}\right) \widehat{f}_{\lesssim 2^{k_{*}}} \left(\eta_{3}\right) 
    \widehat{f}\left(\eta_{4}\right)  \right) 
    \\  
    &\phantom{=} + e^{-i\tau\varphi} \ \frac{\partial_{\xi}\varphi}{\varphi} \  \widehat{f}_{\lesssim 2^{k_{*}}} \left(\eta_{1}\right) \widehat{f}_{\lesssim 2^{k_{*}}} \left(\eta_{2}\right) \widehat{f}_{\lesssim 2^{k_{*}}} \left(\eta_{3}\right) 
    \widehat{f}\left(\eta_{4}\right) 
    \\
    &\phantom{=} + \left[\omega(\xi) \psi_{k}(\xi-\xi_0) \ \tau \int\diff \eta_{1} \diff \eta_{2}  \diff \eta_{3}  \ e^{-i\tau\varphi} \frac{\partial_{\xi}\varphi}{\varphi} \  \widehat{f}_{\lesssim 2^{k_{*}}} \left(\eta_{1}\right) \widehat{f}_{\lesssim 2^{k_{*}}} \left(\eta_{2}\right) \widehat{f}_{\lesssim 2^{k_{*}}} \left(\eta_{3}\right) 
    \widehat{f}\left(\eta_{4}\right) \right]_{\tau=t_1}^{t_2}. 
\end{align*}
Thus, we need to estimate $\left\|\partial_{\tau}\hat{f}_{\lesssim 2^{k_*}}\right\|_{L^{\infty}_{\xi}}$ (when the time derivative hits the $\eta_j$ term for $j=1,2,3$) and 
$$
\left\|\psi_{k}\left(\xi-\xi_0\right)\phi_{\lesssim 2^{k_{*}}}\left(\eta_1\right)\phi_{\lesssim 2^{k_{*}}}\left(\eta_2\right)\phi_{\lesssim 2^{k_{*}}}\left(\eta_3\right)\partial_{\tau}\hat{f}\left(\eta_4\right)\right\|_{L^{\infty}_{\xi}} 
$$
(when the time derivative hits the $\eta_4$ term) before going further. 
\begin{lemma}\label{lemma:time_derivative_bnd}
\begin{equation}\label{eqn:time_derivative_bnd}
\left\|\partial_{\tau}\widehat{f}_{\lesssim 2^{k_{*}}}\right\|_{L^{\infty}_{\xi}} \lesssim \eps_1^4 2^{-m}. 
\end{equation}
\end{lemma}
\begin{lemma}\label{lemma:time_derivative_bnd_eta4}
\begin{equation}\label{eqn:time_derivative_bnd_eta4}
\left\|\psi_{k}\left(\xi-\xi_0\right)\phi_{\lesssim 2^{k_{*}}}\left(\eta_1\right)\phi_{\lesssim 2^{k_{*}}}\left(\eta_2\right)\phi_{\lesssim 2^{k_{*}}}\left(\eta_3\right)\partial_{\tau}\hat{f}\left(\eta_4\right)\right\|_{L^{\infty}_{\xi}} \lesssim \eps_1^4 2^{-m}. 
\end{equation}
\end{lemma}
\noindent Proofs of these two lemmas are provided in subsection \ref{ss:proofs_of_lemmas}. Combining \eqref{eqn:case2_phase_bnd_below}, \eqref{eqn:time_derivative_bnd}, \eqref{eqn:time_derivative_bnd_eta4}, the volume trick, and $4\alpha \leq p_0$ from case 1, we get
$$
\sum_{\substack{1\leq 2^{m} \leq 2^{m_{*}} \\ 2^{2k_{*}} \lesssim 2^{k}\ll 1}} \left\|I_{mk}\right\|_{L^2_{\xi}} \lesssim \eps_1^4 t^{p_0}.
$$
Now, we turn to bounding $J_{m}$. Applying the same tactics we used to handle the weighted norm, we find
\begin{align*}
     \sum_{2^{2k_{*}}\lesssim 2^{k}\ll 1}\left\|J_{mk}\right\|_{L^{\infty}_{\xi}} &\lesssim \eps_1^4 2^{m\left(-\frac12+\alpha\right)}. 
     \end{align*}
Since $\alpha$ is small, summing over $m$ gives the desired result.
\par \noindent \textbf{Case 3}
\par\noindent
We now localize according to 
\begin{align*}
\left|\eta_{j}-\eta_0\right| &\sim 2^{k_{j}}, \quad 2^{k_{*}}\leq 2^{k_j}\leq 2^{-10} \quad \forall \ j=1,2,3, \quad k_1 \geq k_2\geq k_3,
\\
k_{1}&\gg k_{2} \ \text{or} \ k_{1}\gg k_{3} \ \text{or} \ k_{2}\gg k_{3},
\\
\left|\xi-\xi_0\right| &\leq 2^{3}2^{2k_{*}},  \quad \text{or} \quad \left|\xi-\xi_0\right| \sim 2^{k}, \quad 2^{2k_{*}}\lesssim 2^{k}\leq 2^{-10}.
\end{align*}
Without loss of generality assume $|\xi-\xi_0|\sim 2^{k}$: nothing major changes if $|\xi-\xi_0|$ is allowed to be small. First, we may use the mean value theorem and the non-vanishing of $\omega''(\eta)$ near $\eta_0$ to find that 
\begin{equation}\label{eqn:one_kj_dominates}
    \left|\partial_{\eta_1-\eta_2} \varphi \right| \gtrsim 2^{k_{1}}.
\end{equation}
We now integrate by parts in the direction $\partial_{\eta_1-\eta_2}$:
\begin{align*}
\hspace{-1.5cm}
I_{mk_1k_2k_3k} &\doteq \psi_{k}(\xi-\xi_0) \ \omega(\xi)
    \int_{t_1}^{t_2}  \diff \tau \int\diff\eta_{1}\ \diff\eta_{2}\ \diff\eta_{3}\ \frac{\partial_{\eta_1-\eta_2}\left(e^{-i\tau\varphi}\right)}{-i \partial_{\eta_1-\eta_2}\varphi}\  \partial_{\xi}\varphi \ \hat{f}_{k_{1}}(\eta_{1})\ \hat{f}_{k_2}(\eta_{2})\ \hat{f}_{k_3}(\eta_{3})  \ \hat{f}\left(\eta_{4} \right)
    \\
    &=  -i \psi_{k}(\xi-\xi_0) \ \omega(\xi) \int_{t_1}^{t_2}  \diff \tau \int\diff\eta_{1}\diff\eta_{2}\diff\eta_{3}\ e^{-i\tau\varphi} \frac{\partial_{\xi}\varphi}{ \partial_{\eta_1-\eta_2}\varphi} \ \partial_{\eta_{1}}\hat{f}_{k_{1}}(\eta_{1})\ \hat{f}_{k_2}(\eta_{2})\ \hat{f}_{k_3}(\eta_{3})  \ \hat{f}\left(\eta_4 \right)
      \\
     &\phantom{=}  +i \psi_{k}(\xi-\xi_0) \ \omega(\xi) \int_{t_1}^{t_2}  \diff \tau \int\diff\eta_{1}\diff\eta_{2}\diff\eta_{3}\ e^{-i\tau\varphi} \frac{\partial_{\xi}\varphi}{ \partial_{\eta_1-\eta_2}\varphi} \ \hat{f}_{k_{1}}(\eta_{1})\ \left(\partial_{\eta_{2}}\hat{f}_{k_2}(\eta_{2})\right)\ \hat{f}_{k_3}(\eta_{3})  \ \hat{f}\left(\eta_{4} \right)
     \\
      &\phantom{=} -i \psi_{k}(\xi-\xi_0) \ \omega(\xi) \int_{t_1}^{t_2}  \diff \tau \int\diff\eta_{1}\diff\eta_{2}\diff\eta_{3}\ e^{-i\tau\varphi} \partial_{\eta_1-\eta_2}\left(\frac{\partial_{\xi}\varphi}{ \partial_{\eta_1-\eta_2}\varphi} \right) \ \hat{f}_{k_{1}}(\eta_{1})\ \hat{f}_{k_2}(\eta_{2})\ \hat{f}_{k_3}(\eta_{3})  \ \hat{f}\left(\eta_{4} \right)
     \\
     &\doteq i\left(\tilde{I}_{1} + \tilde{I}_{2} + \tilde{I}_{3}\right). 
\end{align*}
Note in particular how all terms involving derivatives of $\hat{f}\left(\eta_{4} \right)$ nicely cancel one another out. Before estimating each of the $\tilde{I}_{j}$'s, we need a few preliminary bounds. First, the product rule and our bootstrap assumptions give
\begin{subequations}\label{eqn:derivative_hits_cutoff}
    \begin{align}
        \left\|\partial_{\eta_{1}}\hat{f}_{k_{1}}(\tau, \eta_{1})\right\|_{L^2_{\eta_{1}}} &\lesssim \eps_1\left(2^{mp_0} + 2^{\frac{k_1}{2}}\right),
        \\
        \left\|\left(\partial_{\eta_{2}}\hat{f}_{k_2}(\tau, \eta_{2})\right)\right\|_{L^2_{\eta_{2}}} &\lesssim \eps_1\left(2^{mp_0} + 2^{\frac{k_2}{2}}\right). 
    \end{align}
\end{subequations}
Additionally, we have $|\omega(\xi)|, |\partial_{\xi}\varphi|\lesssim 1$ in the region of interest. Finally, by \eqref{eqn:one_kj_dominates} we find 
    \begin{align}
    \label{eqn:J1_symbol_bound}
     \left|\frac{\partial_{\xi}\varphi}{ \partial_{\eta_1-\eta_2}\varphi} \right| \lesssim 2^{-k_{1}}, \quad \left|\partial_{\eta_1-\eta_2}\left(\frac{\partial_{\xi}\varphi}{ \partial_{\eta_1-\eta_2}\varphi} \right) \right| \lesssim 2^{-2k_{1}}. 
    \end{align}
Now, we control $\tilde{I}_{1}, \tilde{I}_{2}$, and $\tilde{I}_{3}$ using our $L^2_{\xi}$-multilinear estimate. Owing to \eqref{eqn:derivative_hits_cutoff}, $\tilde{I}_{1}$ and $\tilde{I}_{2}$ are basically identical, so we only show details for $\tilde{I}_{1}$. 
\par First, we handle $\tilde{I}_{1}$. Using \eqref{eqn:J1_symbol_bound}, we can easily verify that the hypothesis of lemma
\ref{lemma:multiplier_cheat_code} is satisfied by the symbol 
\begin{equation}\label{eqn:anomalous_case3_symbol}
M = \frac{\partial_{\xi}\varphi}{\partial_{\eta_1-\eta_2}\varphi} \ \psi_{k}\left(\xi-\xi_0\right) \ \psi_{k_1}\left(\eta_1-\eta_0\right)\ \psi_{k_2}\left(\eta_2-\eta_0\right)\ \psi_{k_3}\left(\eta_3-\eta_0\right),
\end{equation}
with $A=2^{-k_{1}}$:  the $\partial_{\xi}$ derivatives can be controlled trivially, and each partial directional differentiation in the $\eta_j$'s introduces a factor of at most $2^{-k_{1}}$, which can be bounded above by $2^{-k_{2}}$ or $2^{-k_{3}}$ as required. Then, we may apply proposition  \ref{prop:multilinear_estimates}
 with $M$ given by \eqref{eqn:anomalous_case3_symbol} to discover
\begin{align*}
\hspace{-1cm}
    \left\|\tilde{I}_{1}\right\|_{L^2_{\xi}} &\lesssim \left\|\psi_{k}(\xi-\xi_0)\omega(\xi)\right\|_{L^{\infty}_{\xi}} \int_{t_1}^{t_2} \diff \tau \left\|\int\diff\eta_1\diff\eta_2\diff\eta_3 \ 
    e^{-i\tau\varphi}\frac{\partial_{\xi}\varphi}{\partial_{\eta_1-\eta_2}\varphi} \partial_{\eta_{1}}\widehat{f}_{k_{1}}(\eta_1) \widehat{f}_{k_2}(\eta_{2})  \widehat{f}_{k_3}(\eta_{3})  \widehat{f}\left(\eta_{4}\right)
    \right\|_{L^2_{\xi}}
    \\
    &\lesssim \eps_1^{4}\left(2^{-k_{1}}2^{m\left(p_0 -\frac12\right)} + 2^{-\frac{k_{1}}{2}}2^{-\frac{m}{2}}\right).
\end{align*}
Note how we have used \eqref{eqn:derivative_hits_cutoff} and the presumed localization away from any degenerate frequencies, which we recall secures $t^{-\frac12}$ decay. We put the $u_{k_{1}}$ term in $L^2_{x}$ all the other $u$'s into $L^{\infty}_{x}$: the $\eta_4$ term enjoys $t^{-\frac12}$ decay owing to the implicit localization of $\eta_4$ near $-r(\eta_0)$. 
\par Next is $\tilde{I}_{3}$, which requires a bit more care. Combining \eqref{eqn:J1_symbol_bound} with a little bit of computation shows that the symbol
$$
 M = \partial_{\eta_1-\eta_2}\left(\frac{\partial_{\xi}\varphi}{ \partial_{\eta_1-\eta_2}\varphi} \right)  \ \psi_{k}\left(\xi-\xi_0\right) \ \psi_{k_1}\left(\eta_1-\eta_0\right)\ \psi_{k_2}\left(\eta_2-\eta_0\right)\ \psi_{k_3}\left(\eta_3-\eta_0\right),
 $$
verifies the hypothesis of lemma \ref{lemma:multiplier_cheat_code} with $A=2^{-2k_1}$.  Na\"{i}vely, we would combine this insight the $L^2_{\xi}$ multilinear estimate as above to try and get the bound we want. However, this turns out to give a growth rate of $\left(\frac14\right)^{-}$ when we want a smaller growth rate of $\left(\frac16\right)^{-}$. To improve our bound, we must integrate by parts in $\partial_{\eta_1-\eta_2}$ again: this is acceptable since the integrand of $\tilde{I}_{3}$ contains no derivatives of $\widehat{f}$, so we won't pick up norms of $f$ for which we have no bootstrap estimates. This gives 
\begin{align*}
\hspace{-2cm}
    \widetilde{I}_{3} &=   -i\psi_{k}(\xi-\xi_0) \ \omega(\xi) \int_{t_1}^{t_2}  \diff \tau   \ \tau^{-1}\int\diff\eta_{1}\diff\eta_{2}\diff\eta_{3}
    \\
    &\phantom{=} \frac{e^{-i\tau\varphi}}{\partial_{\eta_1-\eta_2}\varphi} \partial_{\eta_1-\eta_2}\left(\frac{\partial_{\xi}\varphi}{ \partial_{\eta_1-\eta_2}\varphi} \right) \ \partial_{\eta_1}\hat{f}_{k_{1}}(\eta_{1})\ \hat{f}_{k_2}(\eta_{2})\ \hat{f}_{k_3}(\eta_{3})  \ \hat{f}\left(\eta_4 \right)
    \\
     &\phantom{=} - \frac{e^{-i\tau\varphi}}{\partial_{\eta_1-\eta_2}\varphi} \partial_{\eta_1-\eta_2}\left(\frac{\partial_{\xi}\varphi}{ \partial_{\eta_1-\eta_2}\varphi} \right) \ \hat{f}_{k_{1}}(\eta_{1})\ \partial_{\eta_2}\hat{f}_{k_2}(\eta_{2})\ \hat{f}_{k_3}(\eta_{3})  \ \hat{f}\left(\eta_4\right)
     \\
    &\phantom{=} +e^{-i\tau\varphi} \partial_{\eta_1-\eta_2}\left(\frac{1}{\partial_{\eta_1-\eta_2}\varphi} \partial_{\eta_1-\eta_2}\left(\frac{\partial_{\xi}\varphi}{ \partial_{\eta_1-\eta_2}\varphi} \right) \right)\ \hat{f}_{k_{1}}(\eta_{1})\ \hat{f}_{k_2}(\eta_{2})\ \hat{f}_{k_3}(\eta_{3})  \ \hat{f}\left(\eta_4 \right)
     \\
     &= \tilde{I}_{1}'+ \tilde{I}_{2}' + \tilde{I}_{3}'. 
\end{align*}
By symmetry, it suffices to control $\tilde{I}_{1}'$ and $\tilde{I}_{3}'$ only. We quickly find that
    \begin{align}
        \label{eqn:J1_IBP}
     \left|\frac{1}{\partial_{\eta_1-\eta_2}\varphi} \partial_{\eta_1-\eta_2}\left(\frac{\partial_{\xi}\varphi}{ \partial_{\eta_1-\eta_2}\varphi} \right) \right| \lesssim 2^{-3k_1}, \quad 
      \left| \partial_{\eta_1-\eta_2}\left(\frac{1}{\partial_{\eta_1-\eta_2}\varphi} \partial_{\eta_1-\eta_2}\left(\frac{\partial_{\xi}\varphi}{ \partial_{\eta_1-\eta_2}\varphi} \right) \right) \right| \lesssim 2^{-4k_1}.
    \end{align}
These estimates allow us to control $\tilde{I}_{1}'$ exactly as we did $\tilde{I}_{1}$. So, it remains to handle $\tilde{I}_{3}'$: after applying \eqref{eqn:J1_IBP} to meet the conditions of lemma \ref{lemma:multiplier_cheat_code}, we use proposition \ref{prop:multilinear_estimates} to discover
\begin{align*}
\left\|\tilde{I}_{3}'\right\|_{L^2_{\xi}}\lesssim \eps_1^4 \left(2^{-\frac72 k_1}2^{-\frac{3m}{2}}+  2^{-k_{1}}2^{m\left(p_{0}-\frac12\right)} + 2^{-\frac{k_{1}}{2}}2^{-\frac{m}{2}}\right).
\end{align*}
Putting our bounds on $\tilde{I}_{1}, \tilde{I}_{2}$, and $\tilde{I}_{3}$ together and summing over all frequency patches (I repress the full notation to improve readability) gives
\begin{align*}
    \sum_{\text{freqs.}}
    \left\|I_{mk_1k_2k_3}\right\|_{L^2_{\xi}}
 &\lesssim  \eps_1^4 \left[2^{m\left(p_0-\frac12\alpha\right)}+1+2^{m\left(\frac14-\frac{13}{4}\alpha\right)} \right].
\end{align*}
Summing over $m$ gives 
$$
\sum_{m=0}^{m_{*}}  \sum_{\text{freqs.}} \left\|I_{mk_1k_2k_3}\right\|_{L^2_{\xi}} \lesssim \eps_1^4\left[t^{p_0}+t^{\frac14-\frac{13}{4}\alpha}\right].
$$
 For $\alpha$ large enough so that $\frac14 -\frac{13}{4}\alpha \leq p_0$, this bound gives us something workable. Additionally, notice how this argument demonstrates that we \emph{must} take $\alpha>0$ in order for the proof to work! 
\par Turning to the corresponding LP piece of $J_m$, we integrate by parts in the direction $\partial_{\eta_1-\eta_2}$ using \eqref{eqn:one_kj_dominates}, then bound the resulting oscillatory integrals using our $L^{\infty}_{\xi}$ multilinear estimate. We show details only for $|\xi-\xi_0|\sim 2^{k}$, as the remaining situation is similar. The integration by parts gives
\begin{align*}
J_{mk_1k_2k_3k}&= i \psi_{k}(\xi-\xi_0) \ \omega(\xi) \int_{t_1}^{t_2}  \diff \tau \ \tau^{-1} \int\diff\eta_{1}\diff\eta_{2}\diff\eta_{3}\ e^{-i\tau\varphi}\frac{\partial^2_{\eta_1-\eta_2}\varphi}{\left( \partial_{\eta_1-\eta_2}\varphi\right)^2} \ \hat{f}_{k_{1}}(\eta_{1})\ \hat{f}_{k_2}(\eta_{2})\ \hat{f}_{k_3}(\eta_{3})  \ \hat{f}\left(\eta_4 \right)
     \\
     &\phantom{=}  -i \psi_{k}(\xi-\xi_0) \ \omega(\xi) \int_{t_1}^{t_2}  \diff \tau \ \tau^{-1}\int\diff\eta_{1}\diff\eta_{2}\diff\eta_{3}\ e^{-i\tau\varphi} \frac{1}{ \partial_{\eta_1-\eta_2}\varphi} \ \partial_{\eta_{1}}\hat{f}_{k_{1}}(\eta_{1})\ \hat{f}_{k_2}(\eta_{2})\ \hat{f}_{k_3}(\eta_{3})  \ \hat{f}\left(\eta_4  \right)
      \\
     &\phantom{=}  +i \psi_{k}(\xi-\xi_0) \ \omega(\xi) \int_{t_1}^{t_2}  \diff \tau  \ \tau^{-1}\int\diff\eta_{1}\diff\eta_{2}\diff\eta_{3}\ e^{-i\tau\varphi} \frac{1}{ \partial_{\eta_1-\eta_2}\varphi} \ \hat{f}_{k_{1}}(\eta_{1})\ \left(\partial_{\eta_{2}}\hat{f}_{k_2}(\eta_{2})\right)\ \hat{f}_{k_3}(\eta_{3})  \ \hat{f}\left(\eta_4 \right).
\end{align*}
After verifying that each symbol satisfies the conditions of lemma \ref{lemma:multiplier_cheat_code},  \eqref{eqn:one_kj_dominates} and proposition \ref{prop:multilinear_estimates_Linfty} imply 
\begin{align*}
    \left\|J_{mk_1k_2k_3k}\right\|_{L^{\infty}_{\xi}} &\lesssim \int_{t_1}^{t_2}\diff \tau \ 2^{-2k_{1}} \tau^{-1} \left(\eps_1 2^{\frac{k_1}{2}}\right)\left(\eps_1\tau^{-\frac12}\right)^2\left(\eps_1\right) + 2^{-k_1}\tau^{-1}\left(\eps_1\tau^{p_0}\right)\left(\eps_1\tau^{-\frac12}\right)^2\left(\eps_1\right)
    \\
    &\lesssim \eps_1^4 \left[2^{-m}2^{-\frac32 k_1} + 2^{m(p_0-1)}2^{-k_1}\right]
  \\ 
  \Rightarrow
  \sum_{m} \sum_{\text{freqs.}}\left\|J_{mk_1k_2k_3k}\right\|_{L^{\infty}_{\xi}} &\lesssim  \eps_1^4 \sum_{m}\left[2^{m\left(-\frac14-\frac32\alpha\right)}+2^{m\left(p_0-\frac12-\alpha\right)}\right] \lesssim \eps_1^4,
\end{align*}
since $\alpha$ is small and $p_0< \frac16$. 
\par \noindent \textbf{Case 4}
\par \noindent 
We now localize according to 
\begin{align*}
\left|\eta_{j}-\eta_0\right| &\sim 2^{k_{j}}, \quad 2^{k_{*}}\lesssim 2^{k_{j}}\lesssim 1, \quad \forall \ j=1,2,3, \quad 2^{k_{1}} \approx 2^{k_{2}} \approx 2^{k_{3}},
\\
\left|\xi-\xi_0\right| &\lesssim 2^{2k_{*}}, \quad \text{or} \quad \left|\xi-\xi_0\right| \sim 2^{k}, \quad 2^{2k_{*}}\lesssim 2^{k} \lesssim 1. 
\end{align*}
We begin by illustrating why a secondary localization is inevitable in this case. First, none of the $|\eta_j-\eta_0|$ dominate and we cannot integrate by parts in $\partial_{\eta_{1}-\eta_{2}}$, $\partial_{\eta_{1}-\eta_{3}}$, \textit{et cetera} and copy case 3. Additionally, since we may have $|\xi-\xi_0| \approx |\eta_j-\eta_0|^2$ for some $\eta_{j}$, we cannot bound the phase below and copy case 2. As it turns out, the phase can be bounded below \emph{sometimes} after performing our secondary localization, and we'll explain this in detail later. Finally, we turn to $\partial_{\eta_j}$, $j=1,2,3$. Can $|\partial_{\eta_j}\varphi|$ vanish in this case? Let's focus on $j=1$ for concreteness: switching to $j=2,3$ gives no substantial difference. First, note that $\partial_{\eta_{1}}\varphi = \omega'(\eta_1) - \omega'(\eta_4)$, so we can bound $|\partial_{\eta_{1}}\varphi|$ below by comparing $\eta_1$ and $\eta_4$. This seems acceptable since $\eta_1\approx \eta_0\approx 5.1$ and $\eta_4\approx -r(\eta_0)\approx -1.1,$ and we are therefore tempted to just apply the mean value theorem and get a lower bound $|\partial_{\eta_{1}}\varphi|\gtrsim 1$. However, $\omega''$ vanishes between $\eta_1$ and $\eta_4$ (at $0$ and $\sqrt{3}$), so the lower bound granted by the mean value theorem does not work, and we must be a little more careful. Notice that $\partial_{\eta_{1}}\varphi$ vanishes if and only if $\eta_4 = \pm \eta_1$ or $\pm r(\eta_1)$. This suggests that we can instead compare $r(\eta_1)$ and $-\eta_4$ to guarantee that both the frequencies we're comparing live inside a proper subinterval of $(0,\sqrt{3})$. Consequently, $\omega''(\eta)$ does not vanish when 
$$
\eta \in S \doteq \left[ \min\left\{r(\eta_1), -\eta_4\right\}, \max\left\{r(\eta_1), -\eta_4\right\}\right], 
$$
meaning that a mean value theorem argument can be saved: 
\begin{align*}
    |\partial_{\eta_1}\varphi| &= |\omega'(\eta_1)- \omega'(\eta_4)|
    =|\omega'(r(\eta_1))- \omega'(-\eta_4)|
   \geq \min_{\eta\in S} |\omega''(\eta)| |r(\eta_1)+\eta_4|
    \gtrsim |r(\eta_1)+\eta_4|.
\end{align*}
Therefore, provided $|r(\eta_1)+\eta_4|>0$, integration by parts in the direction $\partial_{\eta_1}$ is permitted. In general, integration by parts in the direction $\partial_{\eta_j}$ is okay when $|r(\eta_j)+\eta_4|>0$. This suggests we perform an inhomogeneous dyadic localization about $r(\eta_j)+\eta_4=0$: 
\begin{align*}
|r(\eta_j)+\eta_4| \lesssim 2^{k_{*}}, \quad \text{or} \quad |r(\eta_j)+\eta_4| &\sim 2^{\ell_{j}}, \quad 2^{k_{*}}\lesssim 2^{\ell_j}\ll 1. 
\end{align*}
We handle these two regions separately. 
\par \noindent
\textbf{Subcase 4-1}
\par \noindent 
In this case, we assume that there is at least one $j=1,2,3$ such that
$$
|r(\eta_j)+\eta_4| \sim 2^{\ell_{j}} \quad 2^{k_{*}}\lesssim 2^{\ell_j}\ll 1. 
$$
Without loss of generality, suppose it is $j=1$. We already know that
\begin{align}
    |\partial_{\eta_1}\varphi|\gtrsim 2^{\ell_1}. 
\end{align}
Additionally, for the particular localization at hand, we have
    \begin{align}
        \left|\frac{\partial_{\xi}\varphi}{\partial_{\eta_1}\varphi}\right|\lesssim 2^{-\ell_1},
        \quad 
         \left|\partial_{\eta_1}\left(\frac{\partial_{\xi}\varphi}{\partial_{\eta_1}\varphi}\right)\right|\lesssim 2^{-2\ell_1},
    \end{align}
and 
\begin{subequations}
    \begin{align}
         \left|\frac{1}{\partial_{\eta_1}\varphi}\partial_{\eta_1}\left(\frac{\partial_{\xi}\varphi}{\partial_{\eta_1}\varphi}\right)\right|\lesssim 2^{-3\ell_1},
        \quad
        \left|\partial_{\eta_1}\left(\frac{1}{\partial_{\eta_1}\varphi}\partial_{\eta_1}\left(\frac{\partial_{\xi}\varphi}{\partial_{\eta_1}\varphi}\right)\right)\right|\lesssim 2^{-4\ell_1}. 
    \end{align}
\end{subequations}
With these estimates at hand, we simply copy the arguments of case 3: integrate by parts in $\partial_{\eta_{1}}$, then integrate by parts again in the term where $\partial_{\eta_{1}}$ did not hit any $\widehat{f}$. Since we're performing a secondary localization, however, the hypothesis of the $L^2_{\xi}$ multilinear estimate must be checked a bit more carefully. Our secondary localization was performed with respect to $\mu_{j} = r(\eta_j)+\eta_4$, which depends \emph{nonlinearly} on $\eta_1$. Accordingly, we must apply the recipe from remark \ref{remark:nonlinear_loc}. This requires bounding $\partial_{\mu_{j}}m$ for each symbol $m$ we care about; these computations can be facilitated using the inverse function theorem to show
$$
\partial_{\mu_1}= \frac{1}{r'(\eta_1)-1}\partial_{\eta_1} - \partial_{\eta_2} - \partial_{\eta_3} +\partial_{\xi}, \quad \text{similarly for} \ \partial_{\mu_2}, \partial_{\mu_3}.  
$$
Since $|r'(\eta_j)|\neq 1$ for the $\eta_j$'s we're interested in, proving the required bounds on $\left|\partial_{\mu_j}m\right|$ is straightforward. 
\par \noindent
\textbf{Subcase 4-2}
\par \noindent 
In this subcase, we assume that 
$$
|r(\eta_j)+\eta_4|\lesssim 2^{k_{*}} \quad \forall \ j=1,2,3. 
$$
Here, no spatial integration by parts is permitted, so we want to try getting a lower bound on $|\varphi|$. To this end, we first show that
\begin{equation}
    2^{k}\approx 2^{k_1}
\label{eqn:subcase4-2_k=k_1}
\end{equation}
in this subcase: in other words, after we perform all our secondary localizations, the range of possible $k$'s becomes fixed. To prove \eqref{eqn:subcase4-2_k=k_1}, start by noticing that 
\begin{align}
\xi-\xi_0 &=\left(1-r'(\eta_0)\right)\left(\eta_1-\eta_0\right) + \sum_{j=2}^3 \left(\eta_j-\eta_0\right)+\mathcal{O}\left(2^{k_*},2^{2k_1}\right). \label{eqn:KAIBA}
\end{align}
The constraints of the present subcase imply that $r(\eta_j)=-\eta_4 +\mathcal{O}\left(2^{k_{*}}\right) \ \forall \ j$, which gives $\eta_2, \eta_3 = \eta_1 + \mathcal{O}\left(2^{k_{*}}\right)$. Plugging the above into \eqref{eqn:KAIBA} yields 
\begin{align*}
\xi-\xi_0 = \left(3-r'(\eta_0)\right)\left(\eta_1-\eta_0\right) + \mathcal{O}\left(2^{k_*},2^{2k_1}\right).
\end{align*}
Since $|r'(\eta_0)|\in (0,2^{-4})$, the above implies \eqref{eqn:subcase4-2_k=k_1}. Now, \eqref{eqn:subcase4-2_k=k_1} and \eqref{eqn:case2:taylor_expansion_phi} together yield
\begin{equation}\label{eqn:sc-4.2_saviour}
|\varphi| \gtrsim 2^{k} - \mathcal{O}\left(2^{2k}\right) \gtrsim 2^{k}. 
\end{equation}
From here, the proof amounts to copying case 2. 
\par As for $J_{m}$, we still need to perform a dyadic localization in each $r(\eta_j)+\eta_4$. In the region where at least one $|r(\eta_j)+\eta_4|$ is bounded below, we may integrate by parts in $\partial_{\eta_{j}}$ and obtain the claim by mimicking the arguments of case 3. In the region where all $|r(\eta_j)+\eta_4|$'s are very small, we may use \eqref{eqn:sc-4.2_saviour} and reduce to case 2.  

\subsubsection{$(0,0,0;0)$}
\label{ss:pure_zero}
\noindent Now we turn to the contributions near the resonance $(0,0,0;0)$. The degeneracy of $\varphi$ near this point presents some difficulties, but these can be resolved using \textbf{null structure}. Specifically, while we could effectively ignore the $\omega(\xi)$ outside of our integral when dealing with the anomalous resonance, to control the contributions from $(0,0,0;0)$ we need to use that $\omega(0)=0$ and $\partial_{\xi}\varphi(0,0,0;0)=0$. As with the anomalous resonance, we localize about the resonance at a suitable scale. Throughout this proof, $\beta>0$ denotes a small number to be constrained as we go along.  For each fixed $t$, let $k_{*}$ be the largest integer such that $2^{k_{*}} \leq 2^{-20}2^{m\left(\beta-\frac13\right)}$. Additionally, we need to use the Taylor expansions
\begin{subequations}\label{eqn:meta_null}
\begin{align}
    \omega(\xi) &= \omega'(\xi_{*}) \xi \quad \forall \ |\xi|\ll 1,
    \\
    \partial_{\xi}\varphi &= -\omega^{(3)}\left(\xi_{**}\right)\xi^2 + \omega^{(3)}\left(\xi_{***}\right) \eta_4^2, \quad \forall \ |\xi|\ll 1. 
\end{align}
\end{subequations}

\par \noindent 
\textbf{Case 1}
\par \noindent
First we localize with bumps very close to the resonance:  
\begin{align*}
\left|\eta_{j}\right| \leq 2^{k_{*}} \quad \forall \ j=1,2,3, \quad \left|\xi\right|\leq 2^{k_{*}} . 
\end{align*}
The relevant LP piece of the integral is then controlled using \eqref{eqn:meta_null} and the volume trick, and upon summing over $m$ we find 
\begin{align*}
    \sum_{m=0}^{m_{*}}  \left\|I_{mk_{*}}\right\|_{L^2_{\xi}} \lesssim  \eps_1^4 t^{\left(\frac{13}{2}\beta-\frac{1}{6}\right)}. 
\end{align*}
If we demand $\beta\leq \frac{1}{39}$ (say, $\beta=\frac{1}{100}$), then we get the bound we want. The same ideas applied to $J_{m}$ give
\begin{align*}
    \left\|J_{m}\right\|_{L^{\infty}_{\xi}} \lesssim \eps_1^4 2^{-\frac14 m} \Rightarrow  \sum_{m}  \left\|J_{m}\right\|_{L^{\infty}_{\xi}} \lesssim \eps_1^4. 
\end{align*}
\par \noindent 
\textbf{Case 2}
\par \noindent Now we let $|\xi|$ get bigger while keeping all the $\eta_j$'s small: 
\begin{align*}
\left|\eta_{j}\right| \leq 2^{k_{*}}, \quad \left|\xi\right| \sim 2^{k}, \quad k_{*}\leq k\ll 1. 
\end{align*}
We start by using Taylor expansion and the assumed sizes of the frequencies to discover
\begin{subequations}
\begin{align}
    \label{eqn:case2_phase_bnd_pure_zero}
    \left|\varphi\right| &\gtrsim 2^{3k}
    \\
    \label{eqn:case2_useful_aux_bnd}
    \left|\partial_{\xi}\varphi\right| &\lesssim 2^{2k}. 
    \end{align}
\end{subequations}
In particular, we have $\left|\frac{\omega\partial_{\xi}\varphi}{\varphi}\right|\lesssim 1$. The estimate \eqref{eqn:case2_phase_bnd_pure_zero} allows us to integrate by parts in $\tau$. To control the oscillatory integrals resulting from this integration by parts, we need to bound $\left\|\partial_{\tau}\widehat{f}_{\lesssim 2^{k_{*}}}(\eta_1)\right\|_{L^{\infty}_{\eta_1}}$. First, recall that 
$$
\partial_{\tau}\widehat{f}_{\lesssim 2^{k_{*}}} = \omega(\eta_1) \phi_{\lesssim 2^{k_*}}(\eta_1) e^{i\tau\omega} \left(u^4\right)^{\wedge}(\eta_1). 
$$
Therefore, 
\begin{align}
|\partial_{\tau}\widehat{f}_{\lesssim 2^{k_{*}}}| &\lesssim 2^{k_{*}}\phi_{\lesssim 2^{k_*}}(\eta_1) \left\|u^4\right\|_{L^1_{x}} \lesssim \eps_1^4 2^{k_{*}} 2^{-\frac{2m}{3}} \phi_{\lesssim 2^{k_*}}(\eta_1). \label{eqn:time_deriv_est_A}
\end{align}
Similarly, we find that 
\begin{equation}
\label{eqn:time_deriv_est_B}
\left\|\psi_{k}(\xi)\phi_{\lesssim 2^{k_{*}}}\left(\eta_1\right) \phi_{\lesssim 2^{k_{*}}}\left(\eta_2\right) \phi_{\lesssim 2^{k_{*}}}\left(\eta_3\right) \partial_{\tau}\widehat{f}\left(\eta_4\right)\right\|_{L^{\infty}_{\xi\eta_1\eta_2\eta_3}} \lesssim \eps_1^4 2^{k_{*}} 2^{-\frac{2m}{3}} . 
\end{equation}
Combining this estimate with the volume trick, \eqref{eqn:case2_phase_bnd_pure_zero}, and \eqref{eqn:case2_useful_aux_bnd}, we get
\begin{align*}
    \left\|I_{mk}\right\|_{L^2_{\xi}}
    \lesssim 2^{\frac12 k}\left(\eps_1^7 2^{4m\beta} + \eps_1^4 2^{3m\beta}\right)\Rightarrow
    \sum_{m=0}^{m_{*}}\sum_{k_{*}<k\ll 1}  \left\|I_{mk}\right\|_{L^2_{\xi}} \lesssim \eps_1^7  t^{4\beta} + \eps_1^4  t^{3\beta}.
\end{align*}
Therefore, we demand $\beta \leq \frac{p_0}{4}$. 
\par We turn to $J_{m}$ now. We may use \eqref{eqn:case2_phase_bnd_pure_zero} to justify an integration by parts in $\tau$. After performing this integration by parts, we use the volume trick, \eqref{eqn:time_deriv_est_A}, and \eqref{eqn:time_deriv_est_B} to find
\begin{align*}
    \left\|J_{mk}\right\|_{L^{\infty}_{\xi}}
    &\lesssim \eps_1^4 2^{-2k} 2^{m\left(-1+3\beta\right)} \Rightarrow \sum_{2^{k_{*}}\lesssim 2^{k} \ll 1} \left\|J_{mk}\right\|_{L^{\infty}_{\xi}} \lesssim \eps_1^4 2^{m\left(-\frac13 +\beta\right)} \lesssim \eps_1^4 2^{m\left(-\frac15\right)}
\end{align*}
where the last inequality follows from $\beta$ being small. Summing over $m$ then gives the desired bound. 
\par \noindent 
\par \noindent \textbf{Case 3}
\par\noindent
We now localize according to 
\begin{align*}
\left|\eta_{j}\right| &\sim 2^{k_{j}}, \quad 2^{k_{*}}\leq 2^{k_j}\leq 2^{-10}, \quad \forall \ j=1,2,3, \quad k_{1}\geq k_{2}\geq k_{3},
\\
k_{1}&\gg k_{2} \ \text{or} \ k_{1}\gg k_{3} \ \text{or} \ k_{2}\gg k_{3},
\\
\left|\xi\right| &\leq 2^{k_{*}},  \quad \text{or} \quad \left|\xi\right| \sim 2^{k}, \quad 2^{k_{*}}\leq 2^{k}\leq 2^{-10}, \quad k\lesssim k_{1}.  
\end{align*}
It suffices to establish a useful growth bound when $k_1\gg k_2$ and $|\xi|\sim 2^{k}$. To begin, use Taylor expansion to get
\begin{equation}\label{eqn:case_3_derivative_bound}
    \left|\partial_{\eta_{1}-\eta_{2}}\varphi\right| \gtrsim 2^{2k_{1}}.
\end{equation}
Additionally, we have from Taylor expansion or the mean value theorem that
\begin{equation}
\label{eqn:case3_aux_bounds}
\left|\partial_{\xi}\varphi\right| \lesssim 2^{k_1+\max\left\{k,k_1\right\}}  \quad \text{and} \quad \left|\partial_{\eta_1-\eta_2}^2\varphi\right|\lesssim 2^{k_{1}}. 
\end{equation}
Consequently, we have 
\begin{align}
\label{eqn:case3_symbol_bounds}
 \left|\frac{\partial_{\xi}\varphi}{\partial_{\eta_1-\eta_2}\varphi}\right|\lesssim 2^{-k_1}2^{\max\left\{k, k_1\right\}} \lesssim 1,
 \quad  \left|\partial_{\eta_1-\eta_2}\left(\frac{\partial_{\xi}\varphi}{ \partial_{\eta_1-\eta_2}\varphi} \right) \right| \lesssim 2^{-2k_{1}}2^{\max\left\{k, k_1\right\}}\lesssim 2^{-k_1}.
\end{align}
The above estimates allow us to verify that the symbols
\begin{align*}
    M_{j} &=  \left[\partial_{\eta_1-\eta_2}^{j-1}\left(\frac{\partial_{\xi}\varphi}{\partial_{\eta_1-\eta_2}\varphi}\right)\right] \ \psi_{k}\left(\xi\right) \ \psi_{k_1}\left(\eta_1\right)\ \psi_{k_2}\left(\eta_2\right)\ \psi_{k_3}\left(\eta_3\right) \quad j=1, 2
\end{align*}
obey the conditions of lemma \ref{lemma:multiplier_cheat_code} with $A=1, 2^{-k_1}$ respectively (handling the $\partial_{\xi}$ terms relies strongly on $2^k\lesssim 2^{k_1}$). Integrating by parts in $\partial_{\eta_1-\eta_2}$ and applying the $L^2_{\xi}$ multilinear estimate from proposition \ref{prop:multilinear_estimates} gives
\begin{align*}
\left\|I_{mk_1k_2k_3k} \right\|_{L^2_{\xi}}
&\lesssim \eps_1^4 \int_{t_1}^{t_2} \diff \tau \ 2^{k-\frac12 \left(k_1+k_2+k_3\right)} 2^{-\frac{4m}{3}} + 2^{k-\frac12 \left(k_2 + k_3\right)} 2^{m\left(p_0-\frac43\right)}
 \\
 \Rightarrow \sum_{\text{freqs.}}   \left\|I_{mk_1k_2k_3k}  \right\|_{L^2_{\xi}} & \lesssim \eps_1^4 \left(2^{m\left(-\frac32\beta+\frac16\right)} +\log t \ 2^{m\left(-\beta + p_0\right)}\right)
 \\
 \Rightarrow 
\sum_{m=0}^{m_{*}}  \sum_{\text{freqs.}} \left\|I_{mk_1k_2k_3k} \right\|_{L^2_{\xi}} &\lesssim \eps_1^4 \left(t^{-\frac32\beta+\frac16} +\ t^{-\frac12\beta + p_0}\right).
\end{align*}
Therefore, we also require that 
\begin{equation}
    \label{eqn:1/6 and p0 constraint}
-\frac32 \beta + \frac16 \leq p_0.  
\end{equation}
\par We now apply the same ideas to control the appropriate piece of $J_{m}$. We use \eqref{eqn:case_3_derivative_bound} to integrate by parts in $\partial_{\eta_1-\eta_2}$, noting that
$$
\left|\partial_{\eta_1-\eta_2}\left(\frac{1}{ \partial_{\eta_1-\eta_2}\varphi} \right)\right| \lesssim 2^{-3k_1}
$$
using linear vanishing of $\omega''$ near the origin. This allows one to establish that the appropriate symbol satisfies the conditions of lemma \ref{lemma:multiplier_cheat_code}, and in turn the multilinear estimate from proposition \ref{prop:multilinear_estimates_Linfty} gives
\begin{align*}
    \left\|J_{mk_1k_2k_3k}\right\|_{L^{\infty}_{\xi}}
    &\lesssim \eps_1^4 \left[ 2^{-\frac32 k_1 -\frac12 k_2 -\frac12 k_3} 2^{-m} + 2^{-k_1-\frac12 k_2-\frac12 k_3} 2^{m\left(p_0-1\right)} \right].
\\
\Rightarrow
\sum_{\text{freqs.}}  \left\|J_{mk_1k_2k_3k}\right\|_{L^{\infty}_{\xi}} &\lesssim \eps_1^4 \left[2^{m\left(-\frac16 -\frac56 \beta\right)} + 2^{m\left(p_0-\frac13-2\beta\right)}\right] \log t 
\\
\Rightarrow  \sum_{m}  \sum_{\text{freqs.}} \left\|J_{mk_1k_2k_3k}\right\|_{L^{\infty}_{\xi}} &\lesssim \eps_1^4 . 
    \end{align*}
\par \noindent 
\textbf{Case 4}
\begin{subequations}
\begin{align}
\left|\eta_{j}\right| &\sim 2^{k_{j}} \quad 2^{k_{*}}\lesssim 2^{k_{j}}\ll 1, \quad \forall \ j=1,2,3, \quad k_{1}\geq k_{2}\geq k_{3}, 
\\
|\xi| &\sim 2^{k}  \quad 2^{k_{*}}\lesssim 2^{k}\ll  1, \quad 2^{k} \gg 2^{k_1}.
\end{align}
\end{subequations}
Our strategy here is to perform integration by parts with respect to $\eta_1$ using the bound \begin{equation}
    \label{eqn:case4_derivative_bound}
    |\partial_{\eta_{1}}\varphi| \gtrsim 2^{2k}. 
\end{equation}
To prove this, we use that $2^k\gg 2^{k_1}$ implies $|\eta_4|\approx 2^k$. Taylor expansion and the difference of squares formula then give \eqref{eqn:case4_derivative_bound}. 
\par Using  \eqref{eqn:case4_derivative_bound}, we can integrate by parts in $\eta_{1}$. Before performing the integration by parts, we need to do a little more work. First, notice that 
\begin{align*}
    \partial_{\eta_{1}}\left(\frac{\partial_{\xi}\varphi}{\partial_{\eta_{1}}\varphi}\right) = \frac{\partial^2_{\eta_1\xi}\varphi}{\partial_{\eta_{1}}\varphi} - \frac{\partial_{\eta_1\eta_{1}}^2\varphi \partial_{\xi}\varphi}{\left(\partial_{\eta_{1}}\varphi\right)^2}.
\end{align*}
Using Taylor expansion or mean value theorem arguments, we quickly find 
    \begin{align}
      |\partial_{\xi}\varphi|\lesssim 2^{2k} \quad \text{and} \quad 
        |\partial_{\eta_1\xi}^2\varphi|, \  |\partial_{\eta_1\eta_1}^2\varphi| \  \lesssim 2^k. 
    \end{align}
The above identities imply
\begin{align}
 \left|\frac{\partial_{\xi}\varphi}{\partial_{\eta_{1}}\varphi}\right| \lesssim 1, \quad  \left| \partial_{\eta_{1}}\left(\frac{\partial_{\xi}\varphi}{\partial_{\eta_{1}}\varphi}\right) \right| \lesssim 2^{-k}. \label{eqn:case4_symbol_bound_b} 
\end{align}
Therefore, the symbols 
\begin{align*}
    M_{j} =  \left[\partial_{\eta_1}^{j}\left(\frac{\partial_{\xi}\varphi}{\partial_{\eta_{1}}\varphi}\right)\right]  \ \psi_{k}\left(\xi\right) \ \psi_{k_1}\left(\eta_1\right)\ \psi_{k_2}\left(\eta_2\right)\ \psi_{k_3}\left(\eta_3\right), \quad j=1,2
\end{align*}
obey the conditions of lemma \ref{lemma:multiplier_cheat_code} with $A= 1, 2^{-k}$, respectively. Using \eqref{eqn:case4_derivative_bound}, we may integrate by parts in $\eta_1$, then use the multilinear estimate from proposition \ref{prop:multilinear_estimates} and \eqref{eqn:bootstrap_on_LP_pieces} to get 
\begin{align*}
    \left\|I_{mk_1k_2k_3k} \right\|_{L^2_{\xi}} 
    &\lesssim \eps_1^4\int_{t_1}^{t_2}\diff\tau \ 2^{-\frac12\left(-k_1+k_2+k_3\right)} 2^{-\frac{4m}{3}} + 2^{k}2^{-\frac12\left(k_2+k_3\right)} 2^{-\frac{4m}{3}}\left(1+2^{mp_0}\right). 
\end{align*}
Note how the ``$1+\tau^{p_0}$'' term comes from when we differentiate $f_{k_{1}}$: the derivative will hit a cutoff and this is the source of the ``$1$''. Also, we have put the $\eta_4$ term in $L^\infty_{x}$ on the first and second lines using the worst-case decay estimate. Summing over all pieces gives
\begin{align*}
    \sum_{m=0}^{m_{*}} 
     \sum_{\text{freqs.}}\left\|I_{mk_1k_2k_3k} \right\|_{L^2_{\xi}} 
     &\lesssim \eps_1^4 \left(t^{-\beta +\frac16} + t^{-\frac12\beta + p_0} + 1\right). 
\end{align*}
The above is $\lesssim \eps_1^4 t^{p_0}$ provided $-\beta +\frac16 \leq p_0$, which is slightly stricter than \eqref{eqn:1/6 and p0 constraint}. 
\par As for $J_{m}$,  we may integrate by parts with respect to $\eta_1$ using \eqref{eqn:case4_derivative_bound}. We also need the  bounds
\begin{align*}
\left|\frac{1}{\partial_{\eta_1}\varphi}\right| \lesssim 2^{-2k}, \quad 
\left|\partial_{\eta_1}\left(\frac{1}{\partial_{\eta_1}\varphi}\right)\right| \lesssim 2^{-3k}.
\end{align*}
From here, we reduce to a minor variant of case 3. 
\par \noindent 
\textbf{Case 5}
\par \noindent
We now localize according to 
\begin{subequations}
\begin{align}
\left|\eta_{j}\right| &\sim 2^{k_{j}}, \quad 2^{k_{*}}\lesssim 2^{k_j}\lesssim 1, \quad \forall \ j=1,2,3, \quad 2^{k_1}\approx 2^{k_{2}}\approx 2^{k_3}
\\
\left|\xi\right| &\lesssim 2^{k_{*}} \quad \text{with} \quad 2^{k_{*}}\lesssim 2^{k_{1}} \quad \text{or}
\\
\left|\xi\right| &\sim 2^{k}, \quad 2^{k_{*}}\lesssim 2^{k}\lesssim 1, \quad 2^{k} \lesssim 2^{k_1}. 
\end{align}
\end{subequations}
In this case, integration by parts in any of the directions $\partial_{\eta_j-\eta_k}$ definitely won't work, and since we may have $|\xi|\approx |\eta_j|$ for any $j$ it also appears that obtaining a globally valid lower bound on $\varphi$ in this case is out of the question. So, we should turn our attention towards the direction $\partial_{\eta_{j}}$. We find by Taylor expansion and the difference-of-squares formula that 
\begin{align}\label{eqn:craig}
    |\partial_{\eta_j}\varphi| \gtrsim \left|\eta_j-\eta_4\right| \ \left|\eta_j+\eta_4\right| . 
\end{align}
This suggests that we perform \emph{six} dyadic localizations, one for each combination $\eta_j \pm \eta_4$. 
\par \noindent
\textbf{Subcase 5-1}
\par \noindent
In this subcase, we have $\left| \eta_j \pm \eta_4\right|\lesssim 2^{k_{*}} \quad \forall \ j=1,2,3$. These restrictions immediately imply $|\xi|, |\eta_j|\lesssim 2^{k_{*}} \quad \forall \ j$, and we're back in case 1. 
\par \noindent
\textbf{Subcase 5-2}
\par \noindent
In this subcase, we suppose there is at least one $j$ such that 
$$
\left| \eta_j \pm \eta_4\right|\sim 2^{\ell_j}, \quad 2^{k_{*}}\lesssim 2^{\ell_j} \ll 1
$$
for \emph{both} of the $\pm$ signs. Without loss of generality, suppose this particular $j$ is $j=1$. There are two additional subcases to consider.
\begin{itemize}
    \item First, suppose $2^{\ell_1}\gtrsim 2^{k_1}$. We know that $\left|\partial_{\eta_1}\varphi\right| \gtrsim 2^{2\ell_1}$.
    The arguments of case 3 and $2^{k}\lesssim 2^{k_1}$ then give
    \begin{align*}
         \left|\frac{\partial_{\xi}\varphi}{\partial_{\eta_1}\varphi}\right| \lesssim 2^{2k_1}2^{-2\ell_1} \lesssim 1, \quad  \left|\partial_{\eta_1}\left(\frac{\partial_{\xi}\varphi}{\partial_{\eta_1}\varphi}\right)\right| \lesssim 2^{2k_1}2^{-2\ell_1} \lesssim 2^{-\ell_1}.
    \end{align*}
    With these bounds at hand, the proof reduces to copying case 3. Note that, since $2^{\ell_1}\gtrsim 2^{k_1}$, we can still use lemma \ref{lemma:multiplier_cheat_code} to justify applications of the multilinear estimate in proposition \ref{prop:multilinear_estimates} even though we are performing a secondary localization (read: it suffices to check the $\eta_j,\xi$-derivatives of our symbols). 
    \item Next, suppose $2^{\ell_1}\ll 2^{k_1}$. In this case, we must have $|\eta_1\pm \eta_4| \gtrsim 2^{k_1}$ for either $+$ or $-$: that is, if $\eta_1\approx \eta_4 + \mathcal{O}\left(2^{\ell_1}\right)$ and $2^{k_1}\gg 2^{\ell_1}$, then $\eta_1$ and $\eta_4$ must lie on the same side of the origin. Accordingly, we have $\left|\partial_{\eta_1}\varphi\right|\gtrsim 2^{k_1}2^{\ell_1}$. This implies 
    \begin{align*}
         \left|\frac{\partial_{\xi}\varphi}{\partial_{\eta_1}\varphi}\right| \lesssim 2^{2k_1}2^{-2\ell_1} \lesssim 2^{k_1}2^{-\ell_1}, \quad \left|\partial_{\eta_1}\left(\frac{\partial_{\xi}\varphi}{\partial_{\eta_1}\varphi}\right)\right| \lesssim 2^{2k_1}2^{-2\ell_1} \lesssim 2^{k_1}2^{-2\ell_1}.
    \end{align*}
    Using these bounds we find that, for $j=1,2$, the symbols 
    \begin{align*}
        M_j &= \left[\partial_{\eta_1}^{j-1}\left(\frac{\partial_{\xi}\varphi}{\partial_{\eta_1}\varphi}\right)\right] \psi_{k}\left(\xi\right) \psi_{k_1}\left(\eta_1\right)\psi_{k_2}\left(\eta_2\right)\psi_{k_3}\left(\eta_3\right) \psi_{\ell_1}\left(\eta_1\pm\eta_4\right)\psi_{\ell_2}\left(\eta_2\pm\eta_4\right) \psi_{\ell_3}\left(\eta_3\pm\eta_4\right)
    \end{align*}
    satisfy the hypothesis of lemma \ref{lemma:multiplier_cheat_code_secondary_localization}  with $A=2^{k_1}2^{-\ell_1}, 2^{k_1}2^{-2\ell_1}$ respectively. We may then apply the $L^2_{\xi}$ multilinear estimate from proposition \ref{prop:multilinear_estimates} to discover 
    \begin{align*}
        \left\|I_{mk_1k_2k_3k\ell_1 \ell_2 \ell_3}\right\|_{L^2_{\xi}}
        &\lesssim \eps_1^4 2^{-\frac{m}{2}}2^{\frac{k}{2}}\left[2^{-2\ell_1}+ 2^{mp_0}2^{-\ell_1}\right].
    \end{align*}
    Note how we used $2^{k}\lesssim 2^{k_1}$ and $|\eta_4|\sim 2^{k_1}$ in the above estimate. Summing over each $\ell_j$ and $k_j$ gives
    \begin{align*}
        \sum_{\text{freqs.}} \left\|I_{mk_1k_2k_3k\ell_1 \ell_2 \ell_3}\right\|_{L^2_{\xi}} &\lesssim \eps_1^4 m^5 \left[2^{m\left(\frac16-2\beta\right)} + 2^{m\left(p_0-\beta\right)}\right]. 
    \end{align*}
For $\beta$ sufficiently large (but still very small), this bound gives the desired estimate upon summation over $m$.
\end{itemize}
The key idea in the above cases is that, when one is performing a secondary localization in a situation in a case leveraging null structure, one must make sure that the null structure persists in the secondary localization, not just the primitive one. 
\par\noindent
\textbf{Subcase 5-3}
\par \noindent 
In this subcase, we suppose that
\begin{align*}
    \left| \eta_1 - \eta_4\right| \lesssim 2^{k_{*}}, \quad \left| \eta_2 + \eta_4\right| \lesssim 2^{k_{*}}, \quad  \left| \eta_3 + \eta_4\right| \lesssim 2^{k_{*}},
\end{align*}
but the other three combinations are $\gtrsim 2^{k_{*}}$: in fact, the three localizations above indicate that we must have 
$$
    \left| \eta_1 + \eta_4\right| \sim 2^{k_{1}}, \quad  \left| \eta_2 - \eta_4\right|  \sim 2^{k_{1}}, \quad \left| \eta_3 - \eta_4\right| \sim 2^{k_{1}}. 
$$
All these localizations imply that $|\xi| \lesssim 2^{k_{*}}$, meaning we are very close to the space-time resonant line from \eqref{eqn:resonant_line_discussion}. So, the situation is much more complicated here than in the analogous situation with the anomalous resonance. We start by localizing $|\xi|$ again according to 
\begin{align*}
    |\xi| \sim 2^{k}, \quad 2^{-100m} \lesssim 2^{k} \lesssim 2^{k_{*}}, \quad \text{or} \quad  |\xi| \lesssim 2^{-100m}. 
\end{align*}
We then perform a tertiary localization according to 
\begin{align*}
    \left| \eta_1 - \eta_4\right| \sim 2^{k}, \quad \left| \eta_2 + \eta_4\right| \sim 2^{k}, \quad  \left| \eta_3 + \eta_4\right| \sim 2^{k}. 
\end{align*}
First, assume $|\xi|\gtrsim 2^{-100m}$. We may use \eqref{eqn:craig},  $\left| \eta_1 + \eta_4\right| \sim 2^{k_1}$, and $\left| \eta_1 - \eta_4\right| \sim 2^{k}$ to discover 
\begin{align}
\label{eqn:michael}
    \left|\partial_{\eta_1}\varphi\right| \gtrsim  2^{k} 2^{k_{1}}.  
\end{align}
Consequently, $|\partial_{\xi}\varphi|\lesssim 2^{2k_1}$ gives
\begin{align*}
    \left|\frac{\partial_{\xi}\varphi}{\partial_{\eta_1}\varphi}\right| \lesssim 2^{-k}2^{k_1}, \quad \left|\partial_{\eta_1}\left(\frac{\partial_{\xi}\varphi}{\partial_{\eta_1}\varphi}\right)\right| \lesssim 2^{-k}. 
\end{align*}
Since we still have $|\omega(\xi)|\lesssim 2^{k}$, integration by parts in $\partial_{\eta_1}$ and the $L^2_{\xi}$ multilinear estimate from proposition \ref{prop:multilinear_estimates} (the hypothesis of which can be checked using lemma \ref{lemma:symbol_algebra_property} and lemma \ref{lemma:multiplier_cheat_code_secondary_localization}) give
$$
\left\|I_{...}\right\|_{L^2_{\xi}}\lesssim \eps_1^4 \left(\log t\right)^3\sum_{k} 2^{m\left(p_0-\frac13\right)} + 2^{m\left(-\frac16-\beta\right)}.
$$
Since $2^{k_{*}}\gtrsim 2^{k} \gtrsim 2^{-100m}$, there are only $\mathcal{O}\left(\log t\right)$ terms in the sum over $k$, hence 
$$
\left\|I_{...}\right\|_{L^2_{\xi}}\lesssim \eps_1^4  \left(\log t\right)^4\left[2^{m\left(p_0-\frac13\right)} + 2^{m\left(-\frac16-\beta\right)}\right]. 
$$
Summing over $m$ gives a growth bound that's even stronger then the one we need, so provided  $|\xi|\gtrsim 2^{-100m}$ we're done. In the complementary case  $|\xi|\lesssim 2^{-100m}$ we have the ``mega-null'' structure $|\omega(\xi)| \lesssim 2^{-100m}$, which immediately gives the claim upon application of the volume trick. 

\par \noindent
\textbf{Subcase 5-4}
\par \noindent
In this subcase, we suppose that
$$
\left| \eta_j - \eta_4\right|\lesssim 2^{k_{*}} \quad \forall \ j =1,2,3, \quad \text{but} \quad \left| \eta_j + \eta_4\right|\gtrsim 2^{k_{*}} \quad \forall \ j =1,2,3. 
$$
Of course, the above restriction actually means $\left| \eta_j + \eta_4\right|\sim 2^{k_{1}} \quad \forall \ j=1,2,3$. This implies 
\begin{align}
    \xi = \sum_{j=1}^4\eta_j = 4\eta_1 + \mathcal{O}\left(2^{k_{*}}\right),
\end{align}
meaning that $|\xi|\sim 2^{k_1}$ in this region. In particular, the smallest $\xi$ can get is $|\xi|\sim 2^{k_*}$. In this worst-case scenario, we simply reduce to case 1. In the complementary case $2^{k}\approx 2^{k_1}\gg 2^{k_{*}}$, then the phase then takes the approximate form
\begin{align}
\label{eqn:brian}
    \varphi = -\omega\left(\xi\right) + 4\omega\left(\frac{\xi}{4}\right) + \mathcal{O}\left(2^{k}2^{2k_{*}}\right) = \frac{15}{16} \xi^3 + \mathcal{O}\left(2^{5k}, 2^{k}2^{2k_{*}}\right),
\end{align}
implying $|\varphi| \gtrsim 2^{3k}$. From here, one proceeds as in case 2: the only difference is that the $L^2_{\xi}$ multilinear estimate from proposition \ref{prop:multilinear_estimates} must be used instead of the volume trick, so one must verify that the symbol 
\begin{align*}
M &= \frac{\partial_{\xi}\varphi}{\varphi} \psi_{k}\left(\xi\right) \psi_{k_1}\left(\eta_1\right)\psi_{k_2}\left(\eta_2\right)\psi_{k_3}\left(\eta_3\right) \psi_{\ell_1}\left(\eta_1 + \eta_4\right)\psi_{\ell_2}\left(\eta_2+\eta_4\right) \psi_{\ell_3}\left(\eta_3+\eta_4\right)
\\
&\phantom{=} \times \phi_{\lesssim k_{*}}\left(\eta_1 - \eta_4\right) \phi_{\lesssim k_{*}}\left(\eta_2-\eta_4\right)  \phi_{\lesssim k_{*}}\left(\eta_3-\eta_4\right) 
\end{align*}
satisfies the conditions of lemmas \ref{lemma:multiplier_cheat_code} and \ref{lemma:multiplier_cheat_code_secondary_localization} depending on the relative sizes of the $k_{j}$'s and $\ell_{j}$'s. This is straightforward using the Taylor series approximation \eqref{eqn:brian}. 
\par\noindent
\textbf{Subcase 5-5}
\par \noindent 
Now, for any $\sigma_j=\pm 1$, we suppose that 
\begin{align*}
    \left| \eta_j + \sigma_j\eta_4\right| \lesssim 2^{k_{*}}, \quad \text{but} \quad  \left| \eta_j - \sigma_j\eta_4\right| \gtrsim 2^{k_{*}}.
\end{align*}
The case $\sigma_1=-\sigma_2=-\sigma_3=-1$ is subcase 5-3, and
the case $\sigma_1=\sigma_2=\sigma_3=-1$ is subcase 5-4. Up to permuting variables, the only subcase that remains is 
$\sigma_1=-\sigma_2=\sigma_3=-1$. In this situation, $\xi = 2\eta_1 + \mathcal{O}\left(2^{k_{*}}\right)$, which is enough to reduce the proof to a minor variant of subcase 5-4. 
\par We turn to $J_{m}$ again. We know from our weighted estimates that we must perform a secondary localization in each $\eta_j\pm \eta_4$ to properly handle this case. Let's use the same naming for subcases that we applied for the weighted estimates. The only subcase that does not immediately reduce to an earlier case is subcase 5-3. We still perform a tertiary localization, and using \eqref{eqn:michael} we find that when $|\xi|\sim 2^{k}\gtrsim 2^{-100m}$ we have
\begin{align*}
    \left|\frac{1}{\partial_{\eta_1}\varphi}\right| \lesssim 2^{-k} 2^{-k_1}, \quad \left|\partial_{\eta_1}\left(\frac{1}{\partial_{\eta_1}\varphi}\right)\right| \lesssim 2^{-k} 2^{-2k_1}.
\end{align*}
Integrating by parts and using our multilinear estimate gives 
\begin{align*}
    \left\|J_{mk_1k_2k_3k}\right\|_{L^{\infty}_{\xi}} \lesssim \eps_1^4 \left[ 2^{-\frac32 k_1-\frac12 k_2 - \frac12 k_3} 2^{-m} + 2^{m\left(p_0-1\right)}2^{-k_1-\frac12 k_2-\frac12 k_3} \right].
\end{align*}
From here, the rest of the proof follows from copying case 3. 
\begin{remark} Where do the localization choices made in subcase 5-3 come from? As already explained, subcase 5-3 forces us to tightly cling to the resonant line $L$ given by \eqref{eqn:resonant_line_discussion}. If our $\eta_{j}$'s, but not our $\xi$, lie on this line, then 
\begin{align*}
\varphi = -\omega(\xi) - \omega(\eta_1) + \omega(\xi+\eta_1) = 3\xi\eta_1\left(\xi+\eta_1\right).
\end{align*}
Consequently, if we are near $L$ and $|\xi|\ll 2^{k_1}$, then we approximately have $\left|\partial_{\eta_1}\varphi\right|\gtrsim |\xi|2^{k_1}$. This means that the gradient of $\varphi$ degenerates along $L$ \emph{only} because of $\xi=0$! Therefore, since we wanted to capture this degeneration properly in subcase 5-3, we had to demand that $|\eta_1-\eta_4|\approx |\xi|$.
\end{remark}
\subsection{$(-\sqrt{3},\sqrt{3},\sqrt{3};0)$}
\label{ss:type3}
\noindent We now focus on the space-time resonance $(-\sqrt{3},\sqrt{3},\sqrt{3};0)$. There is still some null structure to help counteract the degeneracy: $\left|\omega(\xi)\right| \sim  |\xi| \ \forall \ |\xi|\ll 1$. However, note that $\partial_{\xi}\varphi = \mathcal{O}(1)$ near this resonance, so we don't have as much null as we had around $(0,0,0;0)$. Fortunately, we'll discover that there is some hidden null owing to the asymmetry between the order of vanishing of $\varphi$ in $\xi$ versus each $\eta_j$. For the remainder of this subsection, let 
\begin{equation}
    \label{eqn:sigma_defn}
    \sigma_{j} = \left\{
    \begin{aligned}
    -1 \quad j&=1
    \\
    +1 \quad j&=2,3
    \end{aligned} \quad . 
    \right. 
\end{equation}
Additionally, let $k_{*}$ be the largest integer such that $2^{k_{*}} \leq 2^{-20}2^{m\left(\gamma-\frac13\right) }$ where $\gamma>0$ will be constrained as we work. 

\par \noindent 
\textbf{Case 1}
\begin{align*}
\left|\eta_{j}-\sigma_{j}\sqrt{3}\right| \leq 2^{k_{*}} \quad \forall \ j=1,2,3, \quad \text{and} \quad \left|\xi\right| \leq 2^{3k_{*}}. 
\end{align*}
Using the volume trick as we have seen many times before,
$$
\sum_{m}\left\|I_{mk_{*}}\right\|_{L^2_{\xi}} \lesssim \eps_1^4 t^{\frac{13}{2}\gamma-\frac16}.
$$
For this bound to be acceptable, it suffices to demand $\gamma\leq \frac{1}{39}$. As for $J_m$, the volume trick gives
\begin{align*}
    \left\|J_{mk_{*}}\right\|_{L^{\infty}_{\xi}} \lesssim \eps_1^4  2^{6k_{*}}2^{m} \lesssim \eps_1^4 2^{m\left(6\gamma-1\right)},
\end{align*}
which yields $\lesssim \eps_1^4$ when summed over $m$. 
\par \noindent 
\textbf{Case 2}
\begin{align*}
\left|\eta_{j}-\sigma_{j}\sqrt{3}\right| \leq 2^{k_{*}} \quad \forall \ j=1,2,3, \quad \text{and} \quad \left|\xi\right| \sim 2^{k}, \quad 2^{3k_{*}}\lesssim 2^{k} \ll 1.
\end{align*}
First, we use Taylor expansion to discover
\begin{equation}
    \label{eqn:case2_phase_taylor}
    \varphi = -\frac98 \xi + \mathcal{O}\left(\xi^3, \max_{j}|\eta_j-\sigma_j\sqrt{3}|^3\right).
\end{equation}
That is, $|\xi|$ should be compared to $\max_{j}|\eta_j-\sigma_j\sqrt{3}|^3$. Using the above, we find 
\begin{equation}
    \label{eqn:case2_phase_bnd}
|\varphi|\gtrsim 2^k. 
\end{equation}
Since $|\omega(\xi)|\sim 2^{k}$ and $|\partial_{\xi}\varphi|\lesssim 1$, integration by parts in $\tau$ and an application of the volume trick yields
\begin{align*}
\hspace{-1.5cm}
    \left\|I_{mk}\right\|_{L^2_{\xi}}
    &\lesssim \int_{t_1}^{t_2}\diff \tau \ \tau \left\|\omega(\xi)\psi_{k}\left(\xi\right)\int\diff\eta_1\diff\eta_2\diff\eta_3 \ e^{-i\tau\varphi} \ \frac{\partial_{\xi}\varphi}{\varphi} \ \tau\  \partial_{\tau}\left(\widehat{f}_{\lesssim 2^{k_{*}}} \left(\eta_{1}\right)\right) \widehat{f}_{\lesssim 2^{k_{*}}} \left(\eta_{2}\right) \widehat{f}_{\lesssim 2^{k_{*}}} \left(\eta_{3}\right) 
    \widehat{f}\left(\eta_4\right)  \right\|_{L^2_{\xi}}
    \\
    &\phantom{=} + \eps_1^4\int_{t_1}^{t_2} \diff \tau \ 2^{m\left(3\gamma-1\right)}2^{\frac{k}{2}} + \left\{\text{similar terms}\right\}. 
\end{align*}
To estimate the first line, first define a symbol by 
\begin{align*}
M &= \frac{1}{\varphi}\ \psi_{k}\left(\xi\right) \ \phi_{\lesssim 2^{k_*}}\left(\eta_1+\sqrt{3} \right)\ \phi_{\lesssim 2^{k_*}}\left(\eta_2 - \sqrt{3}\right) \ \phi_{\lesssim 2^{k_*}}\left(\eta_3 - \sqrt{3}\right). 
\end{align*}
Since $M$ satisfies lemma \ref{lemma:multiplier_cheat_code} with $A=2^{-k}$, we may apply proposition \ref{prop:multilinear_estimates} and 
 \eqref{eqn:bootstrap_amplitude} to find
\begin{align*}
     \left\|I_{mk}\right\|_{L^2_{\xi}}
    &\lesssim \int_{t_1}^{t_2}\diff \tau \ 2^{m} \ \left\|\partial_{\tau}f_{\lesssim 2^{k_{*}}}\right\|_{L^2_{x}} \ \left(\eps_1 2^{-\frac{m}{3}}\right)^3  + \eps_1^4\int_{t_1}^{t_2} \diff \tau \ 2^{m\left(3\gamma-1\right)}2^{\frac{k}{2}} + \left\{\text{similar terms}\right\}. 
\end{align*}
Since the PDE can be written as $\partial_{\tau}f \simeq u^4$, another application of the bootstrap decay estimate gives 
\begin{align*}
     \left\|I_{mk}\right\|_{L^2_{\xi}}
    &\lesssim\int_{t_1}^{t_2}\diff \tau \ \eps_1^7 2^{-m} +\eps_1^4 2^{m\left(3\gamma-1\right)}2^{\frac{k}{2}} 
\\ \Rightarrow \sum_{3k_{*}\lesssim k \ll 0}\left\|I_{mk}\right\|_{L^2_{\xi}}
    &\lesssim\int_{t_1}^{t_2}\diff \tau \ \eps_1^7 \ \log t \ 2^{-m} +\eps_1^4\tau^{m\left(3\gamma-1\right)} \lesssim \eps_1^7 \left(\log t\right) + \eps_1^4 2^{3m\gamma} 
\\ \Rightarrow
    \sum_{m} \sum_{3k_{*}\lesssim k \ll 0}\left\|I_{...}\right\|_{L^2_{\xi}}
    &\lesssim \eps_1^4  t^{3\gamma}. 
\end{align*}
Provided $\gamma \leq \frac{p_0}{3}$, this bound is acceptable. Since $p_0 = \left(\frac16\right)^{-}$, this is implied by $\gamma\leq \frac{1}{39}$.  
\par For $J_{m}$, we integrate by parts in $\tau$ as above. Using the volume trick on the boundary term and the multilinear estimate on everything else ($|\omega/\varphi|\lesssim 1$), this gives
\begin{align*}
    \left\|J_{mk}\right\|_{L^\infty_{\xi}}
    &\lesssim \eps_1^4 2^{m\left(-1+3\gamma\right)} + \eps_1^7 2^{-\frac{m}{3}}
\end{align*}
which is $\lesssim \eps_1^4$ after summing over $k$ and $m$. 
\par \noindent 
\textbf{Case 3}
\begin{align*}
\left|\eta_{j}-\sigma_{j}\sqrt{3}\right| &\sim 2^{k_{j}}, \quad 2^{k_*}\lesssim 2^{k_j}\ll 1 \quad \forall \ j=1,2,3, \quad k_{1}\geq k_{2} \geq k_3, 
\\
\left|\xi\right| &\sim 2^{k}, \quad 2^{3k_{*}}\lesssim 2^{k} \ll 1, \quad 2^{k} \gg 2^{3k_1}. 
\end{align*}
In this situation, \eqref{eqn:case2_phase_taylor} still tells us that 
\begin{equation}
    \label{eqn:case3_phase_bnd} 
    |\varphi| \gtrsim 2^{k},
\end{equation}
so we are clear to integrate by parts in $\tau$ as above. However, this time we use the multilinear estimate from proposition \ref{prop:multilinear_estimates} on \emph{all} terms arising from the integration by parts: the constraint $2^{k}\gg 2^{3k_1}$ implies that arguments from the previous case can be applied to show that the relevant symbol still satisfies the conditions of lemma \ref{lemma:multiplier_cheat_code}. We may also use \eqref{eqn:bootstrap_open} to obtain
$$
\left\|u_{k_j}\right\|_{L^{\infty}_{x}} \lesssim \eps_1 2^{-\frac{k_j}{2}}2^{-\frac{m}{2}}, \quad j=1,2,3
$$
since being slightly away from the degenerate stationary points gives us a little extra decay. Implementing this recipe, we get
\begin{align*}
    \left\|I_{mk_1k_2k_3k}\right\|_{L^2_{\xi}}&\lesssim \int_{t_1}^{t_2} \diff \tau \ \eps_1^7 2^{-\frac12\left(k_2+k_3\right)} 2^{-\frac{4m}{3}} + \eps_1^4 2^{-\frac12\left(k_1+k_2+k_3\right)}2^{-\frac{3m}{2}}
\\ \Rightarrow
    \sum_{\text{freqs.}}\left\|I_{mk_1k_2k_3k}\right\|_{L^2_{\xi}}
&\lesssim \eps_1^7 2^{-m\gamma}\left(\log t\right)^2 +  \eps_1^4 \log t \ 2^{-\frac{3m}{2}\gamma}.
\end{align*}
Summing over $m$ then gives the claimed bound. 
\par When bounding $J_{m}$, we still have \eqref{eqn:case3_phase_bnd}, so we can integrate by parts as above: 
\begin{align*}
   \left\|J_{mk_1k_2k_3k}\right\|_{L^{\infty}_{\xi}} 
   &\lesssim \eps_1^4 2^{-2m}2^{-\frac12k_1-\frac12k_2} + \eps_1^7 2^{-\frac{5m}{3}} 2^{-\frac12 k_2-\frac12 k_3}
\\ \Rightarrow
    \sum_{\text{freqs.}}\left\|J_{mk_1k_2k_3k}\right\|_{L^{\infty}_{\xi}} & \lesssim \eps_1^4 2^{-\frac{2m}{3}}.
\end{align*}
Summing over $m$ then gives the desired bound. 
\par \noindent 
\textbf{Case 4}
\begin{align*}
\left|\eta_{j}-\sigma_{j}\sqrt{3}\right| &\sim 2^{k_{j}}, \quad 2^{k_*}\lesssim 2^{k_j}\ll 1, \quad \forall \ j=1,2,3, \quad k_{1}\gg k_{2} \geq k_3,
\\ \left|\xi\right| &\sim 2^{k}, \quad 2^{3k_{*}}\lesssim 2^{k} \ll 1, \quad 2^{k}  \lesssim 2^{3k_1} \quad \text{or} \quad |\xi|\lesssim 2^{3k_*}. 
\end{align*}
 Here, cancellations between the leading order and remainder terms in \eqref{eqn:case2_phase_taylor} are possible. So, an integration by parts in $\tau$ is not workable, and we must integrate by parts in a spatial direction instead. Initially, this might make us worried: integrating by parts in space gives us faster growth in time, the $\eta_j$'s are near-degenerate, and we have no extra null from $\partial_{\xi}\varphi$ to help fight the degeneracy. However, the linear vanishing of the phase in $\xi$ again comes to our rescue. Why is this? Simply put, if we're in a situation like this one where we can't integrate by parts in $\tau$, then we must have $|\xi|\lesssim 2^{3k_{1}}\ll 2^{k_1}$. This means that the extra $\omega(\xi)\sim \xi$ out front gives us  
a factor of $2^{3k_{1}}$ to help fight the high-order vanishing of spatial derivatives of $\varphi$. So, even though there is no null from $\partial_{\xi}\varphi$, the expression $\varphi\approx -\frac98 \xi$ gives us some hidden extra null. 
\par Let's put the above discussion into action. First, we use Taylor expansion to find
\begin{align}\label{eqn:case4_critical}
    \left|\partial_{\eta_1-\eta_2}\varphi\right| \gtrsim 2^{2k_1}. 
\end{align}
The above estimate and further Taylor expansions yield
    \begin{align}
    \left| \frac{\partial_{\xi}\varphi}{\partial_{\eta_1-\eta_2}\varphi}\right|\lesssim 2^{-2k_1}, \quad   \left|\partial_{\eta_1-\eta_2}\left(\frac{\partial_{\xi}\varphi}{\partial_{\eta_1-\eta_2}\varphi}\right)\right|\lesssim 2^{-3k_{1}}.
    \end{align}
The hypothesis of lemma
\ref{lemma:multiplier_cheat_code} is satisfied by the symbols 
\begin{align*}
M_j &= \left[\partial_{\eta_1-\eta_2}^{j-1}\left(\frac{\partial_{\xi}\varphi}{\partial_{\eta_1-\eta_2}\varphi}\right) \right]\ \psi_{k}\left(\xi\right) \ \psi_{k_1}\left(\eta_1+\sqrt{3}\right)\ \psi_{k_2}\left(\eta_2-\sqrt{3}\right)\ \psi_{k_3}\left(\eta_3-\sqrt{3}\right), \quad j=1,2, 
\end{align*}
with $A=2^{-2k_{1}}, 2^{-3k_{1}}$ respectively. Additionally, note that $\eta_4= -\sqrt{3} + \mathcal{O}\left(2^{k_1}\right)$ and $|\omega(\xi)|\lesssim 2^{3k_{1}}$ in this situation. We may then integrate by parts in $\partial_{\eta_1-\eta_2}$ and use lemma \ref{lemma:multiplier_cheat_code} and proposition \ref{prop:multilinear_estimates} to obtain 
\begin{align*}
\left\|I_{mk_1k_2k_3k}\right\|_{L^2_{\xi}} &\lesssim \eps_1^4\int_{t_1}^{t_2} \ \diff \tau \ 2^{-\frac12\left(k_1+k_2+k_3\right)} 2^{-\frac{3m}{2}} + 2^{mp_0} 2^{-\frac12\left(-k_1+k_2+k_3\right)}2^{-\frac{3m}{2}}
\\ \Rightarrow
    \sum_{m} \sum_{\text{freqs.}} \left\|I_{mk_1k_2k_3k}\right\|_{L^2_{\xi}} &\lesssim \sum_{m} \eps_1^4\int_{t_1}^{t_2} \diff \tau \ \log t \ 2^{m\left(-\frac32\gamma -1\right)} + \log t \ 2^{m\left(-\gamma +p_0 -\frac76\right)}
\lesssim \eps_1^4 t^{p_0}.
\end{align*}
\par For $J_{m}$, we only show details for the case $|\xi|\sim 2^{k}$, as the rest of the cases are similar but easier. We use \eqref{eqn:case4_critical} to justify integrating by parts in $\partial_{\eta_1-\eta_2}$. We'll also need the  bounds
\begin{align*}
    \left|\omega(\xi) \frac{\partial_{\eta_1-\eta_2}^2\varphi}{\left(\partial_{\eta_1-\eta_2}\varphi\right)^2}\right|\lesssim 1, \quad \left|\omega(\xi) \frac{1}{\partial_{\eta_1-\eta_2}\varphi}\right|\lesssim 2^{k_1}.
\end{align*}
Integrating by parts and using the multilinear estimate gives
\begin{align*}
    \left\|J_{mk_1k_2k_3k}\right\|_{L^{\infty}_{\xi}}
    &\lesssim \eps_1^4 \left[2^{-m} 2^{\frac12 \left(k_1-k_2-k_3\right)} + 2^{m\left(p_0-1\right)} 2^{k_1-\frac12 k_2-\frac12 k_3}\right]
\\ \Rightarrow
   \sum_{m}\sum_{\text{freqs.}}\left\|J_{mk_1k_2k_3k}\right\|_{L^{\infty}_{\xi}} &\lesssim \eps_1^4 t^{-\gamma-\frac23} \log t+ t^{-\gamma +p_0-\frac23}\log t\lesssim \eps_1^4.
\end{align*}
\par \noindent 
\textbf{Case 5}
\begin{align*}
\left|\eta_{j}-\sigma_{j}\sqrt{3}\right| &\sim 2^{k_{j}}, \quad 2^{k_*}\lesssim 2^{k}\ll 1, \quad \forall \ j=1,2,3, \quad 2^{k_{1}} \approx 2^{k_{2}} \approx 2^{k_3}, 
\\
\left|\xi\right| &\sim 2^{k}, \quad 2^{3k_{*}}\lesssim 2^{k} \ll 1, \quad 2^{k} \lesssim 2^{3k_1} \quad \text{or} \quad |\xi| \lesssim 2^{3k_{*}}.
\end{align*}
Now, we cannot integrate by parts in $\partial_{\tau}$ or $\partial_{\eta_j-\eta_k}$. Accordingly, we start by examining the feasibility of integrating by parts in $\partial_{\eta_1}$. Using Taylor expansion, we discover
\begin{align*}
   | \partial_{\eta_1}\varphi| &
    \gtrsim |\eta_1-\eta_4|\left|\eta_1+\eta_4+2\sqrt{3}\right|,
\\
     | \partial_{\eta_2}\varphi| &
    \gtrsim \left|\eta_2-\eta_4-2\sqrt{3}\right||\eta_2+\eta_4|,
    \\
     | \partial_{\eta_3}\varphi| &
    \gtrsim \left|\eta_3-\eta_4-2\sqrt{3}\right||\eta_3+\eta_4|. 
\end{align*}
Conceptually, this means that the smallness of our gradient depends both on
\begin{itemize}
    \item how close the $\eta_j$'s $(j=1,2,3)$ are to $\pm\eta_4$, and
    \item the differences between the displacements $\eta_j-\sigma_j\sqrt{3}$ and $\eta_4+\sqrt{3}$. 
\end{itemize}
If $\sigma_1 = -1, \sigma_2= \sigma_3=+1$,  we must localize according to
\begin{align*}
|\eta_j+\sigma_j \eta_4| &\lesssim 2^{k_{*}}, \quad \text{or}
\\
|\eta_j+\sigma_j \eta_4| &\sim 2^{\ell_{j}} \quad 2^{k_{*}}\lesssim 2^{\ell_j}\ll 1,
\\
\left|\eta_j-\sigma_j \eta_4-\sigma_j\left(2\sqrt{3}\right)\right| &\lesssim 2^{k_{*}}, \quad \text{or}
\\
\left|\eta_j-\sigma_j \eta_4-\sigma_j\left(2\sqrt{3}\right)\right| &\sim 2^{\ell_{j}} \quad 2^{k_{*}}\lesssim 2^{\ell_j}\ll 1,
\end{align*}

\par \noindent
\textbf{Subcase 5-1}
\par \noindent 
If, for all $j$, we have 
\begin{align*}
|\eta_j+\sigma_j \eta_4| \lesssim 2^{k_{*}}, \quad \text{and} \quad \left|\eta_j-\sigma_j \eta_4-\sigma_j\left(2\sqrt{3}\right)\right| \lesssim 2^{k_{*}},
\end{align*}
then $|\xi|, \left|\eta_j+\sigma_{j}\sqrt{3}\right|\lesssim 2^{k_{*}}$ for each $j=1,2,3$ and we're back to case 1 or case 2. 
\par \noindent
\textbf{Subcase 5-2}
\par \noindent 
Now, suppose that 
\begin{align*}
|\eta_1-\eta_4|, \ |\eta_1+\eta_4+2\sqrt{3}| \sim 2^{\ell_1} \gtrsim 2^{k_{*}}. 
\end{align*}
Proving the claimed bound in this case is enough to cover all cases where $|\eta_j+\sigma_j \eta_4|$ and $
\left|\eta_j-\sigma_j \eta_4-\sigma_j\left(2\sqrt{3}\right)\right|$ are not too small. We must split into two additional subcases. 
\begin{itemize}
    \item To begin, suppose $2^{\ell_1}\gtrsim 2^{k_1}$. This means $\left|\partial_{\eta_1}\varphi\right| \gtrsim 2^{2\ell_1}$. We then have 
\begin{align*}
    \left|\frac{\partial_{\xi}\varphi}{\partial_{\eta_1}\varphi}\right| \lesssim 2^{-2\ell_1}, \quad \left|\partial_{\eta_1}\left(\frac{\partial_{\xi}\varphi}{\partial_{\eta_1}\varphi}\right)\right| \lesssim 2^{-3\ell_1}, \quad |\omega(\xi)| \lesssim 2^{3\ell_1}. 
\end{align*}
Using these bounds, the proof mostly reduces to case 4; checking derivatives of our symbols with respect to the $\eta_j$'s is fine in spite of the secondary localization only because $2^{-\ell_1}\lesssim 2^{-k_1}$.
 \item Now, suppose instead that $2^{\ell_1}\ll 2^{k_1}$. Then,
 $$
 |\eta_1-\eta_4|\sim 2^{k_{1}} \quad \text{or} \quad \left|\eta_1+\eta_4+2\sqrt{3}\right|\sim 2^{k_1} \quad \text{(but not both)}. 
 $$
 For example, if $|\eta_1-\eta_4|\sim 2^{\ell_{1}}\ll 2^{k_1}$ then we automatically have 
 $$
 \left|\eta_1+\eta_4+2\sqrt{3}\right| = \left|2\left(\eta_1+\sqrt{3}\right) + \mathcal{O}\left(2^{\ell_1}\right)\right| \sim 2^{k_1}. 
 $$
 We then have the  lower bound $\left|\partial_{\eta_1}\varphi\right| \gtrsim 2^{k_1}2^{\ell_1}$. Consequently, 
\begin{align*}
    \left|\frac{\partial_{\xi}\varphi}{\partial_{\eta_1}\varphi}\right| \lesssim 2^{-k_1}2^{-\ell_1}, \quad 
    \left|\partial_{\eta_1}\left(\frac{\partial_{\xi}\varphi}{\partial_{\eta_1}\varphi}\right)\right| \lesssim 2^{-k_1}2^{-2\ell_1}, \quad  |\omega(\xi)| \lesssim 2^{3k_1}. 
\end{align*}
From here, the proof reduces to a minor variant of case 4. The most significant change from case 4 is that we need to use  lemma \ref{lemma:multiplier_cheat_code_secondary_localization} instead of  lemma \ref{lemma:multiplier_cheat_code} to check that the relevant symbol satisfies the hypothesis of the multilinear estimate. 
    \end{itemize}

\par \noindent
\textbf{Subcase 5-3}
\par \noindent 
Now, suppose
\begin{align*}
    \left| \eta_1 - \eta_4\right|,  \left| \eta_2 + \eta_4\right|,  \left| \eta_3 + \eta_4\right| \lesssim 2^{k_{*}}.
\end{align*}
This puts us at risk of touching the resonant line from \eqref{eqn:resonant_line_discussion}. We then find that 
\begin{align*}
\eta_1+\eta_4+2\sqrt{3} &= 2\left(\eta_1+\sqrt{3}\right)+ \mathcal{O}\left(2^{k_*}\right) = \mathcal{O}\left(2^{k_1}\right)
\end{align*}
and we obtain the (possibly degenerate) lower bound 
\begin{equation}
    \label{eqn:case5_lower_derivative_bound}
    \left|\partial_{\eta_1}\varphi\right| \gtrsim 2^{k_1}\left|\eta_1-\eta_4\right|. 
\end{equation}
Next, localize $|\xi|$ again according to 
\begin{align*}
    |\xi| \sim 2^{k}, \quad 2^{-100m} \lesssim 2^{k} \lesssim 2^{k_{*}}, \quad \text{or} \quad  |\xi| \lesssim 2^{-100m}. 
\end{align*}
and then localize all input frequencies thus: 
\begin{align*}
    \left| \eta_1 - \eta_4\right|, |\eta_2+\eta_3|, |\eta_3+\eta_4| &\sim 2^{k}.
\end{align*} 
When $|\xi|\lesssim 2^{-100m}$ we obtain the claimed bound with the volume trick, and when $|\xi|$ is larger we integrate by parts using the following consequence of \eqref{eqn:case5_lower_derivative_bound}: $\left|\partial_{\eta_1}\varphi\right|\gtrsim 2^{k}2^{k_1}$. All remaining subcases of case 5 can be handled by following the arguments of subcase 5-3. For example, in the situation where 
\begin{align*}
\left|\eta_1+\eta_4+2\sqrt{3}\right|, \left|\eta_2+\eta_4\right| , \left|\eta_3+\eta_4\right|   &\lesssim 2^{k_{*}},
\end{align*}
we have 
\begin{align*}
    \left|\eta_1-\eta_4\right|,  \left|\eta_2-\eta_4-2\sqrt{3}\right|,  \left|\eta_3-\eta_4-2\sqrt{3}\right|  &\sim 2^{k_1}.
\end{align*}
Therefore, we still find $\left|\partial_{\eta_1}\varphi\right| \gtrsim 2^{k}2^{k_1}$, and we can comfortably follow subcase 5-3. The same ideas used to control $I_{m}$ in case 5 also work for $J_{m}$: to save space, we won't go through the details here. 
\begin{remark} The reader may desire some more motivation for the particular localization chosen in subcase 5-3. If we imagine that we're near the resonant line $L$ in each $\eta_j$, but not $\xi$, then we have
\begin{align*}
    \varphi &\approx -\omega(\xi) - \omega(\eta_1) + \omega(\xi+\eta_1)
    \\
    &\approx -\frac98\xi +\xi^3 +\frac{1}{32}\xi\left[\left(\xi+\eta_1+\sqrt{3}\right)^2+\left(\eta_1+\sqrt{3}\right)\left(\xi+\eta_1+\sqrt{3}\right)+\left(\eta_1+\sqrt{3}\right)^2\right].
\end{align*}
This implies that 
$$
\partial_{\eta_1}\varphi \approx \frac{3}{32}\xi \left(\xi+2\left(\eta_1+\sqrt{3}\right)\right),
$$
hence for $|\xi|\ll 2^{k_1}$ we have $\left|\partial_{\eta_1}\varphi\right|\gtrsim |\xi| 2^{k_1}$. Therefore, our lower bound on $\left|\partial_{\eta_1}\varphi\right|$ must include a piece directly proportional to $|\xi|$, and this motivates the choice $|\eta_1-\eta_4|\sim 2^{k}$.
\end{remark}

\subsection{Space-time Resonant Lines}
\label{ss:resonant_line}
\noindent 
We now obtain estimates for the contributions from the entire space-time resonant line in \eqref{eqn:resonant_line_discussion}. There are two important properties we'll leverage: 
\begin{itemize}
    \item First, since this line has $\xi=0$, the $\omega(\xi)$ out front gives us a bit of null structure. 
    \item Second, we've already dealt with the hardest cases $\eta \approx 0, \pm\sqrt{3}$. Therefore, we can specialize to the case where the line parameter $\eta$ is large and therefore nondegenerate. Of course, degeneracy at infinity needs to be carefully addressed, but this causes no big issues. 
\end{itemize}
In light of the second point, we restrict the line a little bit more. First, notice that we can re-write the resonant line as the mutual vanishing set of two independent degree 1 polynomials, $\eta_1+\eta_2$ and  $\eta_2-\eta_3$. Next, for each fixed time-interval index $m$, let $k_{\text{lo}}$ be the largest integer satisfying $2^{k_{\text{lo}}} \leq C_{\text{lo}} 2^{m\left(-\frac13 +\beta\right)}$, and let $\khi$ be the smallest integer such that $2^{\khi} \geq C_{\text{hi}} 2^{m\left(\frac19- \zeta\right)}$, where $\beta\ll 1$ is the same $\beta$ from our study of the $(0,0,0;0)$ resonance and $\zeta\approx \frac{1}{12}$ will be constrained as we go through the proof (since we've already carefully shown how to determine the ``growth perturbation'' parameters $\alpha, \beta, \gamma$ earlier on, we won't fuss over being precise about $\zeta$). Since we can basically cut out the degenerate cases $\eta=0,\pm\sqrt{3}$, to finish bounding contributions from $L$, it suffices only to bound the contributions from 
$$
L_{\text{hi}} = \left\{2^5 \leq |\eta_1|\leq 2^{\khi}, |\eta_1|\sim |\eta_2|\sim |\eta_3|, \ |\eta_1+\eta_2|\ll 1, |\eta_2-\eta_3|\ll 1, |\xi|\ll 1\right\}. 
$$
In practice, we work piece-by piece on each
$$
\hspace{-1cm}
L_{\text{hi}, kk_1\ell_1\ell_2} = \left\{|\eta_1|\sim 2^{k_{1}}\in \left[2^5, 2^{\khi}\right], \ |\eta_1+\eta_2|\leq 2^{\klo} \ \text{or} \ \sim 2^{\ell_1}, |\eta_2-\eta_3|\leq 2^{\klo} \ \text{or} \ \sim 2^{\ell_2}, |\xi|\leq 2^{2\klo} \ \text{or} \ \sim 2^{k}\right\},
$$
then sum the result over $k, k_1, \ell_1, \ell_2,$ and $m$ to get what we want. Note that we can cut off all nondegenerate frequencies of magnitude below $\approx 2^{5}$ without losing any important details: the constraint $2^{k_{1}}\in \left[2^5, 2^{\khi}\right]$ means that we may be close to the degenerate frequencies $\pm \infty$, so any argument that works for this frequency patch should carry over with trivial modifications to a frequency patch of finite extent. Additionally, we are allowed to cut off frequencies above $2^{\khi}$ using Sobolev embedding; I'll explain the details a little later. 
\par Before diving into the estimates, we perform a Taylor expansion to get some insight into the relationship between $k$ and the $\ell_j$'s. First, pick any $\overline{\eta}\in\mathbb{R}$ and expand the phase $\varphi$ about the space-time resonance $\left(\eta_1, \eta_2, \eta_3; \xi\right) = \left(-\overline{\eta}, \overline{\eta}, \overline{\eta}; 0\right)$ to obtain 
\begin{align*}
    \varphi&\approx -\left(1-\omega'\left(\overline{\eta}\right)\right)\xi + \frac12 \omega''\left(\overline{\eta}\right)\left( -\left(\eta_1+ \overline{\eta}\right)^2+  \left(\eta_2- \overline{\eta}\right)^2+  \left(\eta_3- \overline{\eta}\right)^2-\left(\eta_4+ \overline{\eta}\right)^2\right),
\end{align*}
If $\left|\overline{\eta}\right|\gtrsim 2^5$, then the above simplifies to $\varphi \approx -\xi + \mathcal{O}\left(\left(\eta_j\pm \overline{\eta}\right)^2\right)$. This tells us that we should demand the following: if $k\gtrsim k_{\text{min}}$ and $\ell_1, \ell_2 \gtrsim \ell_{\text{min}}$, then $k_{\text{min}} = 2 \ell_{\text{min}}$. This is similar to the localization choices we made when handling the anomalous resonance. Now, what should $\ell_{\text{min}}$ be? When handling the cases $\overline{\eta}=0, \pm\sqrt{3}$, we chose $\ell_{\text{min}} = \klo \approx -\frac{m}{3}$ as a reflection of weakened dispersive decay. Since the present case $|\overline{\eta}|\gg 1$ runs the risk of getting near the degenerate frequencies $\xi=\pm \infty$, we should also demand $\ell_{\text{min}} = \klo$ here. Also, without loss of generality, we always take $\ell_1\geq \ell_2$. 
\par\noindent
\textbf{Case 1} 
\begin{align*}
  \left|\eta_1+\eta_2\right|, \left|\eta_2-\eta_3\right| \lesssim 2^{\klo}, \quad |\xi| \lesssim 2^{2\klo}.
\end{align*}
Here, no integration by parts is allowed. Using the $L^2_{\xi}$ multilinear estimate right away appears to be the best option since $2^{k_1}\gtrsim 1$ means most terms in our integrand enjoy relatively fast dispersive decay. The relevant symbol here is 
$$
M = \partial_{\xi}\varphi \ \phi_{\lesssim 2\klo}(\xi) \ \phi_{\lesssim \klo}(\eta_1+\eta_2)\ \phi_{\lesssim \klo}(\eta_2-\eta_3) \ \psi_{k_{1}}(\eta_1)  \ \psi_{k_{2}}(\eta_2)  \ \psi_{k_{3}}(\eta_3). 
$$
Since $\partial_{\xi}\varphi$ and all its derivatives are bounded, we know immediately that $M$ satisfies the conditions of lemmas \ref{lemma:multiplier_cheat_code} and \ref{lemma:multiplier_cheat_code_secondary_localization} with $A=1$. Applying the multilinear estimate from proposition \ref{prop:multilinear_estimates} then gives
\begin{align*}
\left\|I_{...}\right\|_{L^2_{\xi}}\lesssim  \eps_1^4 \ 2^{2\klo} \ 2^{\frac92 k_1}\  2^{\frac{m}{2}},
\end{align*}
where to avoid cluttering the notation we've expressed the relevant LP piece of our integral simply as $I_{...}$. Summing over $m$ and the $k_j$'s gives 
$$
\sum_{m} \sum_{\text{freqs.}} \left\|I_{...}\right\|_{L^2_{\xi}}\lesssim \eps_1^4 t^{\frac13 + 2\beta-\frac92 \zeta}. 
$$
Thus, as long as $\frac13 + 2\beta-\frac92 \zeta \leq p_0$, all is well. Since $\beta=\frac{1}{50}$, we are free to choose $\zeta\approx \frac{1}{12}$ so the above bound is perfectly manageable. Note that, if we only allow the $k_j$'s to get so large (representing the non-degenerate frequencies near the resonant line), the same proof works. In fact, the argument is even easier since we don't have to keep track of the $2^{k_{1}}$'s arising from the dispersive estimate. 
\par For $J_m$, the multilinear estimate immediately gives 
\begin{align*}
\left\|J_{...}\right\|_{L^{\infty}_{\xi}} \lesssim \eps_1^4 2^{m\left(-\frac23+2\beta\right)}2^{3 k_1}  
\Rightarrow \sum_{\text{freqs.}} \left\|J_{...}\right\|_{L^{\infty}_{\xi}} \lesssim \eps_1^4 2^{m\left(-\frac13 + 2\beta - 3 \zeta\right)} .
\end{align*}
Since $\beta \ll \zeta$, the above sums over $m$ to get $\lesssim \eps_1^4. $
\par\noindent
\textbf{Case 2} 
\begin{align*}
  \left|\eta_1+\eta_2\right| &\sim 2^{\ell_1}, \quad  \ 2^{\klo} \lesssim 2^{\ell_1}<2^{-10}, \quad \text{or} \quad \left|\eta_1+\eta_2\right| \lesssim 2^{\klo},
  \\
  \left|\eta_2-\eta_3\right| &\sim 2^{\ell_2}, \quad  \ 2^{\klo} \lesssim 2^{\ell_2}<2^{-10},\quad \text{or} \quad \left|\eta_2-\eta_3\right| \lesssim 2^{\klo},
  \\
  |\xi|&\sim 2^{k}, \quad 2^{2\klo} \lesssim 2^{k}<2^{-10} \quad k \gg \ell_1.
\end{align*}
We only show details in the case $\left|\eta_2-\eta_3\right| \sim 2^{\ell_2} \gtrsim 2^{\klo}$ as the remaining case is similar but even easier. 
Performing a Taylor expansion immediately gives the lower bound
\begin{equation}
\label{eqn:case2_key}
    |\varphi| \gtrsim 2^{k}.
\end{equation}
Integrating by parts in time and using the multilinear estimate from proposition \ref{prop:multilinear_estimates} (checking the conditions of lemma \ref{lemma:multiplier_cheat_code_secondary_localization} is trivial since $k\gg \ell_1$) gives 
\begin{align*}
\left\|I_{...}\right\|_{L^2_{\xi}} \lesssim \eps_1^4 \left[ \ 2^{\frac92 k_1}2^{-\frac{m}{2}}+  2^{3k_{1}}2^{-\frac{m}{3}}\right]. 
\end{align*}
Summing over $m,k, k_1, k_2, k_3, \ell_1,$ and $\ell_2$ gives the desired bound. 
\par To cover the entire case for the relevant piece of $J_m$, it suffices to only show complete details for $\left|\eta_2-\eta_3\right| \sim 2^{\ell_2} \gtrsim 2^{\klo}$. We integrate by parts in time using \eqref{eqn:case2_key} and use the multilinear estimate to find
\begin{align*}
\left\|J_{...}\right\|_{L^\infty_{\xi}} \lesssim  \eps_1^4 2^{-m}2^{3k_1} \Rightarrow \sum_{\text{freqs.}} \left\|J_{...}\right\|_{L^\infty_{\xi}} \lesssim \eps_1^4 2^{m\left(-\frac23 -3\zeta\right)} \left(\log t\right)^3\lesssim \eps_1^4 2^{-\frac{2m}{3}},
\end{align*}
and summing over $m$ gives the claim. 
\par\noindent
\textbf{Case 3} 
\begin{align*}
  \left|\eta_1+\eta_2\right| &\sim 2^{\ell_1}, \quad  \ 2^{\klo} \lesssim 2^{\ell_1}<2^{-10}
  \\
  \left|\eta_2-\eta_3\right| &\sim 2^{\ell_2}, \quad  \ 2^{\klo} \lesssim 2^{\ell_2}<2^{-10}
  \\
  |\xi|&\sim 2^{k}, \quad \ 2^{2\klo} \lesssim 2^{k}<2^{-10}, \quad 2^k \lesssim 2^{\ell_2}, \quad \text{or} \quad |\xi|\lesssim 2^{\klo}. 
\end{align*}
As usual, we only show details for $|\xi|\sim 2^k$ since the remaining case is similar. Using the mean value theorem and $|\omega''(\eta_2)|\sim 2^{-3k_1}$, we get $| \partial_{\eta_2-\eta_3}\varphi| \gtrsim 2^{-3k_1}2^{\ell_2}$. We then have 
    \begin{align}
    \label{eqn:case3__semi_helpful_bounds}
        \left|\frac{\partial_{\xi}\varphi}{\partial_{\eta_2-\eta_3}\varphi}\right|\lesssim 2^{3k_1-\ell_2}, \quad \text{and} \quad 
         \left|\partial_{\eta_2-\eta_3}\left(\frac{\partial_{\xi}\varphi}{\partial_{\eta_2-\eta_3}\varphi}\right)\right|\lesssim 2^{3k_1-2\ell_2},
    \end{align}
which allows us to show that the relevant symbols satisfy the conditions of lemmas \ref{lemma:multiplier_cheat_code} and lemma \ref{lemma:multiplier_cheat_code_secondary_localization}. Additionally, 
    \begin{align}
\label{eqn:case3_helpful_bounds}
        \left|\omega(\xi)\frac{\partial_{\xi}\varphi}{\partial_{\eta_2-\eta_3}\varphi}\right|\lesssim 2^{3k_{1}}, \quad 
         \left|\omega(\xi)\partial_{\eta_2-\eta_3}\left(\frac{\partial_{\xi}\varphi}{\partial_{\eta_2-\eta_3}\varphi}\right)\right|\lesssim 2^{3k_{1}-\ell_2}. 
    \end{align}
Integrating by parts and using lemma \ref{lemma:multiplier_cheat_code} and proposition \ref{prop:multilinear_estimates} in conjunction with \eqref{eqn:case3_helpful_bounds} gives
\begin{align*}
   \left\|I_{...}\right\|_{L^2_{\xi}} &\lesssim \eps_1^4 \int_{t_1}^{t_2}\diff\tau \ 2^{(3+\frac92)k_1-\ell_2}2^{-\frac{3m}{2}} + 2^{3k_1}2^{m\left(p_0-\frac43\right)}
   \\
\Rightarrow
\sum_{m}\sum_{\text{freqs.}}  \left\|I_{mk_1k_2k_3k\ell_1\ell_2}\right\|_{L^2_{\xi}} &\lesssim \eps_1^4 \left(t^{\frac23 - \beta-7\zeta} 
 +t^{p_0} \right).
 \end{align*}
Accordingly, we must have $\frac23-\beta-7\zeta \leq p_0$. Since $p_0= \left(\frac16\right)^{-}$, $\beta\ll 1$, and $\zeta\approx \frac{1}{12}$, this bound is acceptable. 
\par For $J_m$, we show details just for $|\xi|\sim 2^k$. Using \eqref{eqn:case3_helpful_bounds} to integrate by parts in $\partial_{\eta_2-\eta_3}$ and applying the multilinear estimate, we find
\begin{align*}
   \left\|J_{...}\right\|_{L^2_{\xi}}
   &\lesssim \eps_1^4 \left[2^{6k_1}2^{m\left(p_0-1\right)} + 2^{6k_1-\ell_2}2^{-m}\right]
   \\
   \Rightarrow
    \sum_{\text{freqs.}}  \left\|J_{...}\right\|_{L^2_{\xi}} &\lesssim \eps_1^4 \left[2^{m\left(p_0-\frac13-6\zeta\right)}+ 2^{m\left(-\beta-6\zeta\right)}\right]\left(\log t\right)^3 \lesssim \eps_1^4 2^{-\frac{m}{6}},
\end{align*}
where we have used that $p_0< \frac16$ to obtain the last equality. Summing over $m$ then gives the claim. 
\subsection{Space-time Resonant Curves}\label{ss:res_curves}
\noindent Next, I explain how to apply the techniques from the previous subsection to control the contributions from the resonant curve $\Gamma$ from \eqref{eqn:resonant_curve_discussion}. Since we've already controlled the worst point on this curve, namely $\left(-\sqrt{3}, \sqrt{3}, \sqrt{3};0\right)$, the remaining cases should be straightforward. In keeping with the convention that $|\eta_1|$ is always larger than or equal to $|\eta_2|$ and $|\eta_3|$, it suffices to control all the pieces 
$$
\hspace{-1cm}
\Gamma_{\text{hi}, kk_1\ell_1\ell_2} = \left\{|\eta_1|\sim 2^{k_{1}}\in \left[2^5, 2^{\khi}\right], \ |r\left(\eta_1\right)-\eta_2|\leq 2^{\klo} \ \text{or} \ \sim 2^{\ell_1}, |\eta_2+\eta_3|\leq 2^{\klo} \ \text{or} \ \sim 2^{\ell_2}, |\xi|\leq 2^{2\klo} \ \text{or} \ \sim 2^{k}\right\}
$$
where $\klo, \khi$ are defined as in our investigation of the resonant line $L$. Note that this localization implicitly means that $|\eta_2|\sim2^{k_{2}}, |\eta_3|\sim 2^{k_{3}}$ only if $k_{2}, k_{3} \approx 0$. Accordingly, we can take $|\eta_4|\sim 2^{k_1}$
where we're working on each piece of the resonant curve. Additionally, note that following the discussion of the resonant line we can determine the correct relationship between the $k_j$'s and $\ell_j$'s by Taylor-expanding the phase about $\left(\overline{\eta}, r\left(\overline{\eta}\right), -r\left(\overline{\eta}\right);0\right)$. This yields
    $$
    \varphi \approx -\xi + \mathcal{O}\left(\left(\eta_1-\overline{\eta}\right)^2, \left(\eta_2-r\left(\overline{\eta}\right)\right)^2, \left(\eta_3+r\left(\overline{\eta}\right)\right)^2\right),
    $$
    so our localization choices should agree exactly with those in the previous subsection. This is excellent news, as it means that integration by parts in $\tau$ can give us a lot of analytical mileage. 
    \par Now, note that the analogues of cases 1 and 2 from the resonant line discussion can be handled easily when we move over to the resonant curve. To see why, notice that we now have the improved decay $\left\|u_{k_2}\right\|_{L^{\infty}_{x}}, \ \left\|u_{k_3}\right\|_{L^{\infty}_{x}} \lesssim \eps_1 2^{-\frac{m}{2}}$ since $|\eta_2|, |\eta_3|\approx 1$; that is, there are no factors of $2^{\frac{3k_1}{2}}$ to carry around on the right-hand side as in the resonant line. So, it seems like this case is even easier than the resonant line. However, when we flatten our coordinates according to 
\begin{equation}\label{eqn:coord_flattening}
\left(\eta_1, \eta_2, \eta_3\right) \mapsto \left(r(\eta_1), \eta_2, \eta_3\right)
\end{equation}
in order to follow the recipe of remark \ref{remark:nonlinear_loc}, we pick up a Jacobian term $|r'(\eta_1)|^{-1}\sim 2^{3k_{1}}$. Therefore, the bounds here are \emph{exactly} the same as those we saw when discussing the resonant lines.
\par It remains to investigate the analogue of case 3, given by 
\begin{align*}
  \left|r\left(\eta_1\right)-\eta_2\right| &\sim 2^{\ell_1}, \quad  \ 2^{\klo} \lesssim 2^{\ell_1}<2^{-10}
  \\
  \left|\eta_2+\eta_3\right| &\sim 2^{\ell_2}, \quad  \ 2^{\klo} \lesssim 2^{\ell_2}<2^{-10}, \quad 2^{\ell_1} \gtrsim 2^{\ell_2}, 
  \\
  |\xi|&\sim 2^{k}, \quad \ 2^{2\klo} \lesssim 2^{k}<2^{-10}, \quad 2^k  \lesssim 2^{\ell_2}, \quad \text{or} \quad  |\xi|\lesssim 2^{\klo}.
\end{align*}
In this case, we have $\left|\partial_{\eta_2}\varphi\right| \gtrsim 2^{\ell_1}$ since $|\eta_4|\sim 2^{k_1}$. This justifies an integration by parts in $\partial_{\eta_1}$. Once we perform this integration by parts and flatten our coordinates using \eqref{eqn:coord_flattening}, our symbol (including the absorbed Jacobian term) is $\lesssim 2^{3k_1}2^{-\ell_1}$. From here, we simply apply the arguments of case 3 in the resonant line case, taking advantage of improved decay in the $\eta_2$ and $\eta_3$ terms ($\ell_2\lesssim\ell_1$ is also required to check the conditions of lemma \ref{lemma:multiplier_cheat_code_secondary_localization}). Finally, I remark that the strategy we have just applied also works for $2^{\ell_1}\ll 2^{\ell_2}$ and $2^{k}\lesssim 2^{\ell_1}$, except that one integrates with respect to $\partial_{\eta_1}$ instead  of $\partial_{\eta_2}$. 
\subsection{Part IV: Weighted $L^2_{x}$ and $L^\infty_{\xi}$ Bounds in Non-Resonant Cases}
\noindent Throughout the discussion below, we set
    \begin{align}
        \left|\xi\right| \sim 2^{k}, \quad \left|\eta_{j}\right| \sim 2^{k_{j}}, 
    \end{align}
where $k, k_1, k_2, k_3 \in \mathbb{Z}$. We always take $k_1\geq k_2\geq k_3$ without loss of generality. Our goal is to show that, if we define
\begin{subequations}
\begin{align}
    I_{mk_1k_2k_3k} &\doteq \psi_{k}(\xi)\omega(\xi) \int_{t_1}^{t_2}\diff\tau \ \tau \int\diff\eta_1\diff\eta_2\diff\eta_3 \ e^{-i\tau\varphi} \partial_{\xi}\varphi \ \widehat{f}_{k_1}\left(\eta_1\right) \ \widehat{f}_{k_2}\left(\eta_2\right) \ \widehat{f}_{k_3}\left(\eta_3\right) \ \widehat{f}\left(\eta_{4}\right),
\\
\label{eqn:Jmk1k2k3defn}
    J_{mk_1k_2k_3k} &\doteq \psi_{k}(\xi)\omega(\xi) \int_{t_1}^{t_2}\diff\tau\int\diff\eta_1\diff\eta_2\diff\eta_3 \ e^{-i\tau\varphi}\ \widehat{f}_{k_1}\left(\eta_1\right) \ \widehat{f}_{k_2}\left(\eta_2\right) \ \widehat{f}_{k_3}\left(\eta_3\right) \ \widehat{f}\left(\eta_{4}\right),
\end{align}
\end{subequations}
then there are small numbers $b,c>0$ such that 
$$
\sum_{\substack{k_1, k_2, k_3\in \mathbb{Z} \\ k_{j{\geq \klo}}}}\left\| I_{mk_1k_2k_3k}\right\|_{L^2_{\xi}} \lesssim \eps_1^4 2^{-\frac12|k|} 2^{m\left(p_0-b\right)}
$$
with $p_0 = \left(\frac16\right)^{-}$ and 
$$
\sum_{\substack{k_1, k_2, k_3\in \mathbb{Z} \\ k_{j{\geq \klo}}}}\left\| J_{mk_1k_2k_3k}\right\|_{L^2_{\xi}} \lesssim \eps_1^4 2^{-\frac12|k|}2^{-mc}, 
$$
perhaps up to insignificant polylogarthmic factors. This way, we can sum over $k\in\mathbb{Z}$ and $m=0,1,..., m_{*}$ and get the claimed bounds. Let $k_{\text{lo}}$ be the largest integer satisfying $2^{k_{\text{lo}}} \leq C_{\text{lo}} 2^{m\left(-\frac13 +\beta\right)}$ and let $\khi$ be the smallest integer such that $2^{\khi} \geq C_{\text{hi}} 2^{m\left(\frac19- \zeta\right)}$,
where $\beta=\frac{1}{50}$ and $\zeta>\approx \frac{1}{12}$ (the same $\zeta$ used in subsections \ref{ss:resonant_line} and \ref{ss:res_curves}) will be constrained over the course of the problem. 
\subsubsection{Case 1: $2^{k}\geq 2^{\khi}$}
\noindent 
\textbf{Subcase 1.1: At least one $2^{k_j}\geq 2^{\khi}$}
\par\noindent 
Suppose without loss of generality that $2^{k_1}\geq 2^{\khi}$, and $2^{k_2}$, $2^{k_3}$ are $< 2^{\khi}$.  Since $\varphi$ and all its derivatives are bounded, the symbol
$$
M = \partial_{\xi}\varphi \ \psi_{k}(\xi) \psi_{k_1}(\eta_1) \psi_{k_2}(\eta_2) \psi_{k_3}(\eta_3)
$$
satisfies the hypotheses of lemma \ref{lemma:multiplier_cheat_code}. Using the $L^2_{\xi}$ multilinear estimate from proposition \ref{prop:multilinear_estimates} in conjunction with our linear dispersive estimate
$$
\left\|u_{k_j}\right\|_{L^{\infty}_{x}} \lesssim \eps_1 2^{-(s-1)k_{j}} 2^{mp_1} 
$$
and $|\omega(\xi)|\lesssim 2^{-k}$, we have (without integrating by parts!)
\begin{align*}
    \left\|I_{mk_1k_2k_3k}\right\|_{L^2_{\xi}} \lesssim 2^{-k}\eps_1^4\int_{t_1}^{t_2} \diff \tau \  2^{m} 2^{-(s-1)k_1} 2^{mp_1} \left(2^{-\frac{m}{3}}\right)^2
    \lesssim \eps_1^4  2^{-k} 2^{-(s-1)k_1} 2^{m\left(p_1+\frac43\right)}.
\end{align*}
At first glance, the $p_1+\frac43$ in the rightmost term seems like a source of trouble. However, using $s\geq 100\gg 1$ and $2^{k_1} \geq 2^{\khi} \gtrsim 2^{m\left(\frac19-\zeta\right)}$, we can find some $a,b>0$ such that 
\begin{align*}
    \left\|I_{mk_1k_2k_3k}\right\|_{L^2_{\xi}} \lesssim \eps_1^4 2^{-k} 2^{-ak_1}2^{-bm}. 
\end{align*}
If $k_2, k_3 > \klo$, then there are only $\mathcal{O}\left(m\right)\lesssim \mathcal{O}\left(\log t\right)$ terms in the sums over $k_2, k_3$, so in this case
\begin{align*}
    \sum_{k_1, k_2, k_3}\left\| I_{mk_1k_2k_3k}\right\|_{L^2_{\xi}} \lesssim \eps_1^4 2^{-k} m^2 2^{-bm} \lesssim \eps_1^4 2^{-k} 2^{-\frac{b}{2}m}, 
\end{align*}
which is stronger than what we wanted to show (but precisely what we want to show for $J_{mk_1k_2k_3k}$, which will be helpful below). If $k_2$ or $k_3$ or both are equal to $\klo$, the same argument works, with the only difference being that we don't collect as many factors of $m$ when summing. As for $J_{mk_1k_2k_3k}$, the same ideas yield
\begin{align*}
    \left\|J_{mk_1k_2k_3k}\right\|_{L^{\infty}_{\xi}} \lesssim \eps_1^4 2^{-k} 2^{m}  \left(2^{-(s-1)k_{1}} 2^{mp_1} \right) \left(2^{-\frac{m}{3}} \right) = \eps_1^4 2^{-k} 2^{-(s-1)k_1} 2^{m\left(p_1+\frac23\right)}\lesssim 2^{-k} 2^{-ak_1}2^{-bm}
    \end{align*}
 for some $a, b>0$, and summing over $k_1, k_2, k_3$ gives the required bound. 
\par \noindent 
\textbf{Subcase 1.2: $2^{k_j}\in \left(2^{\klo}, 2^{\khi}\right) \ \forall \ j$}
\par\noindent 
In this situation, $|\xi|\gg |\eta_j|$ for all $j$, hence $|\eta_4|\sim 2^{k}$. Using the same ideas applied to establish subcase 1.1, we get 
\begin{align*}
    \left\|I_{mk_1k_2k_3k}\right\|_{L^2_{\xi}} \lesssim 2^{-k}\eps_1^4\int_{t_1}^{t_2} \diff \tau \ 2^{m}\left(2^{m}\right)^{-\frac23} 2^{-k(s-1)}2^{mp_1}
    \lesssim \eps_1^4 2^{-sk}2^{m\left(p_1+\frac43\right)} 
\end{align*}
(the first two terms get the worst-case estimate, and the last term goes in $L^{\infty}_{x}$ too since its Fourier support is contained in $\left\{\gtrsim 2^{\khi}\right\}$). Summing over the $k_{j}$ only introduces a factor of $\mathcal{O}\left(m^3\right)$, so we're done using $s\gg 1$ again. The exact same argument allows us to control $\left\|J_{mk_1k_2k_3k}\right\|_{L^{\infty}_{\xi}}$. 

\noindent 
\textbf{Subcase 1.3: $2^{k_1}\in \left(2^{\klo}, 2^{\khi}\right) \ \forall \ j, \ \text{at least one of} \ 2^{k_2}, 2^{k_3} = 2^{\klo}$}
\par\noindent 
We only show details in the case $k_2 > \klo$: the main ideas are the same and the final bound is unaffected. The same ideas used in the previous two subcases together with $\left\|u_{k_3}\right\|_{L^{\infty}_x} \lesssim \eps_1 2^{k_3}$ (from the linear dispersive estimate) yield 
\begin{align*}
     \left\|I_{mk_1k_2k_3k}\right\|_{L^2_{\xi}} \lesssim \eps_1^4 2^{\frac53m}  2^{-sk}2^{k_{3}}.
\end{align*}
Summing over $k_1, k_2, k_3$ and using $s\gg 1$ gives us exactly the bound we require. Again, the same strategy works for bounding  $\left\|J_{mk_1k_2k_3k}\right\|_{L^{\infty}_{\xi}}$: just use $s\gg 1$ to overpower all the difficulties. 

\noindent 
\textbf{Subcase 1.4: $k_{j}= \klo\ \forall \ j$}
\par\noindent 
In this case we find that $|\eta_4|\sim 2^{k}$. We can then adapt subcase 1.1 to get what we need. 

\subsubsection{Case 2: $2^{k}\in \left[2^5, 2^{\khi}\right)$}
\noindent 
\textbf{Subcase 2.1: At least one $2^{k_j}\geq 2^{\khi}$}
\par\noindent 
This is a direct copy of subcase 1.1. 
\par \noindent 
\textbf{Subcase 2.2: $2^{k_1}, 2^{k_2}, 2^{k_3}\in \left[2^5, 2^{\khi}\right)$, $2^k\gg 2^{k_1}$}
\par\noindent
In this situation, we know that $|\eta_4|\approx |\xi| \sim 2^{k}$. Accordingly, Taylor expanding $\omega'(\eta_1)$ and $\omega'(\eta_4)$ about $\pm \infty$ gives $\partial_{\eta_1}\varphi \gtrsim 2^{-2k_1}$. Additionally, we have the null structure
\begin{equation}\label{eqn:null_structure_at_infty_if_xi_eta4_comparable}
|\partial_{\xi}\varphi|\lesssim 2^{-2k} \quad \text{when} \quad |\eta_4|\sim 2^{k}. 
\end{equation}
Combining these expressions gives 
\begin{align*}
    \left|\frac{\partial_{\xi}\varphi}{\partial_{\eta_1}\varphi}\right|\lesssim 1, \quad \left|\partial_{\eta_1}\left(\frac{\partial_{\xi}\varphi}{\partial_{\eta_1}\varphi}\right)\right|\lesssim 1. 
\end{align*}
Integrating by parts and using lemma \ref{lemma:multiplier_cheat_code} to trivially verify the conditions of the $L^2_{\xi}$ multilinear estimate from proposition \ref{prop:multilinear_estimates}, we find
\begin{align*}
    \left\|I_{mk_1k_2k_3k}\right\|_{L^2_{\xi}} &\lesssim \eps_1^4 2^{-k}\left[2^{m\left(p_0-\frac13\right)}2^{\frac{3k_2}{2}}2^{\frac{3k_3}{2}}+ 2^{m\left(p_0-\frac12\right)}2^{\frac{3k_1}{2}}2^{\frac{3k_2}{2}}2^{\frac{3k_3}{2}}\right]
\\ \Rightarrow 
    \left\|I_{mk_1k_2k_3k}\right\|_{L^2_{\xi}} &\lesssim  \eps_1^4 2^{-k} 2^{m\left(p_0-2\zeta\right)}
\end{align*}
as desired. Now, we deal with $J_{mk_1k_2k_3k}$. Using our existing symbol bounds and the $L^{\infty}_{\xi}$ multilinear estimate from proposition \ref{prop:multilinear_estimates_Linfty} gives
\begin{align*}
    \left\|J_{mk_1k_2k_3k}\right\|_{L^{\infty}_{\xi}} &\lesssim 
   \eps_1^4 2^{-k} 2^{\frac32 \left(k_2+k_3\right)} 2^{m\left(p_0-1\right)}
\\ \Rightarrow 
       \left\|J_{mk_1k_2k_3k}\right\|_{L^{\infty}_{\xi}} &\lesssim \eps_1^4 2^{-k} 2^{-3\zeta m},
 \end{align*}
 which is what we want. 
\par \noindent 
\textbf{Subcase 2.3: $2^{k_1}, 2^{k_2}, 2^{k_3}\in \left[2^5, 2^{\khi}\right)$, $2^{k_1}\approx 2^{k_3}\approx 2^{k_3}$, $2^{k}\lesssim 2^{k_1}$}
\par\noindent
This case is a little delicate because it includes situations that are space-resonant and nearly time-resonant owing to the vanishing of $\varphi$ at infinity. Additionally, the best we can say about $\eta_4$ is $0\leq |\eta_4| \lesssim 2^{k}$, so $\eta_4$ is not really ``localized''. To handle this, we divide $\eta_4$-space into three sections using one vanilla bump and two annular bumps: 
$$
1 \equiv \phi_{\lesssim 2^{-1}}\left(\eta_4\right) + \psi_{\approx \left(2^{-1}, 2^3\right]}\left(\eta_4\right) +  \psi_{\approx \left(2^3, 2^{\khi}\right]}\left(\eta_4\right)
$$
We now proceed region-by-region. 
\begin{itemize}
    \item In the support of $\phi_{\lesssim 2^{-1}}\left(\eta_4\right)$, we find that $|\partial_{\eta_1}\varphi|\gtrsim 1$ by Taylor expansion, and the proof reduces to an easier version of subcase 2.2 (the symbol bounds in each case are identical). 
\item In the support of $\psi_{\approx \left(2^{-1}, 2^3\right]}\left(\eta_4\right)$, we run the risk of having $|r(\eta_j)| \approx |\eta_4| \ \forall j=1,2,3$. If the above expression holds, then we are nearly space resonant. For example, we may be near a space resonance of the form $\left(\eta, \eta, \eta; 3\eta+r(\eta)\right)$. However, from our earlier resonance computations in subsection \ref{ss:resonance_computation}, we know that these are actually \emph{pure} space resonances: $|\varphi| \gtrsim 1$,
since $|\xi|>0$. We may then integrate by parts in $\tau$ and use lemma \ref{lemma:multiplier_cheat_code} and the $L^2_{\xi}$ multilinear estimate from proposition \ref{prop:multilinear_estimates} to obtain 
$$
\left\|I_{mk_1k_2k_3k}\right\|_{L^2_{\xi}} \lesssim \eps_1^4 2^{-k}.
$$
Notice how we put the $\partial_{\tau}f$'s in $L^2_x$ and always used the worst-case decay to estimate the $L^\infty_x$ norms: this choice allows us to trivialize a bunch of other cases later on.  Summing over the $k_j$'s gives
$$
\sum_{\text{freqs.}}\left\|I_{mk_1k_2k_3k}\right\|_{L^2_{\xi}}\lesssim \eps_1^4 2^{-k} 2^{mp_0}
$$
The same ideas applied to $J_{mk_1k_2k_3k}$ give
$$
\left\|J_{mk_1k_2k_3k}\right\|_{L^\infty_{\xi}} \lesssim \eps_1^4 2^{-k}2^{-\frac{2m}{3}} \Rightarrow \sum_{\text{freqs.}}\left\|J_{mk_1k_2k_3k}\right\|_{L^\infty_{\xi}} \lesssim \eps_1^4 2^{-k}2^{-\frac{m}{3}}. 
$$
So, we're done in the case where $|r(\eta_j)|\approx |\eta_4|$ for all $j$. 

\par Alternatively, if at least one $|r(\eta_j)|$ is relatively far from $|\eta_4|$, we instead integrate by parts in the direction $\partial_{\eta_j}$ to prove the claim. The full details for how to properly localize around $|r(\eta_j)|= |\eta_4|$ have already been handled in the much more difficult space-time resonant cases (see subsection \ref{ss:anomalous}). 
\item Finally, in the support of $\psi_{\approx \left(2^3, 2^{\khi}\right]}\left(\eta_4\right)$, we may have that $|\eta_j| \approx |\eta_4| \ \forall  \ j=1,2,3$. This occurs, for example, near the space resonances $\left(\eta, \eta, \eta; 2\eta\right)$ or $\left(\eta, \eta, \eta; 4\eta\right)$. Again appealing to our earlier resonance computations in subsection \ref{ss:resonance_computation}, we find that near these space resonances $\varphi \sim |\xi|^{-1}$. This gives rise to the lower bound $|\varphi|\gtrsim 2^{-k}$, which initially seems like bad news. However, we know that $|\eta_4|\sim 2^{k}$ in this worst-case space-resonant scenario, so we still have the null structure $|\partial_{\xi}\varphi|\lesssim 2^{-2k}$. Therefore, $\left|\frac{\partial_{\xi}\varphi}{\varphi}\right| \lesssim 2^{-k} \lesssim 1$ and we reduce to the arguments of the previous item. When $
|\eta_j| \gg |\eta_4|$ for some $j$ (or vice versa), we may integrate by parts in $\partial_{\eta_j}$ to get what we need. 
\end{itemize} 

\par \noindent 
\textbf{Subcase 2.4: $2^{k_1}, 2^{k_2}, 2^{k_3}\in \left[2^5, 2^{\khi}\right)$, $2^{k_j}\gg 2^{k_{j'}}$ for some $j<j'$, $2^{k}\lesssim 2^{k_1}$}
\par\noindent
Without loss of generality suppose $2^{k_{1}}\gg 2^{k_2}$. Since our localization prohibits situations like $\eta_2\approx r(\eta_1)$, in order to be near a space resonance we must have $|\eta_1|\approx |\eta_2| \approx |\eta_3|$, which is not true here. We therefore know that at least one $|\partial_{\eta_j}\varphi|$ is bounded below, and the problem reduces to identifying which component this is. This is accomplished by localizing $\eta_4$ exactly as in subcase 2.3. 
\begin{itemize}
    \item The argument covering the support of  $\phi_{\lesssim 2^{-1}}\left(\eta_4\right)$ is identical to that presented in the previous subcase. 
    \item  In the support of $\psi_{\approx \left(2^{-1}, 2^3\right]}\left(\eta_4\right)$, we may have $|r(\eta_1)| \approx |\eta_4|$. However, if this holds, then $|\eta_4| \ll |r(\eta_2)|$,
    so we are clear to integrate by parts in the direction $\partial_{\eta_2}$; a similar argument shows that integration by parts in $\partial_{\eta_1}$ is feasible when $|r(\eta_2)| \approx |\eta_4|$. 
    \item The support of $\psi_{\approx \left(2^3, 2^{\khi}\right]}\left(\eta_4\right)$ can be handled by mimicking the ideas of the previous item: just replace $r(\eta_j)$ with $\eta_j$. 
\end{itemize}

\par \noindent 
\textbf{Subcase 2.5: $2^{k_1}, 2^{k_2}\in \left[2^5, 2^{\khi}\right), 2^{k_3}\in \left(2^{\klo}, 2^{-1}\right]$ or $k_3=\klo$}
\par\noindent
By Taylor expansion, we immediately verify that $|\partial_{\eta_1-\eta_3}\varphi| \gtrsim 1.$
We then get the symbol bounds 
$$
\left|\frac{\partial_{\xi}\varphi}{\partial_{\eta_1-\eta_3}\varphi}\right| \lesssim 1, \quad \left|\partial_{\eta_1-\eta_3}\left(\frac{\partial_{\xi}\varphi}{\partial_{\eta_1-\eta_3}\varphi}\right)\right| \lesssim 2^{\max\left\{-3k_1, k_3\right\}} \lesssim 1. 
$$
Suppose first that $k_3=\klo$. We may verify the multilinear estimate hypotheses for the above symbols using lemma \ref{lemma:multiplier_cheat_code}. Additionally, recall that 
\begin{align*}
    \left\|\partial_{\eta}\left(\phi_{\klo}(\eta)\widehat{f}(\eta)\right)\right\|_{L^2_{\eta}} \leq \left\|\phi_{\klo}\partial_{\eta}\widehat{f}\right\|_{L^2_{\eta}}  + 2^{-\frac12 \klo}\left\|\widehat{f}\right\|_{L^{\infty}_{\eta}}.
\end{align*}
Under our bootstrap assumptions, this implies
$$
\left\|\partial_{\eta}\left(\phi_{\klo}(\eta)\widehat{f}(\eta)\right)\right\|_{L^2_{\eta}} \lesssim \eps_1 2^{m \max\left\{p_0, \frac16-\frac12\beta\right\}}. 
$$
 In particular, the above tells us that we must pick $\beta$ large enough to ensure $\frac16-\frac12\beta\leq p_0$. With all this in mind, integrate by parts to get
\begin{align*}
    \left\|I_{mk_1k_2\klo k}\right\|_{L^2_{\xi}} &\lesssim \eps_1^4 2^{-k} \int_{t_1}^{t_2}\diff\tau  \ 2^{\frac32\left(k_1+k_2\right)} 2^{-\frac{4m}{3}} + 2^{\frac32 k_2}2^{m\left(p_0-\frac76\right)} +  2^{\frac32\left(k_1+k_2\right)}2^{m\left(p_0-\frac43\right)}
\end{align*}
where we have used $\left\|u_{\leq 2^{\klo}}(\tau)\right\|\lesssim \eps_1 2^{-\frac{m}{3}}$. 
Summing gives
\begin{align*}
   \sum_{k_1, k_2} \left\|I_{mk_1k_2\klo k}\right\|_{L^2_{\xi}} &\lesssim \eps_1^4 2^{-k}  \left(2^{-3m\zeta} + 2^{m\left(-\frac32\zeta + p_0\right)}\log t + 2^{m\left(-3\zeta + p_0\right)}\right) \lesssim \eps_1^4 2^{-k}2^{mp_0}. 
\end{align*}
On the other hand, if $k_3>\klo$, then by the same arguments we have 
\begin{align*}
    \left\|I_{mk_1k_2k_3k}\right\|_{L^2_{\xi}} &\lesssim \eps_1^4 2^{-k} \int_{t_1}^{t_2}\diff\tau  \ 2^{\frac32\left(k_1+k_2\right)-\frac12 k_3} 2^{-\frac{3m}{2}} + 2^{\frac32 k_2-\frac12 k_3}2^{m\left(p_0-\frac43\right)} +  2^{\frac32\left(k_1+k_2\right)}2^{m\left(p_0-\frac43\right)}
\end{align*}
(notice how we used different  decay estimates for the $\eta_1$, $\eta_2$ terms versus the $\eta_3$ term), which still sums to 
$$
 \sum_{k_1, k_2, k_3} \left\|I_{mk_1k_2k_3k}\right\|_{L^2_{\xi}} \lesssim \eps_1^4 2^{-k} 2^{mp_0}, 
 $$
 for $\zeta$ sufficiently large. As for $J_{mk_1k_2k_3k}$, a similar approach gives 
 \begin{align*}
     \sum_{\text{freqs.}} \left\|J_{mk_1k_2k_3k}\right\|_{L^{\infty}_{\xi}} \lesssim \eps_1^4  2^{-k} 2^{-am}
 \end{align*}
 for some $a>0$. 
\par \noindent 
\textbf{Subcase 2.6: $2^{k_1}, 2^{k_2}\in \left[2^5, 2^{\khi}\right), 2^{k_3}\in\left[2^{-1}, 2^5\right], \ 2^k\gg 2^{k_1}$}
\par\noindent
Since $|\eta_4|\sim 2^{k}$, Taylor expansion gives $|\partial_{\eta_1}\varphi|\gtrsim 2^{-2k_1}$ and we also have the null structure $|\partial_{\xi}\varphi|\lesssim 2^{-2k_1}$. We then get the symbol bounds 
$$
\left|\frac{\partial_{\xi}\varphi}{\partial_{\eta_1}\varphi}\right| \lesssim 1, \quad \left|\partial_{\eta_1}\left(\frac{\partial_{\xi}\varphi}{\partial_{\eta_1}\varphi}\right)\right| \lesssim 2^{-k_1} \lesssim 1,
$$
and reduce to subcase 2.2. 

\par \noindent
\textbf{Subcase 2.7: $2^{k_1}\in \left[2^5, 2^{\khi}\right], \ 2^{k_2}, 2^{k_3}\in\left(2^{\klo}, 2^{-1}\right]$ or $k_2, k_3 = \klo$}
\par\noindent 
As in subcase 2.5, $|\partial_{\eta_1-\eta_3}\varphi|\gtrsim 1$. However, this time we must be more careful, either relating $k$ and $k_1$ or leveraging some extra dispersive decay in the $\eta_4$ term. 
\par First, suppose $k\gtrsim k_1$. Here, $|\eta_4|$ is difficult to control (it can be $\sim 2^{k}$ or very small), so we avoid venturing down that road. We have the same symbol bounds from subcase 2.5,
$$
\left|\frac{\partial_{\xi}\varphi}{\partial_{\eta_1-\eta_3}\varphi}\right| \lesssim 1, \quad \left|\partial_{\eta_1-\eta_3}\left(\frac{\partial_{\xi}\varphi}{\partial_{\eta_1-\eta_3}\varphi}\right)\right| \lesssim 1,
$$
and because there is no null to come to our rescue the situation seems hopeless (strictly speaking, the symbol bounds could be sharpened when $k\gg k_1$ to get enough extra decay to work, but we want a trick that also works for $k\approx k_1$). However, using the trivial observation $2^{-k} = 2^{-\frac12 k}2^{-\frac12 k} \lesssim 2^{-\frac12 k} 2^{-\frac12 k_1}$, we discover 
\begin{align*}
    \left\|I_{k_1k_2k_3k}\right\|_{L^2_{\xi}} &\lesssim \eps_1^4 2^{-\frac12 k} \int_{t_1}^{t_2}\diff\tau  \ 2^{k_1}2^{-\frac{7m}{6}} + 2^{-\frac12 k_1}2^{m\left(p_0-1\right)} +  2^{k_1}2^{m\left(p_0-\frac76\right)}
\end{align*}
which sums over $m, k_1$ to get $\lesssim \eps_1 t^{p_0} 2^{-\frac12 k}$.  Notice how tight the bound is on the term where $\partial_{\eta_1}$ hits $\widehat{f}\left(\eta_1\right)$: without our splitting $2^{-k}$ trick, our bound would be off by a power of $m$.
\par Next, suppose $k_1 \gg k$, but we still have $k_2=k_3=\klo$. This means $|\eta_4|\sim 2^{k_1}$. Consequently, we have some null structure,
$$
|\partial_{\xi}\varphi| \lesssim 2^{-2k},
$$
but it is not very useful since it's not expressed in terms of $k_1$, the index we want to sum over. In particular, the best symbol bounds we can get are
$$
\left|\frac{\partial_{\xi}\varphi}{\partial_{\eta_1-\eta_3}\varphi}\right| \lesssim 1, \quad \left|\partial_{\eta_1-\eta_3}\left(\frac{\partial_{\xi}\varphi}{\partial_{\eta_1-\eta_3}\varphi}\right)\right| \lesssim 1. 
$$
Also, the ``split $2^{-k}$'' trick we've just used will not work. However, now the $\eta_4$ term inside $I_{k_1k_2k_3k}$ undergoes faster decay:
$$
\left\|\left[\left(\text{annular bump supported on freqs. of size} \ \approx 2^{k_1}\right)\times \widehat{f}\left(\eta_4\right)\right]^{\vee}(\tau)\right\|_{L^{\infty}_{x}} \lesssim \eps_1 2^{\frac32 k_1} 2^{-\frac{m}{2}}.
$$
Using this extra decay in the term where $\partial_{\eta_1}$ hits $\widehat{f}\left(\eta_1\right)$, we find 
\begin{align*}
    \left\|I_{k_1k_2k_3k}\right\|_{L^2_{\xi}} &\lesssim \eps_1^4 2^{-k} \int_{t_1}^{t_2}\diff\tau  \ 2^{\frac32k_1} 2^{-\frac{7m}{6}} + 2^{\frac32k_1}2^{m\left(p_0-\frac76\right)} 
\end{align*}
which sums over $k_1$ to give $\lesssim \eps_1 t^{p_0} 2^{-k}$ for $\zeta$ large enough. The cases where $k_2 > \klo$ or $k_2, k_3 > \klo$ are similar but easier: for whatever variable has frequencies over $2^{\klo}$, use the  dispersive decay in $\left(2^{\klo}, 2^{-1}\right]$ to make the sums work out. For the $L^{\infty}_{\xi}$ piece $J_{mk_1k_2k_3k}$, we don't have to worry about the ``split-$2^{-k}$'' trick and can essentially just copy subcase 2.5. 

\par \noindent
\textbf{Subcase 2.8: $2^{k_1}\in\left[2^5, 2^{\khi}\right], 2^{k_2}\in \left[2^{-1}, 2^5\right], 2^{k_3}\in \left(2^{\klo}, 2^{-1}\right]$ or $k_3=\klo$}
\par\noindent 
When $k\gtrsim k_1$, this is a carbon copy of arguments we saw in subcase 2.7 (one high frequency, two degenerate frequencies). If $k_1\gg k$, note in particular that this means $k_1\gg k_2$ as well: this is because $2^{k}\in \left[2^5, 2^{\khi}\right)$, so the largest $2^{k_{2}}$ can get is $\approx 2^{k}$. Consequently, we find $|\eta_4| \approx 2^{k_1}$ and again reduce to subcase 2.7. 

\par \noindent 
\textbf{Subcase 2.9: $2^{k_1}\in \left[2^5, 2^{\khi}\right), 2^{k_2} \in \left[2^{-1}, 2^{\khi}\right), 2^{k_3}\in\left[2^{-1}, 2^5\right]$, $2^k\gg 2^{k_1}$}
\par\noindent
From the arguments of subcase 2.6, we still have $|\partial_{\eta_1}\varphi|\gtrsim 2^{-2k_1}$ and $|\partial_{\xi}\varphi|\lesssim 2^{-2k_1}$. Further, we also still have access to the ``split $2^{-k}$'' trick from subcase 2.7, so we're all done. 

\par \noindent 
\textbf{Subcase 2.10: $2^{k_1}, 2^{k_2}\in \left[2^5, 2^{\khi}\right), 2^{k_3}\in\left[2^{-1}, 2^5\right]$,  $2^k\lesssim 2^{k_1}$}
\par\noindent
This case is quite interesting because we can have situations like $\eta_3\approx \pm r(\eta_1)$ giving rise to potential space resonances. Of course, such situations also appeared in subcase 2.9, but there we could bulldoze over them using $|\eta_4|\sim 2^k\gg 2^{k_{1}}$. To add even more complexity, $\eta_3$ may be degenerate: $\eta_3\approx \pm \sqrt{3}$ is allowed here. These problems must be handled by chopping up $\left[2^{-1}, 2^5\right]$, as we chopped up $\eta_4$-space in subcases 2.3 and 2.4. First, pick $\delta<1$ such that $\delta < 2^{-10}\left|\sqrt{3}-r\left(2^4\right)\right|$. So, while $\eta_3\in \left[\sqrt{3}-\delta, \sqrt{3}+\delta\right]$, we cannot possibly have $\eta_3\approx \pm r(\eta_1)$. Then, we chop up our interval using bump functions according to
$$
1\equiv \psi_{\approx \left[2^{-1}, \sqrt{3}-\delta\right]}(\eta_3) + \phi_{\lesssim \delta} \left(|\eta_3|-\sqrt{3}\right) + \psi_{\approx \left[\sqrt{3}+\delta, 2^5\right]}(\eta_3).
$$
To keep the argument as short as possible, let's define
\begin{itemize}
    \item Region I to be the support of $\psi_{\approx \left[2^{-1}, \sqrt{3}-\delta\right]}$
    \item Region II to be the support of $\phi_{\lesssim \delta} \left(|\cdot|-\sqrt{3}\right) $, and 
    \item Region III to be the support of $\psi_{\approx \left[\sqrt{3}+\delta, 2^5\right]}$. 
\end{itemize}
We then prove the desired bound by working on each piece individually. 
\begin{itemize}
    \item When $\eta_3$ is in Region I, we may encounter space resonances that feature $|\eta_3| = |r(\eta_1)|$ or $|r(\eta_2)|$. To get around this issue, we mimic subcases 2.3 and 2.4: by going back to the computations of subsection \ref{ss:resonance_computation}, we find that these are pure space resonances, and an integration by parts in $\tau$ will do the trick. The non-space resonant cases can be handled via integration by parts in $\partial_{\eta_1-\eta_3}$ or  $\partial_{\eta_2-\eta_3}$. 
    \item Region II can also be handled with arguments described in subcases 2.3 and 2.4. 
\item When $\eta_3$ is in Region III, the reflection will not cause problems, but degeneracy will. To remedy this degeneracy, localize $\eta_3$ about $\pm \sqrt{3}$ again at a dyadic scale $2^{\ell}$, $\ell\geq\ell_0,$ with $2^{\ell_0} \approx 2^{\klo}$. Next, observe by Taylor expansion that $|\partial_{\eta_1-\eta_3}\varphi| \approx \left|-\frac18 + \mathcal{O}\left(2^{-2k}, 2^{2\ell}\right) \right|\gtrsim 1$. From here, the proof is a routine modification of earlier work. Note that the dyadic decomposition really is required because of potential issues that can arise when the derivative $\partial_{\eta_3}$ hits $\widehat{f}_{k_3}\left(\eta_3\right)$.
\end{itemize}

\par \noindent 
\textbf{Subcase 2.11: $2^{k_1}\in \left[2^5, 2^{\khi}\right), 2^{k_2}, 2^{k_3}\in\left[2^{-1}, 2^5\right]$, $2^{k}\lesssim  2^{k_1}$ }
\par\noindent
This is a straight generalization of the ideas of subcase 2.10, and we define $\delta>0$ exactly as we did there. To start, we chop up the relevant $\eta_2, \eta_3$ intervals using soft bumps:
\begin{align*}
1\equiv \psi_{\approx \left[2^{-1}, \sqrt{3}-\delta\right]}(\eta_3) + \phi_{\lesssim \delta} \left(|\eta_3|-\sqrt{3}\right) + \psi_{\approx \left[\sqrt{3}+\delta, 2^5\right]}(\eta_3).
\\ 
1\equiv \psi_{\approx \left[2^{-1}, \sqrt{3}-\delta\right]}(\eta_2) + \phi_{\lesssim \delta} \left(|\eta_2|-\sqrt{3}\right) + \psi_{\approx \left[\sqrt{3}+\delta, 2^5\right]}(\eta_2).
\end{align*}
We also use the Region I, II, III notation established in subcase 2.10. 
Using our convention that $k_2\geq k_3$, then, there are really only four cases to consider.
\begin{itemize}
    \item If $\eta_2$ is in Region III, then it is large enough for us to effectively reduce to subcase 2.10.
    \item If both $\eta_2$ and $\eta_3$ are in Region I, then roughly speaking we're in a ``$k_1\gg k_2\approx k_3$'' situation. Since $\eta_2, \eta_3$ are nondegenerate, we thus reduce to an easier version of subcase 2.4; along the way, one will have to account for space resonances with
$$
\left(\eta_1, \eta_2, \eta_3\right) = \left(\eta, r(\eta), -r(\eta)\right),
$$
but these are easily handled since our work in subsection \ref{ss:resonance_computation} says they are not time-resonant. 
    \item If $\eta_2$ is in Region II and $\eta_3$ is in Region I, then after localizing $\eta_2$ at a dyadic scale near $\pm \sqrt{3}$, we have $|\partial_{\eta_1-\eta_2}\varphi| \gtrsim 1$. With this bound, we reduce to an easier variant of subcase 2.5 (easier because we have only finitely many $k_3$'s to sum over).
    \item If $\eta_2, \eta_3$ are both in Region II, we still have $\left|\partial_{\eta_1-\eta_3}\varphi\right|\gtrsim 1$ (we could also switch $\eta_3$ for $\eta_2$!) and thereby reduce to subcase 2.7 after localizing about $\pm \sqrt{3}$ in both $\eta_2$ and $\eta_3$.  
\end{itemize}

\par \noindent
\textbf{Subcase 2.12: $2^{k_1}, 2^{k_2}, 2^{k_3}\in \left[2^{-1}, 2^5\right]$}
\par\noindent 
We again adopt the splitting philosophy used in subcases 2.10 and 2.11, keeping the exact same notation. 
\begin{itemize}
    \item If all $\eta_j$ are in Region III, then we simply copy subcase 2.2, 2.3, or 2.4 depending on their relative sizes to get what we want.
    \item If $\eta_1, \eta_2$ are in Region III, then we apply subcase 2.6 or 2.10. 
    \item If $\eta_1$ alone is in Region III, then we apply subcase 2.6 or 2.11. 
    \item If all $\eta_j$ are in Region II (that is, they are all close to the nonzero degnerate frequencies $\pm \sqrt{3}$), then we have a little more work to do. Let $\sigma_{j}=\pm 1$, then define
    \begin{equation}
        \label{eqn:bigOmega}
        \Omega_{\sigma_1\sigma_2\sigma_3}(\xi) \doteq \frac{\sqrt{3}}{4}\sum_{j=1}^{3}\sigma_j  -\omega(\xi) + \omega\left(\xi-\left(\sum_{j=1}^{3}\sigma_j\right) \sqrt{3}\right).
    \end{equation}
    Note that $\Omega_{-\sigma_1-\sigma_2-\sigma_3}(\xi) = -\Omega_{\sigma_1\sigma_2\sigma_3}(-\xi)$, so to understand the behaviour of the $\Omega$'s it suffices to consider $\sigma_1\sigma_2\sigma_3 = +++, -++$ only. See figure \ref{fig:Omega} for a plot of these two functions: note in particular that, as $|\xi|\rightarrow \infty$, $\Omega_{\sigma_1\sigma_2\sigma_3}$ converges to $\frac{\sqrt{3}}{4}\sum_{j=1}^{3}\sigma_{j} \neq 0.$ Now, observe by Taylor expansion that $\varphi = \Omega_{\sigma_1\sigma_2\sigma_3}(\xi) + \mathcal{O}\left(\delta\right)$. By shrinking $\delta$ if necessary, the above becomes $|\varphi|\gtrsim 1$ in this region. From here, the problem is covered by the arguments of subcase 2.3. 
     \begin{figure}
\centering
\begin{tikzpicture}[scale=1.1]
\begin{axis}[xmin=-10, xmax=13, ymin=-0.65, ymax=1.8, samples=300, no markers,
      axis x line = middle,
      axis y line = middle,
    ylabel=\Large{$\Omega_{\sigma_1\sigma_2\sigma_3}(\xi)$},
    xlabel=\Large{$\xi$},
    every axis x label/.style={
    at={(ticklabel* cs:1.05)},
},
every axis y label/.style={
    at={(ticklabel* cs:1.05)},
    anchor=south,
},
yticklabels={,,-0.5,},
xticklabels={,-10,-5,,,},
legend pos= south east,
]
  \addplot[indigo(web), ultra thick, domain=-10:13]  {0.25*1.73205080757*(3)-x/(1+x*x) +(x-5.19615242271
)/(1+(x-5.19615242271
)*(x-5.19615242271
))};
\addlegendentry{$+++$}
 \addplot[lightseagreen, ultra thick, domain=-10:13, dashed]  {0.25*1.73205080757*(1)-x/(1+x*x) +(x-1.73205080757
)/(1+(x-1.73205080757
)*(x-1.73205080757
))};
\addlegendentry{$-++$}
\end{axis}
\end{tikzpicture}
\caption{The two important $\Omega_{\sigma_1\sigma_2\sigma_3}$ used in subcase 2.12. The two roots of the $-++$ curve are $\xi=0, \sqrt{3}$.}
\label{fig:Omega}
\end{figure}
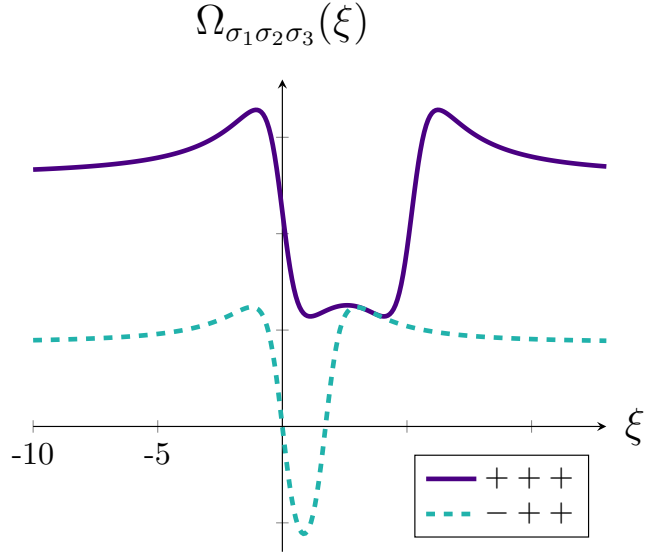
    \item Suppose $\eta_1$ is in Region II and at least one of $\eta_2,\eta_3$ are Region I. Here, we necessarily have $2^{k_1}\ll 2^{k}$, hence $|\partial_{\eta_1}\varphi| \gtrsim 1$ since $\omega'(\eta_1)\approx \omega'(\sqrt{3})\neq 0$. With this bound in hand, the rest of the proof amounts to redoing an easier version of subcase 2.6 or subcase 2.7, depending on whether or not $\eta_2$ is in Region II or Region I (localization about $\pm\sqrt{3}$ for any frequencies in Region II will be required). 
    \item If all $\eta_j$ are in Region I, we still have $|\partial_{\eta_1}\varphi|\gtrsim 1$ from the previous item. Since all frequencies $\eta_j$ are nondegenerate, we simply integrate by parts in the direction $\partial_{\eta_1}$ to get
    $$
\left\|I_{k_1k_2k_3k}\right\|_{L^2_{\xi}} \lesssim \eps_1^4 2^{-k} \int_{t_1}^{t_2}\diff\tau \ 2^{-\frac{3m}{2}} +  2^{m\left(p_0-\frac43\right)} \lesssim \eps_1^4 2^{-k} .
$$
Since we only need to sum over a finite number of $k_1, k_2, k_3$, this bound is perfectly OK. The same arguments give the desired control on the $L^{\infty}_{\xi}$ piece. 
\end{itemize}
\par \noindent
\textbf{Subcase 2.13: $2^{k_1}, 2^{k_2}\in \left[2^{-1}, 2^5\right], 2^{k_{3}}\in\left(2^{\klo}, 2^{-1}\right)$ or $k_3=\klo$}
\par\noindent 
We must split $\eta_1, \eta_3$ into three regions as in subcase 2.11 and subcase 2.12.
\begin{itemize}
    \item If $\eta_1, \eta_2$ are in Region III, the proof reduces to subcase 2.6.
    \item If $\eta_1$ is in Region III and $\eta_2$ is in Region II, we get a copy of subcase 2.8.
    \item if $\eta_1$ is in Region III and $\eta_2$ is in Region I, we get an easier version of subcase 2.6.
    \item If $\eta_1, \eta_2$ are both in Region II, localize both about $\pm \sqrt{3}$ at scales $2^{\ell_1}, 2^{\ell_2}$, respectively. We then find by Taylor expansion that $\left|\partial_{\eta_1-\eta_3}\varphi\right| = \left|-\frac18 - 1 +\mathcal{O}\left(2^{2\ell_1}, 2^{2k_3}\right)\right|\gtrsim 1$. From here on out, the arguments move along the lines of subcases 2.5 and 2.7. I'll sketch the main ideas for the worst-case scenario $\ell_1=\ell_2=k_3\klo$ (so all three frequencies are degenerate, or at least very close to degenerate), as the other cases are similar and just involve keeping track of the sums over $\ell_1, \ell_2, k_3$ as we have dealt with many times before. 
    \begin{align*}
    \hspace{-1cm}
        \left\|\left\{\text{localization in} \ \eta_1, \eta_2\right\}I_{mk_1k_2\klo k}\right\|_{L^2_{\xi}}\lesssim \eps_1^4 2^{-k} 2^{mp_0}.
    \end{align*}
    No summation is necessary for the worst-case scenario, so we're all done this case. The same ideas allow us to bound $\left\|J_{mk_1k_2k_3k}\right\|_{L^{\infty}_{\xi}}$. 
    \item If $\eta_1$ is in Region II and $\eta_2$ is in Region I, the arguments of the previous item still work. 
    \item If $\eta_1, \eta_2$ are both in Region I, we use Taylor expansion to find $|\partial_{\eta_3}\varphi| \gtrsim 1$ since $|\eta_4|\sim 2^{k}\gg 2^{k_3}$. From here, the proof becomes an easier version of subcase 2.6. 
    \end{itemize}

\par \noindent
\textbf{Subcase 2.14: $2^{k_1}\in \left[2^{-1}, 2^5\right], 2^{k_2}, 2^{k_{3}}\in\left(2^{\klo}, 2^{-1}\right)$ or $k_2, k_3=\klo$}
\par\noindent 
As is unsurprising by now, we break $\eta_1$ space into three pieces with the notation established in subcase 2.11. 
\begin{itemize}
\item If $\eta_1$ is in Region III or Region I, we reduce to subcase 2.7.
\item If $\eta_1$ is in Region II, we can easily modify the arguments of subcase 2.13 used when we had three degenerate frequencies, since $|\partial_{\eta_1-\eta_3}\varphi|\gtrsim 1$ still holds. 
\end{itemize}

\par \noindent
\textbf{Subcase 2.15: $2^{k_{j}}\in \left(2^{\klo}, 2^{-1}\right)$ or $k_j =\klo \ \forall \ j$}
\par\noindent
Here, we find that $|\eta_4|\sim 2^{k}$. By permuting variables, we reduce to subcase 2.7. 

\subsubsection{Case 3: $2^{k}\in \left[2^{-1}, 2^{5}\right]$}
\noindent Note that, in this case, we need only prove that each norm piece is 
$\lesssim \eps_1^4 t^{p_0}$ since there are only finitely many $k$'s to sum over. Additionally, this case includes the anomalous space-time resonance $\left(\eta_0, \eta_0, \eta_0; \xi_0\right)$ with $\eta_0\approx 5.2, \xi_0\approx 14.2$. Finally, we can reduce this case down to six subcases using the following observations: 
\begin{itemize}
    \item first, we cannot have $2^{k}\gg 2^{k_1}$ if $k_1\geq 5$;
    \item second, a careful examination of case 2 shows that only subcases 2.3, 2.7, 2.8, 2.12, 2.13, and 2.15 strongly depend on the size of $|\xi|$, so we only have to worry about the analogues of these cases here. 
\end{itemize}

\par \noindent 
\textbf{Subcase 3.1: $2^{k_1}, 2^{k_2}, 2^{k_3}\in \left[2^5, 2^{\khi}\right)$, $2^{k_1}\approx 2^{k_2}\approx 2^{k_3} $}
\par\noindent
This can be handled using the arguments of subcase 2.3, with the only proviso being that the case $|\eta_4|\in \left(2^3, 2^{\khi}\right]$ must be slightly modified. The worst-case scenario here is $|\eta_j|\approx |\eta_4| \quad \forall  \ j=1,2,3. $
However, since $|\xi|$ is bounded (and far from zero here), we have $|\varphi| \gtrsim 1$, and reduce to an easier version of subcase 2.3. 

\par \noindent
\textbf{Subcase 3.2: $2^{k_1}\in \left[2^5, 2^{\khi}\right], \ 2^{k_2}, 2^{k_3}\in\left(2^{\klo}, 2^{-1}\right]$ or $k_2, k_3 = \klo$}
\par\noindent 
If $k_1\gg k$, we reduce to subcase 2.7. Then, it remains to obtain a useful bound when $k\approx k_1$. Following subcase 2.7, we still have $|\partial_{\eta_1-\eta_3}\varphi|\gtrsim 1$. However, here, there is no $2^{-k}$ to split and get extra decay. A bit more thought shows this is no problem: since we only care about $k_1\approx k$ and $k$ can only get so large, $k_{1}\lesssim 2^{5}$ and we therefore sum over only finitely many $k_1$'s. Consequently, in the worst-case $k_2=k_3=\klo$, integration by parts and the multilinear estimate give
\begin{align*}
   \sum_{k_1\approx k} \left\|I_{mk_1\klo \klo k}\right\|_{L^2_{\xi}}  \lesssim \eps_1^4 2^{mp_0},
\end{align*}
which is exactly the bound we want. The same ideas work for getting control on $\left\|J_{mk_1k_2k_3k}\right\|_{L^{\infty}_{\xi}}$. 

\par \noindent
\textbf{Subcase 3.3: $2^{k_1}\in\left[2^5, 2^{\khi}\right], 2^{k_2}\in \left[2^{-1}, 2^5\right], 2^{k_3}\in \left(2^{\klo}, 2^{-1}\right]$ or $k_3=\klo$}
\par\noindent 
This can be handled exactly like subcase 3.2 (itself basically a copy of subcase 2.7).

\par \noindent 
\textbf{Subcase 3.4: $2^{k_1}, 2^{k_3}, 2^{k_3}\in\left[2^{-1}, 2^5\right]$}
\par\noindent
We may encounter the anomalous space-time resonance in this range of frequencies, so we need to be a little bit careful. We split into Region I, II, III as before, but to properly accommodate all space-time resonances we must also treat $|\eta_j|\approx \eta_0 \approx 5.2$ specially. 

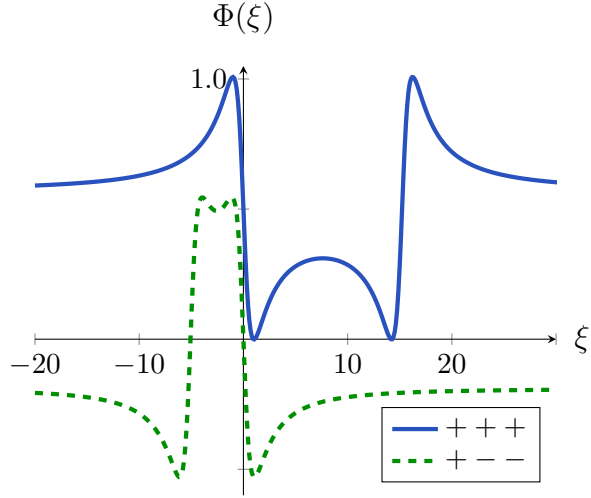
\begin{figure}
\centering
\begin{tikzpicture}[scale=1.]
\begin{axis}[xmin=-20, xmax=30, ymin=-0.6, ymax=1.05, samples=500, no markers,
      axis x line = middle,
      axis y line = middle,
    ylabel=\large{$\Phi(\xi)$},
    xlabel=\large{$\xi$},
    every axis x label/.style={
    at={(ticklabel* cs:1.05)},
},
every axis y label/.style={
    at={(ticklabel* cs:1.05)},
    anchor=south,
},
legend pos = south east,
yticklabels={,,$-0.5$,,$1.0$},
xticklabels={,$-20$,$-10$,0,$10$,$20$},
]
    \addplot[ceruleanblue, ultra thick, domain=-20:40]  {(x-3*5.0762)/(1+(x-3*5.0762)^2)  - x/(1+x^2) + (3)*(5.0762/(1+5.0762*5.0762))};
  \addlegendentry{$+++$}
   \addplot[islamicgreen, ultra thick, domain=-20:40, dashed]  {(x+5.0762)/(1+(x+5.0762)^2)  - x/(1+x^2) + (-1)*(5.0762/(1+5.0762*5.0762))};
  \addlegendentry{$+--$}
\end{axis}
\end{tikzpicture}
\caption{A sketch of $\Phi(\xi)$. The $+--$ curve has roots at $-\eta_0$ and $0$. CAUTION : While the thickness of the lines makes it look like $+++$ has only two second-order roots, zooming in shows that there are in fact four first-order roots at $\approx 0.9, 1.1, 14.1$, and $14.3$.} 
\label{fig:phi_region3}
\end{figure}

 



\begin{itemize}
\item First, suppose that for each $j=1,2,3$ we have $|\eta_j \pm \eta_0| \lesssim 1$. In this case, Taylor expansion shows that the phase is described by the function
$$
\Phi_{\sigma_1\sigma_2\sigma_3}(\xi) = -\omega(\xi) + \omega\left(\xi-\left(\sum_{j=1}^3 \sigma_{j}\right) \eta_0\right) + \left(\sum_{j=1}^3 \sigma_{j} \right)\omega(\eta_0). 
$$
See figure \ref{fig:phi_region3} for a plot in the only important cases $+++, +--$. Now, we can definitely hit the space-time resonances $\mp \left(\eta_0, \eta_0, \eta_0;\xi_0\right)$. We now bound the different $|\xi|$'s directly. 
\begin{itemize}
    \item If all the $\eta_j$'s are near $\eta_0$ and $|\xi-\xi_0|\lesssim 1$, we have already handled this case in our discussion of the anomalous resonances. The same goes for the situation where all $\eta_j$'s are near $-\eta_0$ and $|\xi+\xi_0|\lesssim 1$. 
    \item  If all the $\eta_j$'s are near $\pm\eta_0$ but $|\xi\pm\xi_0|\lesssim 1$, then figure \ref{fig:phi_region3} implies that $|\varphi|\gtrsim 1$. We then appeal to the arguments of subcase 2.3 to conclude. The same goes for $\eta_j$'s in the $+--$ case: figure \ref{fig:phi_region3} tells us that $\pm \xi_0$ are not time resonances here. 
    \item Suppose $|\xi|<2^{-1}$. In the $+++$ case, figure \ref{fig:phi_region3} implies that $|\varphi|\gtrsim 1,$ so again we can finish by a simple appeal to subcase 2.3. The $+--$ case is also not time-resonant because $|\xi|>0$. 
    \item If $\xi\approx 14.3, 0.9, 1.1$, or $-\eta_0\approx -5.2$, we can manually verify by Taylor expansion that $\left|\partial_{\eta_1}\varphi\right|\gtrsim 1$ (ie. these $\xi$ correspond to pure time resonances) and we reduce to an easier variant of subcase 2.2: since all $|\eta_j|$'s are near $\eta_0\neq 0, \sqrt{3}$,  only the $\eta_4$ term in our oscillatory integral can experience the worst-case decay. 
    \item In every other case, figure \ref{fig:phi_region3} tells us that we have $|\varphi|\gtrsim 1$ and we may again apply arguments from subcase 2.3. 
\end{itemize}
\item Next, suppose that 
$$
|\eta_1\pm \eta_0|\lesssim 1, \quad |\eta_2\pm \eta_0|\lesssim 1, \quad |\eta_3\pm r(\eta_0)|\lesssim 1,
$$
which means we're near the STRs $\pm \left(\eta_0, \eta_0, -r(\eta_0); \xi_0\right)$. This can be handled via the techniques of the previous item after making a linear change of variables. 
  \item Now, suppose that at least one $\eta_j$ is in Region III, but we are still in the complement of the two cases considered above. We split up $\eta_4$-space into several pieces using bump functions, as described in many of the subcases already covered. 
  \begin{itemize}
      \item If $|\eta_4|\approx |\eta_1|$ or $|\eta_4|\approx |r(\eta_1)|$, our earlier resonance computations from subsection \ref{ss:resonance_computation} show that $|\varphi|\gtrsim 1$ and we conclude by following subcase 2.3. 
      \item In every other case, we have $|\partial_{\eta_1}\varphi| \gtrsim 1$ and we reduce to subcase 2.2: even though some of our input frequencies may be degenerate, this doesn't cause problems since we sum over only a finite number of $k_{j}$'s. 
  \end{itemize}
  \item Now, suppose that all $\eta_j$ are in Region II. We split $\eta_4$-space again, but this time life is a bit easier:
  \begin{itemize}
      \item If $|\eta_4|\approx \sqrt{3}$, then we may be space resonant. We proceed by mimicking the arguments of subcase 2.12 (see especially figure \ref{fig:Omega}). In the $+++$ case, the leading order phase is $\gtrsim 1$ unconditionally, so we're fine. In the $-++$ case, the phase can only vanish if $\xi=0,\sqrt{3}$. Since we assume $2^{k}\in \left[2^{-1}, 2^{5}\right]$, $\xi=0$ is not allowed. Additionally, $\xi=\sqrt{3}$ and $|\eta_j|\approx \sqrt{3}$ in the $-++$ case implies $\eta_4\approx 0$ which is not allowed by our localization in $\eta_4$. Therefore, as long as $|\eta_4|\approx \sqrt{3}$, we have $|\varphi|\gtrsim 1$ and we reduce to subcase 2.3. 
     \item In every leftover situation, $|\partial_{\eta_1}\varphi|
    \gtrsim 1$ and we copy subcase 2.2 (again, here degenerate frequencies cause no substantial problems). 
  \end{itemize}
  \item If $\eta_1$ is in Region II and $\eta_3$ is in Region I, we mostly repeat the techniques of the above item, with the following change: if $|\eta_4|\approx \sqrt{3}$ and $\eta_3$ is far from Region II, we can't use figure \ref{fig:Omega} anymore. However, since $\eta_3$ is not near Region II, observe that $|\partial_{\eta_3}\varphi|\gtrsim 1$, so we reduce to subcase 2.2 again. 
  \item If all $\eta_j$ are in Region I, we mostly follow the same approach we used when all $\eta_j$ were in Region III and we were non-STR. However, the reflection of $\eta_1$ is not defined throughout the entirety of Region I, so we need to be a tiny bit more careful. First, notice that 
  $$
  |\eta_4| \leq |\xi| +\sum_{j=1}^{3}|\eta_j|\leq 2^{6}+ 3\left(\sqrt{3}-\delta\right) \lesssim 80. 
  $$
  Therefore, if $|\eta_1|<r\left(80\right)$, then we \emph{cannot} have $|r(\eta_4)| = |\eta_1|$. With this in mind, we can form a nice splitting of $\eta_1$-space: if $|\eta_1|\geq 0.9r(80)$, we localize exactly as we did when all $\eta_j$'s were in Region III, and if $|\eta_1|<0.9r(80)$ then we can do the same but ignore localization about $|\eta_4|\approx |r(\eta_1)|$ (since, by construction, this is impossible). This splitting of $\eta_1$-space can be accomplished by using two smooth bump functions supported in $\left[0.8*r(80), 80\right]$ and $\left[2^{-1}, 0.9*r(80)\right]$ and summing to $1$ everywhere on $\left[2^{-1}, 80\right]$. 
\end{itemize}

\par \noindent 
\textbf{Subcase 3.5: $2^{k_1}, 2^{k_3},\in\left[2^{-1}, 2^5\right], 2^{k_3}\in\left(2^{\klo}, 2^{-1}\right)$}
\par\noindent
This case proceeds almost identically to subcase 2.13, except for when $\eta_1, \eta_2$ are in Region I. First, if $k\gg k_1$ then $|\partial_{\eta_3}\varphi| \gtrsim 1$. Similarly, if $k_1\gg k_3$ then
$|\partial_{\eta_1-\eta_3}\varphi|\gtrsim 1$. This leaves the case $k\approx k_1\approx k_3\approx k_3$. Here, we have to split up into $|\eta_4|\approx 2^{k}$ and $|\eta_4|\ll 2^{k}$. In the latter case we simply use $|\partial_{\eta_1}\varphi|\gtrsim 1$, and in the former case we use our earlier resonance computations to find that $|\varphi| \gtrsim 1$, since $|\xi|$ is not small enough to give rise to a time resonance. 

\par \noindent
\textbf{Subcase 3.6: $2^{k_{j}}\in \left(2^{\klo}, 2^{-1}\right)$ or $k_j =\klo \ \forall \ j$}
\par\noindent
The case $k\approx k_1\approx k_2\approx k_3 > \klo$ is essentially contained in subcase 2.13. If $k\gg k_1$, then $|\eta_4|\sim 2^{k}$. By Taylor expansion, we discover $|\partial_{\eta_1}\varphi| \gtrsim 1$. From now on we only consider the worst-case scenario $k_1=k_2=k_3=\klo$. Integrating by parts and using the mutlilinear estimate gives
$$
\left\|I_{m\klo\klo\klo k}\right\|_{L^2{_\xi}} \lesssim \eps_1^4  2^{mp_0}
$$
as desired. Similar ideas likewise show that 
$$
\left\|J_{m\klo\klo\klo k}\right\|_{L^\infty{_\xi}}\lesssim \eps_1^4 2^{-am}
$$
for some small $a>0$. 
\subsubsection{Case 4: $2^{k} \in \left(2^{\klo}, 2^{-1}\right)$}
\noindent 
We again use the work already done in case 2 to make sure we only have to deal with a handful of subcases here. 

\par \noindent 
\textbf{Subcase 4.1: $2^{k_j}\in \left[2^5, 2^{\khi}\right) \quad \forall \ j, \ 2^{k_1}\approx 2^{k_2}\approx 2^{k_3}$ }
\par\noindent 
This can be handled similarly to subcase 2.3. The only part of the argument in subcase 2.3 that breaks down here is when $|\eta_4|\in \left(2^3, 2^{\khi}\right]$. Since $|\xi|\sim 2^{k}$ is now allowed to be very small, we may be close to space-time resonances of the form $\left(-\eta, \eta, \eta; 0\right)$ with $|\eta|\geq 2^4.$ These worst-case situations have already been handled in our discussion of the resonant lines, so we're all done this subcase. Note in particular that smallness of $|\xi|$ prohibits space resonances like $\left(\eta, \eta, \eta; 2\eta\right)$ with $|\eta|\geq 2^4$. 
\par \noindent
\textbf{Subcase 4.2: $2^{k_1}\in \left[2^5, 2^{\khi}\right], \ 2^{k_2}\in\left(2^{\klo}, 2^{\khi}\right] \ \text{or} \ k_2= \klo, \ 2^{k_3}\in\left(2^{\klo}, 2^{\khi}\right] \ \text{or} \ k_3= \klo$}
\par\noindent 
Here, the worst-case scenario is that we're close to the resonant curve discussed in subsection \ref{ss:res_curves}; this includes resonances like $\left(\eta, -\eta, r(\eta);0\right)$ or $\left(\eta, r\left(\eta\right), -r(\eta);0\right)$. These troublesome cases have already been handled in subsection \ref{ss:res_curves}. In the complementary cases, we can use $k_1\gg k$ to reduce to variants of the arguments in subcases 2.5, 2.7, and 2.8. 

\par \noindent 
\textbf{Subcase 4.3: $2^{k_1}, 2^{k_2}, 2^{k_3}\in\left[2^{-1}, 2^5\right]$}
\par\noindent
Here we are a $t$-dependent distance away from the STR $\left(-\sqrt{3}, \sqrt{3}, \sqrt{3};0\right)$ and must exercise a bit of caution. 
\begin{itemize}
    \item If $\eta_1$ is in Region III, we effectively reduce to one of the earlier subcases of case 4. 
    \item If all $\eta_j$ are in Region II, the graphical methods from subcase 2.12 (see figure \ref{fig:Omega}) tell us that, as long as we're not in the $-++$ situation, we have a lower bound of the form $|\varphi| \gtrsim 1$, so we're all done. This does not hold for the $-++$ situation, but we have already handled that scenario in subsection \ref{ss:type3}. 
\item If $\eta_1$ and at least one of $\eta_2, \eta_3$ are in Region I, then $k\ll k_1$ implies $|\partial_{\eta_1}\varphi| \gtrsim 1$ by Taylor expansion. 
\item If all $\eta_j$ are in Region I, we may have $|\eta_1|\approx |\xi|$ so the bound from the previous item may not work. If any of the $k_{j}\gg k$, however, then the bounds persist. Thus, we need only worry about $k_1\approx k_2\approx k_3\approx k$. This is covered by our analysis of the $(0,0,0;0)$ resonance in subsection \ref{ss:pure_zero}. 
\end{itemize}

\par \noindent 
\textbf{Subcase 4.4: $2^{k_1}, 2^{k_2}\in\left[2^{-1}, 2^5\right], \ 2^{k_3}\in \left(2^{\klo}, 2^{-1}\right)$ or $k_3=\klo$}
\par\noindent
This mostly follows from subcase 2.13, except for the situation where both $\eta_1$ and  $\eta_2$ are in Region I. Then, either $k_{1}\gg k_3$ (in which case we clearly have $|\partial_{\eta_1-\eta_3}\varphi|\gtrsim 1$), $k_1\gg k$ (in which case we use Taylor expansion to bound $|\partial_{\eta_1}\varphi|\gtrsim 1$), or $k_1\approx k_2\approx k_3\approx k$ and we're in a situation already discussed in subsection \ref{ss:pure_zero}. 

\par \noindent
\textbf{Subcase 4.5: $2^{k_{j}}\in \left(2^{\klo}, 2^{-1}\right)$ or $k_j =\klo \ \forall \ j$}
\par\noindent
This was handled in subsection \ref{ss:pure_zero}. 

\subsubsection{Case 5: $k = \klo$}
\noindent
This case can be handled by mostly copying the arguments of case 4, though these arguments occasionally need to be supplemented with appeals to subsections \ref{ss:pure_zero}, \ref{ss:type3}, \ref{ss:resonant_line}, and \ref{ss:res_curves}. 
\subsection{Part V: Proofs of Lemmas}
\label{ss:proofs_of_lemmas}
\subsubsection{Proof of Lemma \ref{lemma:time_derivative_bnd}}
\begin{figure}
\centering
\begin{tikzpicture}[scale=0.9]
\begin{axis}[xmin=5.05, xmax=5.22, ymin=-.014, ymax=.004, samples=200, no markers,
      axis x line = middle,
      axis y line = middle,
    xlabel=\Large{$\eta$},
    ylabel=\Large{$\Xi(\eta)$},
    every axis x label/.style={
    at={(ticklabel* cs:1.1)},
},
every axis y label/.style={
    at={(ticklabel* cs:1.05)},
    anchor=south,
},
yticklabels={,,},
xticklabels={,,},
scaled y ticks = false 
]
  \addplot[alizarin, ultra thick, domain=4.9:5.3]  {3*x/(1+x*x) - sqrt((x*x+3)/(x*x-1))/(1+(x*x+3)/(x*x-1)) - (3*x- sqrt((x*x+3)/(x*x-1)))/(1+( 3*x-sqrt((x*x+3)/(x*x-1)))*(3*x- sqrt((x*x+3)/(x*x-1))))};
    \addplot +[mark=none, denim, very thick, dashed, samples=2] coordinates {(5.19615242271
, -0.02) (5.19615242271
, 0.02)};
\draw[alizarin, fill] (5.0762,0) circle (4pt) node [above right] {$\eta=\eta_{0}$};
\draw[denim, fill] ((5.19615242271,0.) circle (4pt) node [above left] {$\eta=3\sqrt{3}$};
\end{axis}
\end{tikzpicture}
\caption{The red solid line denotes $\Xi(\eta)=3\omega(\eta) - \omega(r(\eta))-\omega(3\eta-r(\eta))$ near $\eta_{0}\approx 5.0762$. The blue dashed line emphasizes that this function does not vanish at $3\sqrt{3}$.} 
\label{fig:time_derivative_helper}
\end{figure}
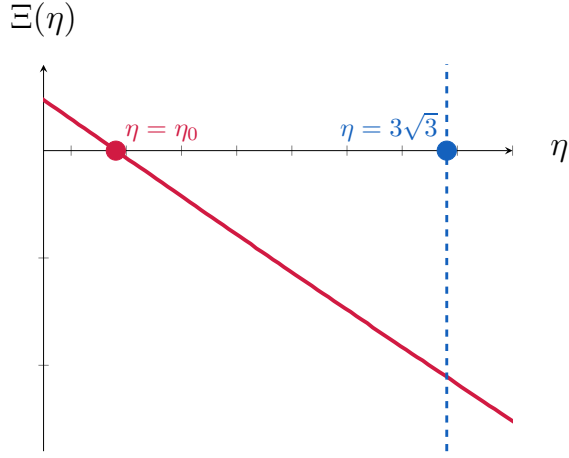
\par First, recall that
\begin{align*}
    \partial_{\tau}\left(\widehat{f}_{\lesssim2^{k_{*}}}\right)(\eta_{1}) \simeq \omega(\eta_{1}) \phi_{\lesssim 2^{k_{*}}}(\eta_1-\eta_0) e^{i\tau\omega(\eta_1)}\int \diff \sigma_1 \ \diff\sigma_2 \ \diff\sigma_3 \ \widehat{u}(\sigma_1) \ \widehat{u}(\sigma_2) \ \widehat{u}(\sigma_3) \ \widehat{u}\left(\eta_1-\sum_{j=1}^{3}\sigma_{j}\right). 
\end{align*}
Let us denote $\sigma_{4}\doteq \eta_1-\sum_{j=1}^{3}\sigma_{j}$. Then, owing to the bump function outside the integral in the above expression, the only sub-region of our integration domain that actually contributes to $ \partial_{\tau}\widehat{f}_{\lesssim2^{k_{*}}}$ is $\left|\sum_{j=1}^{4}\sigma_{j} - \eta_0\right| \lesssim 2^{k_{*}}$. We then claim that, if $(\sigma_1,...,\sigma_4)$ lives in the above sub-region, then at least one $\sigma_{j}$ is not equal to $0$ or $\pm\sqrt{3}$. This is a simple proof by contradiction. Assume we're given $\sigma_{j}$'s in this subregion that are all $0$ or $\pm\sqrt{3}$, then we must have $\sum_{j=1}^{4}\sigma_{4} = \pm K\sqrt{3}$ for $K=0,1,2,3,4$. 
Since $\eta_{0}\approx 5.1$, the only one of these options that is anywhere near possible is $3\sqrt{3}$. Using figure \ref{fig:time_derivative_helper}, however, we find that $|\eta_0-3\sqrt{3}|>0$. 
\par So, at least one $\sigma_{j}$ must be localized away from the degenerate stationary points. Without loss of generality, suppose it is $\sigma_{1}$. Then, for an appropriately (and time-independently) scaled compactly supported triple-bump function $\Phi(\sigma_1)$ localized near both $0$ and $\pm\sqrt{3}$, define $\widehat{v}(\sigma_1) = \left(1-\Phi(\sigma_1)\right) \widehat{u}\left(\sigma_1\right)$. Na\"{i}vely, since the Fourier support of $v$ includes no degenerate stationary points, we expect that $v$ has a leading-order decay ray of $\tau^{-\frac12}$. However, owing to the asymptotic degeneracy of the BBM dispersion relation, things are a little more complicated. For this reason, we first prove the lemma by assuming that $\widehat{v}(\sigma_1)$ has compact support. Consequently, $\left\|v\right\|_{L^{\infty}_{x}}\lesssim \eps_1 \tau^{-\frac12}$ for now. 
\par In the worst case where $\sigma_2, \sigma_3, \sigma_4$ are still near $0, \pm\sqrt{3}$ (which, by the above discussion, is allowed in spite of $\eta_1\approx \eta_0)$, a na\"{i}ve estimate of $\partial_{\tau}\widehat{f}_{\lesssim 2^{k_*}}$ using untangling would then give a decay rate of $\tau^{-\frac16}$, which is weaker than \eqref{eqn:time_derivative_bnd}. To fix this problem, we use a different approach. Suppose that for $j=2,3,4$ the distance between $\sigma_{j}$ and the nearest degenerate frequency is $\leq 2^{\ell}\ll 1$: without loss of generality, assume they are all near $\sqrt{3}$. Consequently, in this worst-case scenario, $\sigma_{1} = \eta_0-3\sqrt{3} + \mathcal{O}\left(2^{\max\left\{\ell,k_{*}\right\}}\right)$, so $\sigma_{1}$ is an absolute distance away from $0, \pm\sqrt{3}$. This yields $\left|\partial_{\sigma_{1}}\varphi\right|\gtrsim 1$, which in turn implies
\begin{align*}
    \left|\frac{1}{\partial_{\sigma_1}\varphi}\right|\lesssim 1,  \quad  \left|\partial_{\sigma_1}\left(\frac{1}{\partial_{\sigma_1}\varphi}\right)\right|\lesssim 1. 
\end{align*}
Using the above bounds and a bit of computation, we find that for $j=1,2$ the symbols
\begin{align*}
M_{j}&=\left[\partial_{\sigma_1}^{j-1}\left(\frac{1}{\partial_{\sigma_1}\varphi}\right)\right] \phi_{\lesssim 2^{k_{*}}}(\eta_1-\eta_0) \ \phi_{\lesssim 2^{\max\left\{\ell,k_{*}\right\}}}\left(\sigma_1-\left(\eta_0-3\sqrt{3}\right) \right)\ \phi_{\lesssim 2^\ell}\left(\sigma_2 - \sqrt{3}\right) \ \phi_{\lesssim 2^\ell}\left(\sigma_3 - \sqrt{3}\right)
\end{align*}
obey the conditions of lemma \ref{lemma:multiplier_cheat_code}. We may then apply the $L^{\infty}_{\xi}$ multilinear estimate from proposition \ref{prop:multilinear_estimates_Linfty} to discover 
\begin{align*}
    \left\|\left(\partial_{\tau}\widehat{f}_{\lesssim 2^{k_{*}}}\right)_{\text{worst}}\right\|_{L^{\infty}_{\eta_1}}\lesssim \eps_1^4\tau^{-1}\left(\tau^{-\frac56}+ \tau^{p_0-\frac23}+\tau^{p_0-\frac56}\right).
\end{align*}
Since $p_0<\frac16$, the terms in parentheses above are $\leq 1$ for all $\tau\geq 1$, so \eqref{eqn:time_derivative_bnd} holds for the worst-case piece of $\partial_{\tau}\widehat{f}_{\lesssim 2^{k_{*}}}$. 
\par In the remaining portion of frequency space, another $\sigma_{j}$ is far from $0$ and $\pm \sqrt{3}$. Without loss of generality, suppose this additional non-degenerate frequency is $\sigma_2$. Call the inverse Fourier transforms of the $\sigma_1$, $\sigma_{2}$ terms (including spectral localization) by $v$ and $w$ respectively. Then, we untangle the convolution and use the ODE for $\widehat{f}_{\lesssim 2^{k_*}}$ to discover 
\begin{align*}
    \left\|\left(\partial_{\tau} \hat{f}_{\lesssim_{2^{k_{*}}}}\right)_{\text{non-worst}}\right\|_{L^{\infty}_{\xi}} &\lesssim \left\|\left(vwu^2\right)^{\wedge}\right\|_{L^{\infty}_{\xi}}
    \lesssim \eps_1^2 \tau^{-1} \left\|u\right\|_{L^2_{x}}^2. 
\end{align*}
Using $H^{1}_{x}$-norm conservation \eqref{eqn:energy_conservation}, the above tells us that 
\eqref{eqn:time_derivative_bnd} also holds for the ``non-worst'' piece, so we're all done when $\widehat{v}$ has compact support. 
\par If $\widehat{v}$ does not have compact support, our integration region may contain points like 
$$
\left(\sigma_1, \sigma_2, \sigma_3, \sigma_4\right) = \left(\eta_0+M, -M, 0, 0\right), \quad M\gg 1,
$$
where the decay of the $v$ term slows down slightly, and we have a bit more work to do. So, now suppose $|\sigma_1|, |\sigma_2|\gg 1$ (absolutely). If just one of $\sigma_3, \sigma_4$ are $\approx 0, \pm \sqrt{3}$, then we can integrate by parts and simply follow the procedure applied in the worst case, but this time we place the $\sigma_1, \sigma_2$ terms in $L^2$. For example, if $|\eta_1| = M \gg 1$ and $\eta_3\approx 0$, we would integrate by parts in the direction $\partial_{\eta_1-\eta_3}$ using
$$
|\partial_{\eta_1-\eta_3}\varphi \left(\eta_0+M, -M, 0, 0\right) |\approx  |\omega'(M)-\omega'(0)| \approx 1. 
$$
In the remaining situation where both $\sigma_3$ and $\sigma_4$ are nondegenerate we simply untangle as before, but leave the $\sigma_1, \sigma_2$ terms inside $L^2$. 
\subsubsection{Proof of Lemma \ref{lemma:time_derivative_bnd_eta4}}
    This follows the exact same method used to prove the previous lemma: the key step is recognizing that, under the prescribed localization, $\eta_4 \approx -r(\eta_0) \neq 0, \pm \sqrt{3}$.

\section{Proof of Proposition \ref{prop:scattering_in_L^2}}
 \label{s:scattering_proof}
From theorem \ref{thm:big_result}, we have a global-in-time solution to \eqref{eqn:3bbm_cauchy_problem_asympt}, $u(t,x)$. Denote its profile by $f(t,x)$. Pick any $t_1, t_2$ such that $t_2>t_1>1$. Let $m_1$ be the largest integer such that $2^{m_1}\leq t_1$ and let $m_2$ be the smallest integer such that $2^{m_2}\geq t_2$. Then, we can write
$$
[t_1, t_2] = \bigcup_{m=m_1}^{m_2}\left\{\tau\sim 2^{m}\right\}\cap [t_1, t_2]. 
$$
Using \eqref{eqn:ODE_for_fhat_framework}, we have 
\begin{align*}
    \widehat{f}\left(t_2\right) -    \widehat{f}\left(t_1\right) &= \omega(\xi)\int_{t_1}^{t_2}\diff \tau \int \diff\eta_1\diff\eta_2\diff\eta_3 \ e^{-i\tau\varphi} \  \widehat{f}\left(\eta_1\right) \ \widehat{f}\left(\eta_2\right) \ \widehat{f}\left(\eta_3\right) \ \widehat{f}\left(\eta_{4}\right)
    \\
    &= \sum_{m=m_1}^{m_2}\int_{\left\{\tau\sim 2^{m}\right\}\cap [t_1, t_2]}\diff \tau \int \diff\eta_1\diff\eta_2\diff\eta_3 \ e^{-i\tau\varphi} \  \widehat{f}\left(\eta_1\right) \ \widehat{f}\left(\eta_2\right) \ \widehat{f}\left(\eta_3\right) \ \widehat{f}\left(\eta_{4}\right)= \sum_{m=m_1}^{m_2} J_{m},
\end{align*}
where the oscillatory integral $J_{m}$ is defined in \eqref{eqn:Jdefn}, up to massaging the endpoints of the temporal integration. We already showed how to bound each $\left\|J_m\right\|_{L^{\infty}_{\xi}}$ during the proof of proposition \ref{prop:bootstrap_close}: in particular, we know that there exists a small $a>0$ such that $\left\|J_m\right\|_{L^{\infty}_{\xi}} \lesssim \eps_0^4 2^{-am}$. Accordingly, 
\begin{align*}
    \left\|  \widehat{f}\left(t_2\right) -    \widehat{f}\left(t_1\right)\right\|_{L^{\infty}_{\xi}} = \left\|\sum_{m=m_1}^{m_2} J_{m}\right\|_{L^{\infty}_{\xi}} \lesssim \eps_0^4 \sum_{m=m_1}^{m_2}  2^{-am} \lesssim \eps_0^4 2^{-am_1} \simeq \eps_0^4 t_1^{-a}. 
\end{align*}
We then find that $\widehat{f}(t,\xi)$ is $L^{\infty}_{\xi}$-Cauchy in time. By completeness of $L^{\infty}_{\xi}$, there exists a unique asymptotic profile $\widehat{F}\in L^{\infty}_{\xi}$ such that $\widehat{f}(t,\xi) \rightarrow \widehat{F}\left(\xi\right)$ in $L^{\infty}_{\xi}$. In fact, we can also say that $\widehat{f}(t,\xi)\rightarrow \widehat{F}(\xi)$ in $L^{2}_{\xi}$: since the $L^{2}_{\xi}$ multilinear estimate from proposition \ref{prop:multilinear_estimates} allows us to put three functions in $L^{\infty}_{x}$ rather than just two (as in the $L^{\infty}_{\xi}$ multilinear estimate from proposition \ref{prop:multilinear_estimates_Linfty}), computations similar to those performed in the proof of proposition \ref{prop:bootstrap_close} would give $\left\|J_m\right\|_{L^{2}_{\xi}} \lesssim \eps_0^4 2^{-\left(a+\frac13\right)m}$, and we argue as before. 
\begin{remark}
    One can use the above arguments and Gagliardo-Nirenberg to recover scattering in $H^{r}_{x}$ for $r<100$.
\end{remark}
\section{Acknowledgements}
The material presented here formed the core of my PhD thesis. I would like to thank my advisor Fabio Pusateri for his patience and invaluable support. Additionally, I am grateful to the readers of my thesis (Adrian Nachman, Dmitry Pelinovsky, Adam Stinchcombe, and Catherine Sulem) for their helpful comments on this material. This work was partially supported by an NSERC CGS-D award. 

\appendix
\begin{appendices}
\section{Computation of Space-Time Resonances}
\label{section:res_comp}
\noindent In this section, we compute all space-time resonances associated to the oscillatory integral in \eqref{eqn:fhat_integrated_framework}. 
\subsection{What modes have the same group velocity?}
\par First, we discuss solutions to the equation 
\begin{equation}
\label{eqn:modes_with_groupvel_c}
\omega'(\xi) = c
\end{equation}
for some fixed $c\in \mathbb{R}$, where $\omega(\xi)=\xi\langle\xi\rangle^{-2}$ is the linBBM dispersion relation. 
\begin{itemize} 
\item Using the global maximum and minimum of $\omega'$, we know \eqref{eqn:modes_with_groupvel_c} has no solutions for $c<-1/8$ or $c>1$. 
\item If $c=1$, the only solution is $\xi=0$. 
\item If $c=0$, then $\xi=\pm1$ are the only solutions.
\item If $c=-1/8$, then $\xi=\pm\sqrt{3}$ are the only solutions. 
\item  If $c\in (0,1)$, then there are two solutions, given explicitly by 
\begin{equation}
 \pm\xi_{*} = \pm\sqrt{\frac{-(2c+1)+\sqrt{8c+1}}{2c}} \label{eqn:invert_from_c}.
\end{equation}
\item For $|\xi|>1$ and $c\in (-1/8, 0)$, \eqref{eqn:modes_with_groupvel_c} has four solutions.  One can find two solutions $\pm\xi_{*}$ using \eqref{eqn:invert_from_c}.  Further, a direct calculation shows that, if $\xi$ satisfies $\omega'(\xi)=c$ then so does $r\left(\xi\right)$ defined by \eqref{eqn:refl_denf}. Thus the four solutions to $\omega'(\xi) = c$ are $\pm\xi_{*}, \ \pm r\left(\xi_{*}\right)$: if we know one solution, we know all four. 
\end{itemize} 
Note that we can interpret $\xi\mapsto r\left(\xi\right)$ for $\xi >1$ as reflection about $\xi=\sqrt{3}$ on the graph of $\omega'(\xi)$. For $\xi<-1$, we can likewise interpret $\xi\mapsto r\left(\xi\right)$ as reflection about $\xi=-\sqrt{3}$ on the graph. In particular, $r\left(r\left(\xi\right)\right) = \xi$. For a plot of $r\left(\xi\right)$, see Figure \ref{fig:tilde}. 

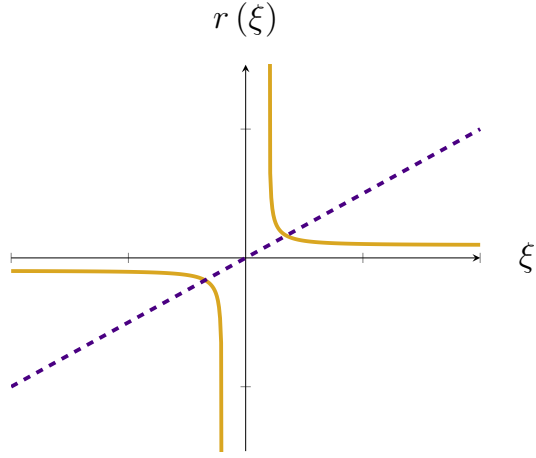
\begin{figure}
\centering
\begin{tikzpicture}[scale=0.9]
\begin{axis}[xmin=-10, xmax=10, ymin=-15., ymax=15., samples=200, no markers,
      axis x line = middle,
      axis y line = middle,
    ylabel=\Large{$r\left(\xi\right)$},
    xlabel=\Large{$\xi$},
    every axis x label/.style={
    at={(ticklabel* cs:1.1)},
},
every axis y label/.style={
    at={(ticklabel* cs:1.05)},
    anchor=south,
},
yticklabels={,,},
xticklabels={,,},
]
  \addplot[goldenrod, ultra thick, domain=1.001:10]  {sqrt((x*x + 3)/(x*x - 1))};
    \addplot[goldenrod, ultra thick, domain=-10:-1.001]  {-sqrt((x*x + 3)/(x*x - 1))};
      \addplot[indigo(web), dashed, ultra thick, domain=-10:10]  {x};

\end{axis}
\end{tikzpicture}
\caption{Plot of $r\left(\xi\right)$. There are vertical asymptotes at $|\xi|=1$. The purple dashed line intersects the graph twice, at $\pm\sqrt{3}$.} 
\label{fig:tilde}
\end{figure}

\subsection{Determining Resonances}
\label{ss:resonance_computation}
We begin by finding the critical points of the nonlinear phase function $\varphi(\eta_{1},\eta_{2},\eta_{3};\xi) .$ A critical point $(\eta_{1}, \eta_{2},\eta_{3})$ must satisfy
\begin{align*}
\partial_{\eta_{1}}\varphi &= \omega'(\eta_{1})-\omega'(\xi-\eta_{1}-\eta_{2}-\eta_{3}) = 0,
\\
\partial_{\eta_{2}}\varphi &= \omega'(\eta_{2})-\omega'(\xi-\eta_{1}-\eta_{2}-\eta_{3}) = 0,
\\
\partial_{\eta_{3}}\varphi &= \omega'(\eta_{3})-\omega'(\xi-\eta_{1}-\eta_{2}-\eta_{3}) = 0.
\end{align*}
In particular,
$$
\omega'(\eta_{1}) = \omega'(\eta_{2}) = \omega'(\eta_{3}) = \omega'(\xi-\eta_{1}-\eta_{2}-\eta_{3}) . 
$$
Using our earlier discussion on solutions of \eqref{eqn:modes_with_groupvel_c}, we know that there are four families of solutions to the above system:
\begin{itemize}
\item Family 1: $|\eta_{1}|=|\eta_{2}|=|\eta_{3}|$,
\item Family 2: $|\eta_{1}|=|\eta_{2}|=|r\left(\eta_{3}\right)|$,
\item Family 3: $|\eta_{1}|=|r\left(\eta_{2}\right)|=|\eta_{3}|$,
\item Family 4: $|r\left(\eta_{1}\right)|=|\eta_{2}|=|\eta_{3}|$.
\end{itemize}
Note that we can only have a solution in families 2, 3, or 4 when $|\eta_{1}|, |\eta_{2}|, |\eta_{3}| >1$. Additionally, in the special case $\xi=0$, we can pick out two breeds of resonant manifolds. 
\begin{enumerate}
    \item First, the line $L$ from \eqref{eqn:resonant_line_discussion} consists entirely of space-time resonances. We can get other resonant lines by changing the place of the single minus sign (for instance, we can have it fall on $\eta_2$ instead of $\eta_1$). 
    \item Additionally, the curve $\Gamma$ from \eqref{eqn:resonant_curve_discussion} is space-time resonant too. Note that $\eta_4$ in this case is $-\eta_1$. We can generate related space-time resonant curves by permuting the $\eta_j$'s for $j=1,...,4$: for instance, we can switch $\eta_2$ and $\eta_4$ to get the new resonant curve
    $$
    \Gamma' =  \left\{\left(\eta_1,\eta_2,\eta_3;\xi\right) = \left(\eta,-\eta,-r(\eta); 0\right) \ | \ \left|\eta\right|>1 \right\}\subseteq \mathbb{R}^4.
    $$
\end{enumerate}
 \ \par \noindent \textbf{Family 1}
 \par \noindent
 In this case we must have 
 \begin{align}
 |\eta_{1}|=|\eta_{2}|=|\eta_{3}| =\begin{cases} 
  |\xi-\eta_{1}-\eta_{2}-\eta_{3}| 
  \\
   \quad\quad\quad \text{or}
 \\
\left|r\left(\xi-\eta_{1}-\eta_{2}-\eta_{3}\right)\right|
 \end{cases}
 \label{eqn:fam_1_main_eqn}
 \end{align}
 There are several different subfamilies,
 \begin{subequations}
 \label{eqn:fam_1_subfam}
 \begin{align}
 \eta_{1}&=\phantom{-}\eta_{2}=\phantom{-}\eta_{3},\label{eqn:fam_1_a}
 \\
 -\eta_{1} &=\phantom{-} \eta_{2} = \phantom{-}\eta_{3},\label{eqn:fam_1_b}
 \\
 \eta_{1} &= -\eta_{2} = \phantom{-}\eta_{3},\label{eqn:fam_1_c}
 \\
 \eta_{1} &= \phantom{-}\eta_{2} =- \eta_{3},\label{eqn:fam_1_d}
 \end{align}
 \end{subequations}
each corresponding to four equations for a critical point by \eqref{eqn:fam_1_main_eqn}. 
\begin{itemize}
\item For \eqref{eqn:fam_1_a}, we have $\xi-\eta_{1}-\eta_{2}-\eta_{3}=\xi-3\eta_{1}$. Then, we must solve 
\begin{align*}
\eta_{1}  = \begin{cases} 
  \xi-3\eta_{1}
  \\
 3\eta_{1} - \xi
 \\
 r\left(\xi-3\eta_{1}\right)
 \\
-r\left(\xi-3\eta_{1}\right). 
 \end{cases}
\end{align*}
The first two cases yield the critical points $(\eta_{1},\eta_{2},\eta_{3}) = (\xi/4,\xi/4,\xi/4), \  (\xi/2,\xi/2,\xi/2)$. Now, the phase only vanishes at these space resonances when $\xi=0$. We have therefore found that $\left(\eta_1,\eta_2, \eta_3, ;\xi\right) = (0,0,0;0)$ is a space-time resonance. This lies on the resonant line $L$ defined in \eqref{eqn:resonant_line_discussion}. 
\par For the third case, we need to solve the nonlinear equation 
$$
3\eta_{1} + r\left(\eta_{1}\right) = \xi,
$$
since $r\left(r\left(\eta_{1}\right)\right)=\eta_{1}$.  From figure \ref{fig:tilde_system_1}, this equation has between $0$ and $2$ solutions for any given $\xi$. To see if these solutions give rise to a time resonance we need to compute
\begin{align*}
\varphi(\eta_{1},\eta_{1},\eta_{1}) &= -\omega(\xi)+3\omega(\eta_{1}) + \omega(\xi-3\eta_{1})
\\
&= -\omega(3\eta_{1}+r\left(\eta_{1}\right))+3\omega(\eta_{1}) + \omega(r\left(\eta_{1}\right)).
\end{align*}
The function in the last equality is plotted in figure \ref{fig:vanishing_phase_1}. We see that there are no roots, and that the function asymptotes to a nonzero number in either spatial direction. Therefore, the phase never vanishes at this critical point, so we're dealing with a pure space resonance. 

\begin{figure}
\centering
\begin{tikzpicture}[scale=0.9]
\begin{axis}[xmin=-10, xmax=10, ymin=-40., ymax=40., samples=200, no markers,
      axis x line = middle,
      axis y line = middle,
    xlabel=\Large{$\eta_{1}$},
    every axis x label/.style={
    at={(ticklabel* cs:1.1)},
},
every axis y label/.style={
    at={(ticklabel* cs:1.05)},
    anchor=south,
},
yticklabels={,,},
xticklabels={,,},
]
  \addplot[ballblue, ultra thick, domain=1.001:10]  {3*x+sqrt((x*x + 3)/(x*x - 1))};
    \addplot[ballblue, ultra thick, domain=-10:-1.001]  {3*x-sqrt((x*x + 3)/(x*x - 1))};
  
\end{axis}
\end{tikzpicture}
\caption{Plot of $3\eta_{1} + r\left(\eta_{1}\right)$. There are vertical asymptotes at $|\eta_{1}|=1$.} 
\label{fig:tilde_system_1}
\end{figure}

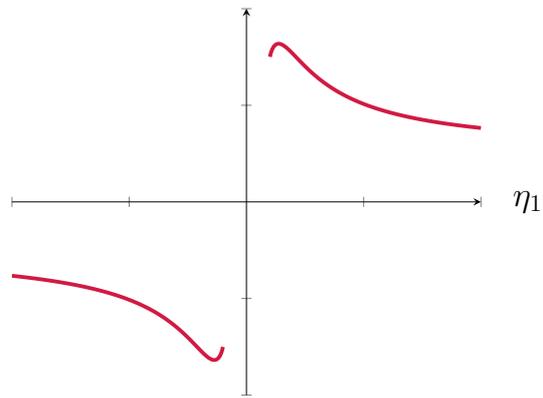
\begin{figure}
\centering
\begin{tikzpicture}[scale=0.9]
\begin{axis}[xmin=-10, xmax=10, ymin=-2., ymax=2., samples=200, no markers,
      axis x line = middle,
      axis y line = middle,
    xlabel=\Large{$\eta_{1}$},
    every axis x label/.style={
    at={(ticklabel* cs:1.1)},
},
every axis y label/.style={
    at={(ticklabel* cs:1.05)},
    anchor=south,
},
yticklabels={,,},
xticklabels={,,},
]
  \addplot[alizarin, ultra thick, domain=1.001:10]  {3*x/(1+x*x) + sqrt((x*x+3)/(x*x-1))/(1+(x*x+3)/(x*x-1)) - (3*x+ sqrt((x*x+3)/(x*x-1)))/(1+( 3*x+sqrt((x*x+3)/(x*x-1)))*(3*x+ sqrt((x*x+3)/(x*x-1))))};
    \addplot[alizarin, ultra thick, domain=-10:-1.001]  {3*x/(1+x*x) - sqrt((x*x+3)/(x*x-1))/(1+(x*x+3)/(x*x-1)) - (3*x- sqrt((x*x+3)/(x*x-1)))/(1+( 3*x-sqrt((x*x+3)/(x*x-1)))*(3*x- sqrt((x*x+3)/(x*x-1))))};
  
\end{axis}
\end{tikzpicture}
\caption{Plot of $3\omega(\eta_{1}) + \omega(r\left(\eta_{1}\right))-\omega(3\eta_{1}+r\left(\eta_{1}\right))$ on $\left\{|\eta_{1}|>1\right\}.$} 
\label{fig:vanishing_phase_1}
\end{figure}

For the fourth case, we need to solve the nonlinear equation 
$$
3\eta_{1} - r\left(\eta_{1}\right) = \xi. 
$$
From figure \ref{fig:tilde_system_2}, this equation always has $2$ solutions for any given $\xi$. To check if the phase vanishes here, we must find if the equation 
$$
3\omega(\eta_{1}) - \omega(r\left(\eta_{1}\right))-\omega(3\eta_{1}-r\left(\eta_{1}\right)) =0.
$$
has any solutions. From figure \ref{fig:vanishing_phase_2} we see there are $2$ solutions. A root-finding program may be used to verify that the roots are approximately $|\eta_{1}|\approx 5.1$. From figure \ref{fig:tilde_system_2}, there is indeed a value of $\xi$ that will yield such $\eta_{1}$ as stationary points, and we conclude that there are two special $\xi$ for which the phase vanishes at this critical point. Since $\xi=3\eta_{1}-r\left(\eta_{1}\right)$, we know the $\xi$-values giving rise to these roots are $\xi\approx \pm14.2$, which are quite far from the degenerate frequencies $|\xi|=0,\sqrt{3}$. 

\begin{figure}
\centering
\begin{tikzpicture}[scale=0.9]
\begin{axis}[xmin=-10, xmax=10, ymin=-40., ymax=40., samples=200, no markers,
      axis x line = middle,
      axis y line = middle,
    xlabel=\Large{$\eta_{1}$},
    every axis x label/.style={
    at={(ticklabel* cs:1.1)},
},
every axis y label/.style={
    at={(ticklabel* cs:1.05)},
    anchor=south,
},
yticklabels={,,},
xticklabels={,,},
]
  \addplot[ballblue, ultra thick, domain=1.001:10]  {3*x-sqrt((x*x + 3)/(x*x - 1))};
    \addplot[ballblue, ultra thick, domain=-10:-1.001]  {3*x+sqrt((x*x + 3)/(x*x - 1))};
  
\end{axis}
\end{tikzpicture}
\caption{Plot of $3\eta_{1} - r\left(\eta_{1}\right)$. There are vertical asymptotes at $|\eta_{1}|=1$.} 
\label{fig:tilde_system_2}
\end{figure}

\begin{figure}
\centering
\begin{tikzpicture}[scale=0.9]
\begin{axis}[xmin=-10, xmax=10, ymin=-1., ymax=1., samples=200, no markers,
      axis x line = middle,
      axis y line = middle,
    xlabel=\Large{$\eta_{1}$},
    every axis x label/.style={
    at={(ticklabel* cs:1.1)},
},
every axis y label/.style={
    at={(ticklabel* cs:1.05)},
    anchor=south,
},
yticklabels={,,},
xticklabels={,,},
]
  \addplot[alizarin, ultra thick, domain=1.25:10]  {3*x/(1+x*x) - sqrt((x*x+3)/(x*x-1))/(1+(x*x+3)/(x*x-1)) - (3*x- sqrt((x*x+3)/(x*x-1)))/(1+( 3*x-sqrt((x*x+3)/(x*x-1)))*(3*x- sqrt((x*x+3)/(x*x-1))))};
    \addplot[alizarin, ultra thick, domain=-10:-1.25]  {3*x/(1+x*x) + sqrt((x*x+3)/(x*x-1))/(1+(x*x+3)/(x*x-1)) - (3*x+ sqrt((x*x+3)/(x*x-1)))/(1+( 3*x+sqrt((x*x+3)/(x*x-1)))*(3*x+sqrt((x*x+3)/(x*x-1))))};
  
\end{axis}
\end{tikzpicture}
\caption{Plot of $3\omega(\eta_{1}) - \omega(r\left(\eta_{1}\right))-\omega(3\eta_{1}-r\left(\eta_{1}\right))$ with values from $\left\{|\eta_{1}|>1.2\right\}$. There are two roots at $|\eta_{1}|\approx 5.1$.} 
\label{fig:vanishing_phase_2}
\end{figure}
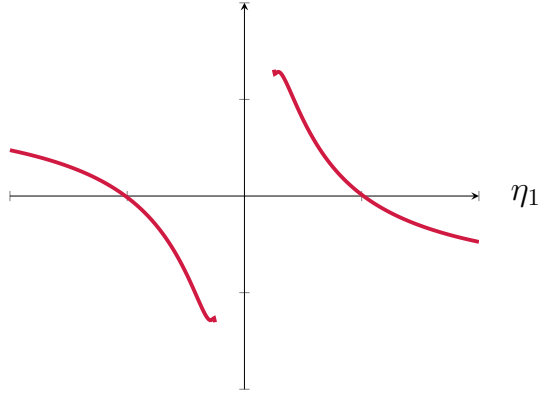

\item For \eqref{eqn:fam_1_b} we have $\xi-\eta_{1}-\eta_{2}-\eta_{3} = \xi+\eta_{1}$. Then, we must solve 

\begin{align*}
\eta_{1}  = \begin{cases} 
  \xi+\eta_{1}
  \\
 -\eta_{1} - \xi
 \\
 r\left(\xi+\eta_{1}\right)
 \\
- r\left(\xi+\eta_{1}\right). 
 \end{cases}
\end{align*}
The first equation can only be satisfied if $\xi=0$, giving rise to a resonance on a transformed version of the line \eqref{eqn:resonant_line_discussion} (the transformation just involves juggling negative signs). The second equation yields $\eta_{1}=-\xi/2$. Thus the second equation gives rise to the stationary point $(\eta_{1},\eta_{2},\eta_{3}) = (-\xi/2, \xi/2, \xi/2)$, and it's easy to see that the phase vanishes at this space resonance if any only if $\xi=0$, reproducing a case already encountered. 
\par The third equation requires us to solve 
$$
-\eta_{1} + r\left(\eta_{1}\right) = \xi.
$$
From figure \ref{fig:tilde_system_3}, this equation always has two solutions. To see if these solutions give rise to time resonances as well, we need to determine if
$$
-\omega(\eta_{1}) + \omega(r\left(\eta_{1}\right))-\omega(-\eta_{1}+r\left(\eta_{1}\right))=0
$$
has any solutions. Using figure \ref{fig:vanishing_phase_3}, we find there are two roots. Since $r\left(\sqrt{3}\right) = \sqrt{3}$, it is trivial to check that these roots occur at $|\eta_{1}|=\sqrt{3}$. Now, $\xi=-\eta_{1}+r\left(\eta_{1}\right)=0$, so we have a space-time resonance at $\left(\eta_1, \eta_2, \eta_3;\xi\right) = \left(-\sqrt{3}, \sqrt{3}, \sqrt{3}; 0\right)$. This lies on the resonant line given in \eqref{eqn:resonant_line_discussion}, and on a transformed version of the resonant curve in \eqref{eqn:resonant_curve_discussion}. 
\begin{figure}
\centering
\begin{tikzpicture}[scale=0.9]
\begin{axis}[xmin=-10, xmax=10, ymin=-40., ymax=40., samples=200, no markers,
      axis x line = middle,
      axis y line = middle,
    xlabel=\Large{$\eta_{1}$},
    every axis x label/.style={
    at={(ticklabel* cs:1.1)},
},
every axis y label/.style={
    at={(ticklabel* cs:1.05)},
    anchor=south,
},
yticklabels={,,},
xticklabels={,,},
]
  \addplot[ballblue, ultra thick, domain=1.001:10]  {-x+sqrt((x*x + 3)/(x*x - 1))};
    \addplot[ballblue, ultra thick, domain=-10:-1.001]  {-x-sqrt((x*x + 3)/(x*x - 1))};
  
\end{axis}
\end{tikzpicture}
\caption{Plot of $-\eta_{1} + r\left(\eta_{1}\right)$. There are vertical asymptotes at $|\eta_{1}|=1$.} 
\label{fig:tilde_system_3}
\end{figure}

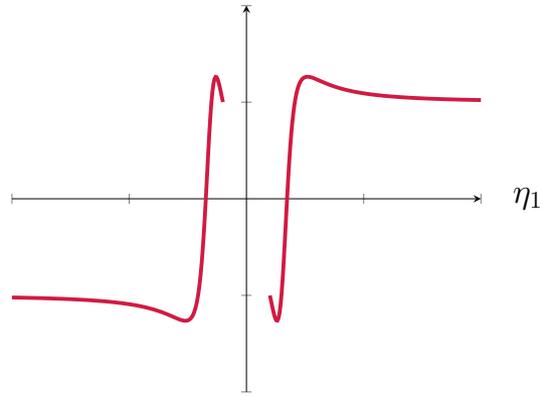
\begin{figure}
\centering
\begin{tikzpicture}[scale=0.9]
\begin{axis}[xmin=-10, xmax=10, ymin=-1., ymax=1., samples=200, no markers,
      axis x line = middle,
      axis y line = middle,
    xlabel=\Large{$\eta_{1}$},
    every axis x label/.style={
    at={(ticklabel* cs:1.1)},
},
every axis y label/.style={
    at={(ticklabel* cs:1.05)},
    anchor=south,
},
yticklabels={,,},
xticklabels={,,},
]
  \addplot[alizarin, ultra thick, domain=1.001:10]  {-x/(1+x*x) + sqrt((x*x+3)/(x*x-1))/(1+(x*x+3)/(x*x-1)) + (x- sqrt((x*x+3)/(x*x-1)))/(1+( x-sqrt((x*x+3)/(x*x-1)))*(x- sqrt((x*x+3)/(x*x-1))))};
    \addplot[alizarin, ultra thick, domain=-10:-1.001]  {-x/(1+x*x) - sqrt((x*x+3)/(x*x-1))/(1+(x*x+3)/(x*x-1)) + (x+sqrt((x*x+3)/(x*x-1)))/(1+( x+sqrt((x*x+3)/(x*x-1)))*(x+ sqrt((x*x+3)/(x*x-1))))};
  
\end{axis}
\end{tikzpicture}
\caption{Plot of $-\omega(\eta_{1}) + \omega(r\left(\eta_{1}\right))-\omega(-\eta_{1}+r\left(\eta_{1}\right))$ on $\left\{|\eta_{1}|>1\right\}.$ There are two roots at $|\eta_{1}|=\sqrt{3}$.} 
\label{fig:vanishing_phase_3}
\end{figure}

The fourth equation requires us to solve 
$$
-\eta_{1} - r\left(\eta_{1}\right) = \xi.
$$
From figure \ref{fig:tilde_system_4}, this equation has between zero and two solutions. The phase vanishes at this space resonance if and only if 
$$
-\omega(\eta_{1}) -\omega(r\left(\eta_{1}\right))+\omega(\eta_{1}+r\left(\eta_{1}\right)) =0
$$
From figure \ref{fig:vanishing_phase_4}, there are no solutions to the above equation.

\begin{figure}
\centering
\begin{tikzpicture}[scale=0.9]
\begin{axis}[xmin=-10, xmax=10, ymin=-40., ymax=40., samples=200, no markers,
      axis x line = middle,
      axis y line = middle,
    xlabel=\Large{$\eta_{1}$},
    every axis x label/.style={
    at={(ticklabel* cs:1.1)},
},
every axis y label/.style={
    at={(ticklabel* cs:1.05)},
    anchor=south,
},
yticklabels={,,},
xticklabels={,,},
]
  \addplot[ballblue, ultra thick, domain=1.001:10]  {-x-sqrt((x*x + 3)/(x*x - 1))};
    \addplot[ballblue, ultra thick, domain=-10:-1.001]  {-x+sqrt((x*x + 3)/(x*x - 1))};
  
\end{axis}
\end{tikzpicture}
\caption{Plot of $-\eta_{1} - r\left(\eta_{1}\right)$. There are vertical asymptotes at $|\eta_{1}|=1$.} 
\label{fig:tilde_system_4}
\end{figure}

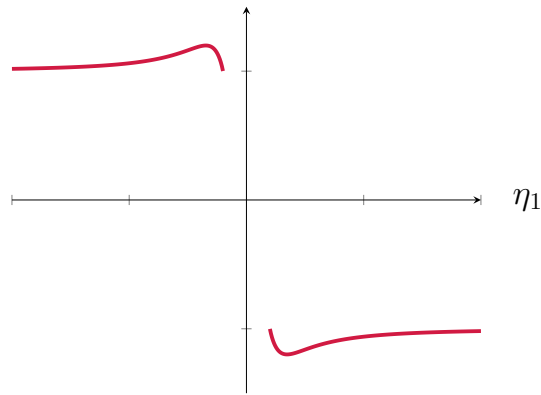
\begin{figure}
\centering
\begin{tikzpicture}[scale=0.9]
\begin{axis}[xmin=-10, xmax=10, ymin=-.75, ymax=.75, samples=200, no markers,
      axis x line = middle,
      axis y line = middle,
    xlabel=\Large{$\eta_{1}$},
    every axis x label/.style={
    at={(ticklabel* cs:1.1)},
},
every axis y label/.style={
    at={(ticklabel* cs:1.05)},
    anchor=south,
},
yticklabels={,,},
xticklabels={,,},
]
  \addplot[alizarin, ultra thick, domain=1.001:10]  {-x/(1+x*x) - sqrt((x*x+3)/(x*x-1))/(1+(x*x+3)/(x*x-1)) + (x+sqrt((x*x+3)/(x*x-1)))/(1+( x+sqrt((x*x+3)/(x*x-1)))*(x+ sqrt((x*x+3)/(x*x-1))))};
    \addplot[alizarin, ultra thick, domain=-10:-1.001]  {-x/(1+x*x) + sqrt((x*x+3)/(x*x-1))/(1+(x*x+3)/(x*x-1)) + (x-sqrt((x*x+3)/(x*x-1)))/(1+( x-sqrt((x*x+3)/(x*x-1)))*(x- sqrt((x*x+3)/(x*x-1))))};
  
\end{axis}
\end{tikzpicture}
\caption{Plot of $-\omega(\eta_{1}) -\omega(r\left(\eta_{1}\right))+\omega(\eta_{1}+r\left(\eta_{1}\right))$ on $\left\{|\eta_{1}|>1\right\}.$ There are no roots.} 
\label{fig:vanishing_phase_4}
\end{figure}

\item 
The cases \eqref{eqn:fam_1_c}, \eqref{eqn:fam_1_d} can be handled exactly like the case \eqref{eqn:fam_1_b}. 

\end{itemize}
 \ \par \noindent \textbf{Family 2}
 \par \noindent
 In this case we must have 
 \begin{align}
 |\eta_{1}|=|\eta_{2}|=|r\left(\eta_3\right)| =\begin{cases} 
  |\xi-\eta_{1}-\eta_{2}-\eta_{3}| 
  \\
   \quad\quad\quad \text{or}
  \\
\left|r\left(\xi-\eta_{1}-\eta_{2}-\eta_{3}\right)\right|
 \end{cases}
 \label{eqn:fam_2_main_eqn}
 \end{align}
 We can break this expression into four subfamilies following \eqref{eqn:fam_1_subfam}:
  \begin{subequations}
 \label{eqn:fam_2_subfam}
 \begin{align}
 \eta_{1}&=\phantom{-}\eta_{2}=\phantom{-}r\left(\eta_3\right),\label{eqn:fam_2_a}
 \\
 -\eta_{1} &=\phantom{-} \eta_{2} = \phantom{-}r\left(\eta_3\right),\label{eqn:fam_2_b}
 \\
 \eta_{1} &= -\eta_{2} = \phantom{-}r\left(\eta_3\right),\label{eqn:fam_2_c}
 \\
 \eta_{1} &= \phantom{-}\eta_{2} =- r\left(\eta_3\right),\label{eqn:fam_2_d}
 \end{align}
 \end{subequations}
 \begin{itemize}
 \item For \eqref{eqn:fam_2_a}, we have $\xi-\eta_{1}-\eta_{2}-\eta_{3} =\xi-2\eta_{1}-r\left(\eta_{1}\right)$. Then, we need to solve
\begin{align*}
\eta_{1}  = \begin{cases} 
\xi-2\eta_{1}-r\left(\eta_{1}\right)
  \\ 
-\xi+2\eta_{1}+r\left(\eta_{1}\right)
 \\
 r\left(\xi-2\eta_{1}-r\left(\eta_{1}\right)\right)
 \\
- r\left(\xi-2\eta_{1}-r\left(\eta_{1}\right)\right). 
 \end{cases}
\end{align*}
These equations reduce to 
\begin{align*}
\xi = \begin{cases}
3\eta_{1}+r\left(\eta_{1}\right)
\\ 
\eta_{1} +r\left(\eta_{1}\right)
\\ 
2(\eta_{1}+r\left(\eta_{1}\right))
\\
2\eta_{1}. 
\end{cases}
\end{align*}
The fourth equation immediately gives us the stationary point $(\eta_{1},\eta_{2},\eta_{3})= \left(\xi/2,\xi/2, r\left(\xi/2\right) \right)$. Note that this only makes sense when $\xi>2$. A quick computation shows that the phase cannot vanish in this case. 
\par For the other three equations, we have solutions of the form $\left(\eta_{1},\eta_{1},r\left(\eta_{1}\right)\right)$ where $\eta_{1}$ solves one of 
\begin{align}
3\eta_{1}+r\left(\eta_{1}\right)&=\xi \label{eqn:fam_2a_rough_1}
\\
\eta_{1}+r\left(\eta_{1}\right)&=\xi \label{eqn:fam_2a_rough_2}
\\
\eta_{1}+r\left(\eta_{1}\right)&=\frac{\xi}{2} \label{eqn:fam_2a_rough_3}.
\end{align}
To check if the phase vanishes at any of these stationary points, notice that the cases \eqref{eqn:fam_2a_rough_1} and \eqref{eqn:fam_2a_rough_2} can be handled using figure \ref{fig:vanishing_phase_1} and figure \ref{fig:vanishing_phase_2}: the phase does not vanish at any such points. For \eqref{eqn:fam_2a_rough_3}, it is enough to find if there are roots to 
$$
2\omega(\eta_{1}) +2\omega(r\left(\eta_{1}\right))-\omega(2\eta_{1}+2r\left(\eta_{1}\right))=0.
$$
Using figure \ref{fig:vanishing_phase_5}, we find this is impossible.

\begin{figure}
\centering
\begin{tikzpicture}[scale=0.9]
\begin{axis}[xmin=-10, xmax=10, ymin=-2., ymax=2., samples=200, no markers,
      axis x line = middle,
      axis y line = middle,
    xlabel=\Large{$\eta_{1}$},
    every axis x label/.style={
    at={(ticklabel* cs:1.1)},
},
every axis y label/.style={
    at={(ticklabel* cs:1.05)},
    anchor=south,
},
yticklabels={,,},
xticklabels={,,},
]
  \addplot[alizarin, ultra thick, domain=1.001:10]  {2*x/(1+x*x) +2* sqrt((x*x+3)/(x*x-1))/(1+(x*x+3)/(x*x-1)) - (2*x+2*sqrt((x*x+3)/(x*x-1)))/(1+( 2*x+2*sqrt((x*x+3)/(x*x-1)))*(2*x+ 2*sqrt((x*x+3)/(x*x-1))))};
    \addplot[alizarin, ultra thick, domain=-10:-1.001]  {2*x/(1+x*x) -2* sqrt((x*x+3)/(x*x-1))/(1+(x*x+3)/(x*x-1)) - (2*x-2*sqrt((x*x+3)/(x*x-1)))/(1+( 2*x-2*sqrt((x*x+3)/(x*x-1)))*(2*x-2*sqrt((x*x+3)/(x*x-1))))};
  
\end{axis}
\end{tikzpicture}
\caption{Plot of $2\omega(\eta_{1}) +2\omega(r\left(\eta_{1}\right))-\omega(2\eta_{1}+2r\left(\eta_{1}\right))$ on $\left\{|\eta_{1}|>1\right\}.$ There are no roots. } 
\label{fig:vanishing_phase_5}
\end{figure}
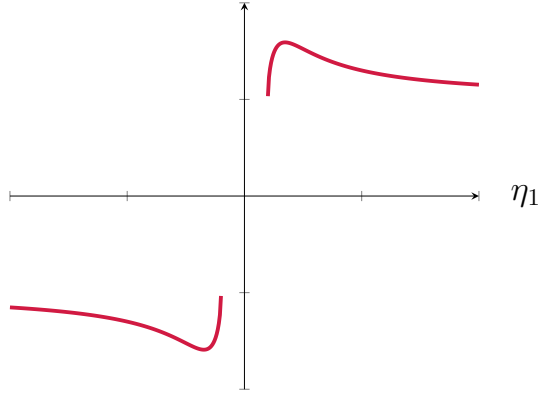

\item   For \eqref{eqn:fam_2_b}, we have $\xi-\eta_{1}-\eta_{2}-\eta_{3} =\xi+r\left(\eta_{1}\right)$. Then, we need to solve
\begin{align*}
\eta_{1}  = \begin{cases} 
\xi+r\left(\eta_{1}\right)
  \\
-\xi-r\left(\eta_{1}\right)
 \\
 r\left(\xi+r\left(\eta_{1}\right)\right)
 \\
- r\left(\xi+r\left(\eta_{1}\right)\right). 
 \end{cases}
\end{align*}
These equations reduce to
 \begin{align*}
\xi  = \begin{cases} 
\eta_{1} - r\left(\eta_{1}\right)
  \\ 
-\eta_{1}-r\left(\eta_{1}\right)
 \\
0
 \\
 -2r\left(\eta_{1}\right).
 \end{cases}
\end{align*}
The fourth equation gives us $(\eta_{1},\eta_{2},\eta_{3}) = \left(-r\left(\xi/2\right), r\left(\xi/2\right), \xi/2\right)$. The phase does not vanish at the point since $|\xi|>2$. Also, the third equation gives rise to a space-time resonance on a transformed version of the curve $\Gamma$ given in \eqref{eqn:resonant_curve_discussion}. The first two equations have between two and four solutions, by figure \ref{fig:tilde_system_4}. Thus we may use figures \ref{fig:vanishing_phase_3} and \ref{fig:vanishing_phase_4} and the relevant discussion from previous cases to find that the only STR contributed by these equations is $\left(-\sqrt{3}, \sqrt{3}, \sqrt{3};0\right)$ (up to permuting the negative sign among the $\eta_j$'s). 

\item   For \eqref{eqn:fam_2_c}, we argue exactly as in the previous item. 

\item  For \eqref{eqn:fam_2_d}, we have $\xi-\eta_{1}-\eta_{2}-\eta_{3} =\xi-2\eta_{1}+r\left(\eta_{1}\right)$. Then, we need to solve
\begin{align*}
\eta_{1}  = \begin{cases} 
\xi-2\eta_{1}+r\left(\eta_{1}\right)
  \\ 
-\xi+2\eta_{1}-r\left(\eta_{1}\right)
 \\
 r\left(\xi-2\eta_{1}+r\left(\eta_{1}\right)\right)
 \\
- r\left(\xi-2\eta_{1}+r\left(\eta_{1}\right)\right). 
 \end{cases}
\end{align*}
These reduce to 
\begin{align*}
\xi  = \begin{cases} 
3\eta_{1} - r\left(\eta_{1}\right)
  \\ 
\eta_{1}- r\left(\eta_{1}\right)
 \\
2\eta_{1}
 \\
2\eta_{1}-2r\left(\eta_{1}\right). 
 \end{cases}
\end{align*}
The third equation is easily solved to find $(\eta_{1},\eta_{2},\eta_{3})= \left(\xi/2,\xi/2, -r\left(\xi/2\right) \right)$. The phase cannot vanish here since $|\xi|>2$
\par 
From our earlier work, we find there are 6 other stationary points of the form $(\eta_{1},\eta_{1},-r\left(\eta_{1}\right))$ with $\eta_{1}$ solving one of the equations
\begin{subequations}
\begin{align}
3\eta_{1}-r\left(\eta_{1}\right) &= \xi \label{eqn:fam_3d_rough_1}
\\
\eta_{1}-r\left(\eta_{1}\right) &= \xi \label{eqn:fam_3d_rough_2}
\\
\eta_{1}-r\left(\eta_{1}\right) &=\frac{\xi}{2} \label{eqn:fam_3d_rough_3}.
\end{align}
\end{subequations}
The phase vanishes in the case of \eqref{eqn:fam_3d_rough_1} if and only if 
$$
3\omega(\eta_{1})-\omega\left(r\left(\eta_{1}\right)\right)-\omega\left(3\eta_{1}-r\left(\eta_{1}\right)\right) =0. 
$$
From figure \ref{fig:vanishing_phase_2}, this has two solutions $|\eta_{1}|\approx 5.1$, corresponding to $|\xi|\approx 14.2$. 
\par 
The phase vanishes in the case of \eqref{eqn:fam_3d_rough_2} if and only if 
$$
\omega(\eta_{1})-\omega\left(r\left(\eta_{1}\right)\right)-\omega\left(\eta_{1}-r\left(\eta_{1}\right)\right) =0. 
$$
From figure \ref{fig:vanishing_phase_3}, this gives us an STR we've already counted. 
\par 
The phase vanishes in the case of \eqref{eqn:fam_3d_rough_3} if and only if 
$$
2\omega(\eta_{1})-2\omega\left(r\left(\eta_{1}\right)\right)-\omega\left(2\eta_{1}-2r\left(\eta_{1}\right)\right) =0. 
$$
From figure \ref{fig:vanishing_phase_6}, this equation has roots if and only if $|\eta_{1}|=\sqrt{3}$ and $\xi=0$, which is again a situation we've already counted. 

\begin{figure}
\centering
\begin{tikzpicture}[scale=0.9]
\begin{axis}[xmin=-10, xmax=10, ymin=-1.1, ymax=1.1, samples=200, no markers,
      axis x line = middle,
      axis y line = middle,
    xlabel=\Large{$\eta_{1}$},
    every axis x label/.style={
    at={(ticklabel* cs:1.1)},
},
every axis y label/.style={
    at={(ticklabel* cs:1.05)},
    anchor=south,
},
yticklabels={,,},
xticklabels={,,},
]
  \addplot[alizarin, ultra thick, domain=1.001:10]  {2*x/(1+x*x) -2* sqrt((x*x+3)/(x*x-1))/(1+(x*x+3)/(x*x-1)) - (2*x-2*sqrt((x*x+3)/(x*x-1)))/(1+( 2*x-2*sqrt((x*x+3)/(x*x-1)))*(2*x-2*sqrt((x*x+3)/(x*x-1))))};
    \addplot[alizarin, ultra thick, domain=-10:-1.001]  {2*x/(1+x*x) +2* sqrt((x*x+3)/(x*x-1))/(1+(x*x+3)/(x*x-1)) - (2*x+2*sqrt((x*x+3)/(x*x-1)))/(1+( 2*x+2*sqrt((x*x+3)/(x*x-1)))*(2*x+2*sqrt((x*x+3)/(x*x-1))))};
  
\end{axis}
\end{tikzpicture}
\caption{Plot of $2\omega(\eta_{1})-2\omega\left(r\left(\eta_{1}\right)\right)-\omega\left(2\eta_{1}-2r\left(\eta_{1}\right)\right)$ on $\left\{|\eta_{1}|>1\right\}.$ There are two roots at $|\eta_{1}|=\sqrt{3}$. } 
\label{fig:vanishing_phase_6}
\end{figure}
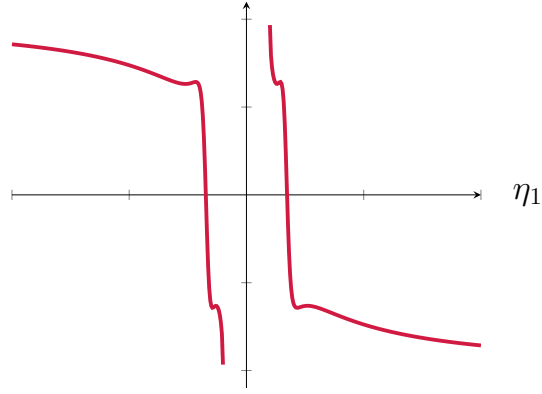
 \end{itemize}
 
 \par \noindent \textbf{Families 3 and 4}
 \par \noindent
By symmetry, we can find all of the resonant points from Family 3 and Family 4 by permuting the $-1$'s and $r$'s in the resonant points from Family 2: for Family 3, we switch $\eta_{2}$ and $\eta_{3}$ then apply the Family 2 results, and for Family 4 we switch $\eta_{1}$ and $\eta_{3}$ and apply the Family 2 results. 

\section{Multilinear Estimates}

\noindent The key $L^2_{\xi}$-estimate we'll need for multilinear operators is as follows: 
\begin{prop}\label{prop:multilinear_estimates}
Suppose $m(\eta_1, \eta_2, \eta_3,\xi)$ satisfies
\begin{equation}
\label{eqn:mult_bound}
\left\|\int \diff\eta_1\diff\eta_2\diff\eta_3\diff\xi \ e^{i\left(\eta_1,\eta_2,\eta_3,\xi\right)\cdot\left(x_1,x_2,x_3,y\right)} m(\eta_1, \eta_2, \eta_3,\xi)\right\|_{L^1_{x_{1},x_{2},x_{3},y}} \leq A. 
\end{equation}
For any $p_{1},...,p_{5}\in [1,\infty]$ satisfying $\sum_{k=1}^{5} \frac{1}{p_{k}} = 1$ and all sufficiently nice $v_{1}(x), ... , v_{5}(x)$, we have
\begin{equation}\label{eqn:me_all_int}
\hspace{-1cm}
   \left| \int \diff\eta_1\diff\eta_2\diff\eta_3\diff\xi \  \widehat{v_{1}}\left(\eta_{1}\right) \  \widehat{v_{2}}\left(\eta_{2}\right)\  \widehat{v_{3}}\left(\eta_{3}\right)\  \widehat{v_{4}}\left(\xi\right)   \  \widehat{v_{5}}\left(-\xi-\sum_{j=1}^{3}\eta_{j}\right) m(\eta_1, \eta_2, \eta_3,\xi) \right| \lesssim A \prod_{k=1}^{5} \left\|v_{k}\right\|_{L^{p_{k}}_{x}}. 
   \end{equation}
Additionally, for any $q_{1},...,q_{4}\in [2,\infty]$ satisfying $\sum_{k=1}^{4} \frac{1}{q_{k}} = \frac12$, we have the $L^2$ estimate 
\begin{equation}\label{eqn:me_L2}
\left\| \int \diff\eta_1\diff\eta_2\diff\eta_3 \ \widehat{u_{1}}\left(\eta_{1}\right) \  \widehat{u_{2}}\left(\eta_{2}\right)\  \widehat{u_{3}}\left(\eta_{3}\right)\  \widehat{u_{4}}\left(\xi-\sum_{j=1}^{3}\eta_{j}\right) \ m(\eta_1, \eta_2, \eta_3,\xi)\right\|_{L^2_{\xi}} \lesssim A \prod_{k=1}^{4} \left\|u_{k}\right\|_{L^{q_{k}}_{x}}. 
    \end{equation}
\end{prop}
\begin{proof} This follows from the arguments used to prove similar results in \cite{IP2014} and \cite{Pusateri2013}.
\end{proof}
\noindent Proposition \ref{prop:multilinear_estimates} is easy to apply \emph{after} we have shown that $\overset{\vee}{m}\in L^{1}_{x_{1}x_{2}x_{3}y}$. Verifying this condition directly may be time-consuming in practice, but there are shortcuts one can sometimes take. For instance, if we've already shown two symbols $m_1, m_2$ satisfy the required integrability condition, then Young's convolution inequality implies their product satisfies this condition as well:
\begin{lemma}\label{lemma:symbol_algebra_property}
    Suppose $m_{1}, m_2\colon \mathbb{R}^{d}\rightarrow\mathbb{C}$ satisfy $\left\|\overset{\vee}{m_{j}}\right\|_{L^1_{x}} \lesssim A_j$ for $j=1,2.$ Then, $\left\|\left(m_{1}m_2\right)^{\vee}\right\|_{L^1_{x}} \lesssim A_1 A_2$. 
\end{lemma}
\qed
\par \noindent The most versatile type of shortcut is simply verifying a pointwise Mikhlin-H\"{o}rmander-type bound on the derivatives of $m$. Such a shortcut requires $m$ to be spectrally localized, but since all our oscillatory integral estimates take place within an LP framework this is perfectly acceptable. 
\begin{lemma}\label{lemma:multiplier_cheat_code}
    Suppose $k_{1}, ... , k_{d}\in \mathbb{Z}$ satisfy $2^{k_{j}}<1$ for all $j$. Let $\psi_{k_{j}}(\xi_j)\colon \mathbb{R}\rightarrow [0,1]$ be the annular bump functions from definition \ref{defn:bumps}. Additionally, suppose $a(\xi)\colon \mathbb{R}^{d}\rightarrow \mathbb{C}$ is smooth on the support of $\prod_{j=1}^{d}\psi_{k_{j}}(\xi_j)$ and satisfies, for some $p>d$,
    \begin{equation}\label{eqn:cheat_code_mikhlin_condition}
    \left|\left(\partial_{\xi_{\ell}}^{q}a\right) \ \prod_{j=1}^{d}\psi_{k_{j}}(\xi_j)\right| \lesssim 2^{-k_{\ell}q}A \quad \forall \ q=0,1,...,p, \quad \forall \ \ell=1,...,d,
    \end{equation}
with a bounding constant independent of $A$. Then, the multiplier $m(\xi) \doteq a(\xi) \ \prod_{j=1}^{d}\psi_{k_{j}}(\xi_j)$
satisfies
\begin{equation}
    \label{eqn:multiplier_cheat_code}
    \left\|\overset{\vee}{m}\right\|_{L^{1}_{x}} \lesssim A 
\end{equation}
with a bounding constant independent of $A$ and the $k_j$'s. 
\end{lemma}
\begin{proof}
    The result follows from a direct calculation using fractional integration by parts. 
\end{proof}
\noindent We'll also need a ``secondarily-localized'' version of lemma \ref{lemma:multiplier_cheat_code}: 
\begin{lemma}\label{lemma:multiplier_cheat_code_secondary_localization}
    Suppose $k_{1}, ... , k_{d}\in \mathbb{Z}$ and $\ell_{1}, ... , \ell_{d}\in \mathbb{Z}$ satisfy
    \begin{itemize}
        \item $2^{k_{j}}, 2^{\ell_{j}}<1$ for all $j$, 
        \item $2^{k_{j}}\approx 2^{k_{j'}}$ for all $j, j'$,
        \item $2^{\ell_{j}}\lesssim 2^{k_j}$ for all j. 
    \end{itemize} 
    Let $\psi_{k_j}, \psi_{\ell_j}$ be defined as in definition \ref{defn:bumps}. Let $B\colon \mathbb{R}^d\rightarrow\mathbb{R}^d$ be an invertible linear transformation, and let $\mu = \mu(\xi) \doteq B\xi \in \mathbb{R}^d$. Additionally, suppose $a(\xi)\colon \mathbb{R}^{d}\rightarrow \mathbb{C}$ is smooth on the support of $\prod_{j=1}^{d}\psi_{k_{j}}(\xi_j)$ and satisfies, for some $p>d$,
    \begin{equation}\label{eqn:cheat_code_mikhlin_condition_sec_loc}
    \left|\left(\partial_{\mu_{n}}^{q}a\right) \ \prod_{j=1}^{d}\psi_{k_{j}}(\xi_j)\right| \lesssim 2^{-\ell_{n}q}A \quad \forall \ q=0,1,...,p, \quad \forall \ n=1,...,d,
    \end{equation}
with a bounding constant independent of $A$. Then, the multiplier $m(\xi) \doteq a(\xi) \ \prod_{j=1}^{d}\psi_{k_{j}}(\xi_j)\prod_{j=1}^{d}\psi_{\ell_{j}}(\mu_j)$ satisfies
\begin{equation}
    \label{eqn:multiplier_cheat_code_secondary_localizations}
    \left\|\overset{\vee}{m}\right\|_{L^{1}_{x}} \lesssim A 
\end{equation}
with a bounding constant independent of $A$, the $k_j$'s, and the $\ell_j$'s. 
\end{lemma}

\begin{proof}
Apply the same ideas from lemma \ref{lemma:multiplier_cheat_code} after performing a suitable substitution.
\end{proof}
\noindent To see why the hypotheses on the $k_j$'s and $\ell_j$'s are the way they are, it's helpful to think about the converse case $\min_{j}\ell_j \gg \max_{j}k_j$. In this situation, differentiation of $m(\xi)$ with respect to the $\xi_j$'s would still be OK: when such derivatives hit $\psi_{\ell_j}\left(\mu_j\right)$ they give rise to a $2^{\ell_j}\ll 2^{-k_{j}}$, putting us in basically the same situation as lemma \ref{lemma:multiplier_cheat_code}. 
\begin{remark}\label{remark:nonlinear_loc}
In practice, one may wish to perform a secondary localization with respect to a new coordinate system $\widetilde{\xi}(\xi)$ depending \emph{nonlinearly} on the primitive coordinate system $\xi$. The proof of the above lemma cannot be modified to successfully accommodate this added complexity. Instead, one must use the following rough recipe to justify the multilinear estimate: 
\begin{enumerate}
    \item localize the integrand of the multilinear operator with respect to $\widetilde{\xi}$;
    \item perform whatever integration by parts is allowed by your localization;
    \item flatten the coordinate system by changing variables $\xi\mapsto \widetilde{\xi}$ inside the integral so we now integrate with respect to $\diff \widetilde{\xi}$, and absorb the Jacobian determinant into the symbol $m(\xi)$;
    \item check that the conditions of lemma \ref{lemma:multiplier_cheat_code_secondary_localization} hold in the flattened coordinates. 
\end{enumerate} \ \par\noindent 
\end{remark}
\par Finally, we'll need an analogue of proposition \ref{prop:multilinear_estimates} valid for the $L^{\infty}_{\xi}$ norm: 
\begin{prop}
    \label{prop:multilinear_estimates_Linfty}
Suppose $m(\eta_1, \eta_2, \eta_3)$ satisfies
\begin{equation}
\label{eqn:mult_bound_Linfty}
\left\|\int \diff\eta_1\diff\eta_2\diff\eta_3 \ e^{i\left(\eta_1,\eta_2,\eta_3\right)\cdot\left(x_1,x_2,x_3\right)} m(\eta_1, \eta_2, \eta_3)\right\|_{L^1_{x_{1},x_{2},x_{3}}} \leq A. 
\end{equation}
For any $p_{1},...,p_{4}\in [1,\infty]$ satisfying $\sum_{k=1}^{4} \frac{1}{p_{k}} = 1$ and all sufficiently nice $v_{1}(x), ... , v_{4}(x)$, we have
\begin{equation}\label{eqn:me_all_int_Linfty}
\hspace{-1cm}
   \left| \int \diff\eta_1\diff\eta_2\diff\eta_3 \  \widehat{v_{1}}\left(\eta_{1}\right) \  \widehat{v_{2}}\left(\eta_{2}\right)\  \widehat{v_{3}}\left(\eta_{3}\right)\     \  \widehat{v_{4}}\left(-\sum_{j=1}^{3}\eta_{j}\right) m(\eta_1, \eta_2, \eta_3) \right| \lesssim A \prod_{k=1}^{4} \left\|v_{k}\right\|_{L^{p_{k}}_{x}}. 
   \end{equation}
In particular, if $m(\eta_1,\eta_2,\eta_3) = M(\eta_1,\eta_2,\eta_3, \xi)$ then 
\begin{equation}\label{eqn:me_all_int_Linfty_true}
\hspace{-1cm}
   \left\| \int \diff\eta_1\diff\eta_2\diff\eta_3 \  \widehat{v_{1}}\left(\eta_{1}\right) \  \widehat{v_{2}}\left(\eta_{2}\right)\  \widehat{v_{3}}\left(\eta_{3}\right)\     \  \widehat{v_{4}}\left(\xi-\sum_{j=1}^{3}\eta_{j}\right) M(\eta_1, \eta_2, \eta_3, \xi) \right\|_{L^{\infty}_{\xi} }\lesssim A \prod_{k=1}^{4} \left\|v_{k}\right\|_{L^{p_{k}}_{x}}. 
   \end{equation}
\end{prop}
\qed
\end{appendices}

\bibliography{BBM.bib} 
\bibliographystyle{plain}

\end{document}